\documentclass[reqno,11pt]{article}
\usepackage{a4wide,xcolor,eucal,enumerate,mathrsfs,hyperref}
\usepackage{amsmath,amssymb,epsfig,bbm}
\usepackage[latin1]{inputenc}

\newcommand{\C}{\mathbb{C}}

\newcommand{\N}{\mathbb{N}}
\newcommand{\Q}{\mathbb{Q}}
\newcommand{\R}{\mathbb{R}}

\newcommand{\BB}{\mathscr{B}}

\newcommand{\mm}{{\mbox{\boldmath$m$}}}
\newcommand{\nn}{{\mbox{\boldmath$n$}}}

\newcommand{\rr}{{\mbox{\boldmath$r$}}}

\newcommand{\srr}{{\mbox{\scriptsize\boldmath$r$}}}

\newcommand{\ggamma}{{\mbox{\boldmath$\gamma$}}}

\newcommand{\ppi}{{\mbox{\boldmath$\pi$}}}

\newcommand{\sggamma}{{\mbox{\scriptsize\boldmath$\gamma$}}}

\newcommand{\sfd}{{\sf d}}

\newcommand{\sft}{{\sf t}}

\newcommand{\sfx}{{\sf x}}

\newcommand{\rme}{{\mathrm e}}

\newcommand{\rmD}{{\mathrm D}}

\newcommand{\Kliminf}{K\kern-3pt-\kern-2pt\mathop{\rm lim\,inf}\limits}  
\newcommand{\diam}{\mathop{\rm diam}\nolimits}   
\newcommand{\argmin}{\mathop{\rm argmin}\limits}   
\newcommand{\Lip}{\mathop{\rm Lip}\nolimits}          
\renewcommand{\d}{{\mathrm d}}
\newcommand{\dt}{{\d t}}

\newcommand{\restr}[1]{\lower3pt\hbox{$|_{#1}$}}

\newcommand{\Leb}[1]{{\mathscr L}^{#1}}      

\newcommand{\down}{\downarrow}              
\newcommand{\up}{\uparrow}
\newcommand{\eps}{\varepsilon}  
\newcommand{\nchi}{{\raise.3ex\hbox{$\chi$}}}
\newcommand{\weakto}{\rightharpoonup}

\newcommand{\media}{\mkern12mu\hbox{\vrule height4pt           %
          depth-3.2pt                                 
          width5pt}\mkern-16.5mu\int\nolimits}        

\newcommand{\forevery}{\text{for every }}

\newcommand{\Pc}[2]{\overline{#1}\kern-2pt^{\vphantom 0}_{#2}}

\newcommand{\RE}{\RelativeEntropy}

\newcommand{\REG}[1]{\mathrm{Ent}_\gamma(#1)}

\newcommand{\BorelSets}[1]{\BB(#1)}

\newcommand{\Probabilities}[1]{\mathscr P(#1)}          
\newcommand{\ProbabilitiesTwo}[1]{\mathscr P_2(#1)}     

\numberwithin{equation}{section}

\newenvironment{proof}{\removelastskip\par\medskip   
\noindent{\em Proof.}\rm}{\penalty-20\null\hfill$\square$\par\medbreak}

\newtheorem{theorem}{Theorem}[section]

\newtheorem{corollary}[theorem]{Corollary}
\newtheorem{lemma}[theorem]{Lemma}
\newtheorem{proposition}[theorem]{Proposition}

\newtheorem{definition}[theorem]{Definition}

\newtheorem{example}[theorem]{Example}
\newtheorem{remark}[theorem]{Remark}

\renewcommand{\mm}{\mathfrak m}
\renewcommand{\nn}{\mathfrak n}

\newcommand{\ent}[1]{{\rm Ent}_{\mm}(#1)}              
\newcommand{\entv}{{\rm Ent}_{\mm}}                    
\newcommand{\prob}{\Probabilities}
\newcommand{\probt}{\ProbabilitiesTwo}
\newcommand{\pro}[1]{\mathscr P_{[#1]}}                   
\newcommand{\proV}{\mathscr P_\Wgh}
\newcommand{\geo}{{\rm{Geo}}}                       
\newcommand{\e}{{\rm{e}}}                           
\newcommand{\fr}{\hfill$\blacksquare$}                      
\newcommand{\BorelSetsStar}[1]{{\mathscr B}^*(#1)}

\newcommand{\relgrad}[1]{|\rmD  #1|_*}                     
\newcommand{\weakgradA}[1]{|\rmD  #1|_{w,\calT}}
\newcommand{\weakgradG}[1]{|\rmD  #1|_{w,\mathcal G}}

\renewcommand{\C}{{\sf Ch}_*}

\usepackage[normalem]{ulem} 
\newcommand{\GGG}{\normalcolor}
\newcommand{\nc}{\normalcolor}

\newcommand{\class}[2]{#1_{[#2]}}
\newcommand{\dto}{\stackrel \sfd\to}
\newcommand{\Deltam}{\Delta_{\sfd,\mm}}
\newcommand{\Wgh}{V}
\newcommand{\sfH}{\mathsf H}
\newcommand{\Heat}[1]{\sfH_{#1}}
\newcommand{\PushPlan}[2]{#1_{\sharp}#2}
\renewcommand{\RE}[2]{\mathrm{Ent}_{#2}(#1)}
\renewcommand{\REG}{\ent}
\newcommand{\AC}[4]{\mathrm{AC}^{#1}(#2;(#3,#4))}
\newcommand{\CC}[3]{C(#1;#2)}

\newcommand{\Energy}[2]{\mathcal E_{#1}[#2]}
\newcommand{\bJ}{[0,1]}
\newcommand{\calT}{{\mathcal T}}
\newcommand{\tmm}{\tilde\mm}

\title{Calculus and heat flow in metric measure spaces and
applications to spaces with Ricci bounds from below}
\begin{document}

\author{Luigi Ambrosio\
   \thanks{\textsf{l.ambrosio@sns.it}}
   \and
   Nicola Gigli\
   \thanks{\textsf{nicola.gigli@unice.fr}}
 \and
   Giuseppe Savar\'e\
   \thanks{\textsf{giuseppe.savare@unipv.it}}
   }

\maketitle

\begin{abstract}
This paper is devoted to a deeper understanding of the heat flow and to the
refinement of calculus tools on metric
measure spaces $(X,\sfd,\mm)$.
Our main results are:
\begin{itemize} 
\item A general study of the relations between the Hopf-Lax semigroup and Hamilton-Jacobi equation in metric spaces $(X,\sfd)$.
\item The equivalence of the
heat flow in $L^2(X,\mm)$ 
generated by a suitable Dirichlet
energy and the Wasserstein gradient flow of the relative entropy
functional $\entv$ in the space of probability
measures $\Probabilities{X}$.
\item The proof of density in energy of Lipschitz functions in the Sobolev space $W^{1,2}(X,\sfd,\mm)$.
\item A fine and very general analysis of the differentiability properties
of a large class of Kantorovich potentials, in connection with the optimal transport
problem, is the fourth achievement of the paper. 
\end{itemize}
Our results apply in particular to
spaces satisfying Ricci curvature bounds in the sense of Lott \&
Villani \cite{Lott-Villani09} and Sturm \cite{Sturm06I,Sturm06II}
and require neither the doubling property nor the validity of the
local Poincar\'e inequality.

\textbf{MSC-classification:} 52C23, 49J52, 49Q20, 58J35, 35K90, 31C25
\end{abstract}

\tableofcontents

\section{Introduction}

Aim of this paper is to provide a deeper understanding of analysis in metric measure spaces, with a particular 
focus on the properties of the heat flow. Our main results, whose validity does not depend on doubling and Poincar\'e assumptions, are:
\begin{itemize}
\item[(i)] The proof that the Hopf-Lax formula produces sub-solutions of the Hamilton-Jacobi equation on general metric 
spaces $(X,\sfd)$ {\nc (Theorem~\ref{thm:subsol})}, and solutions if
$(X,\sfd)$ is a length space {\nc (Theorem~\ref{prop:supersol})} {\nc (in connection to this, under less general assumptions
on the metric structure, closely related results have been independently obtained in \cite{Gozlan}).}
\item[(ii)] The proof of equivalence of the heat flow in $L^2(X,\mm)$ generated by a suitable Dirichlet
energy and the Wasserstein gradient flow in $\Probabilities{X}$ of
the relative entropy functional $\entv$ w.r.t. $\mm$ {\GGG (Theorems
  \ref{thm:coincidence_infty}, \ref{thm:main} and \ref{thm:mainCDK})}.
\item[(iii)] The proof that Lipschitz functions are always dense in
  energy in the Sobolev space $W^{1,2}$ {\GGG (Theorem \ref{thm:density_energy})}. This is 
achieved by showing the equivalence of two weak notions of modulus of the
gradient: the first one (inspired by Cheeger \cite{Cheeger00}, see
also \cite{Heinonen-Koskela98}, \cite{Hailasz-Koskela00}, and the
recent review \cite{Heinonen07}), that we call \emph{relaxed
gradient}, is defined by $L^2(X,\mm)$-relaxation of the pointwise
Lipschitz constant in the class of Lipschitz functions; the second
one (inspired by Shanmugalingam \cite{Shanmugalingam00}), that we
call \emph{weak upper gradient}, is based on the validity of the
fundamental theorem of calculus along almost all curves. These two
notions of gradient will be compared and identified, assuming only $\mm$ to
be locally finite. We might
consider the former gradient as a ``vertical'' derivative, related
to variations in the dependent variable, while the latter is an
``horizontal'' derivative, related to variations in the independent
variable. 
\item[(iv)] A fine and very general analysis of the differentiability properties
of a large class of Kantorovich potentials, in connection with the optimal transport
problem {\GGG (Theorem \ref{thm:brweak})}.
\end{itemize}
Our results apply in particular to spaces satisfying Ricci curvature
bounds in the sense of Lott \& Villani \cite{Lott-Villani09} and
Sturm \cite{Sturm06I,Sturm06II}, that we call in this introduction
$LSV$ spaces. Indeed, the development of a ``calculus'' in this
class of spaces has been one of our motivations. In particular we
are able to prove the following result (see
Theorem~\ref{thm:mainCDK} for a more precise and general statement):
if $(X,\sfd,\mm)$ is a $CD(K,\infty)$ space and
$\mm\in\Probabilities{X}$, then
\begin{itemize}
\item[(a)]
  For every $\mu=f\mm\in\Probabilities{X}$ the Wasserstein slope
  $|\rmD^-\entv|^2(\mu)$ of the relative entropy $\entv$ coincides with
  the Fisher information functional
  $\int_{\{\rho>0\}}|\rmD \rho|_*^2/\rho\,\d\mm$,
  where $|\rmD  \rho|_*$ is the relaxed gradient of $\rho$ (see
  the brief discussion before \eqref{def:Cheeger1}).
\item[(b)]
  For every $\mu_0=f_0\mm\in D(\entv)\cap\ProbabilitiesTwo{X}$ there exists a unique gradient flow
  $\mu_t=f_t\mm$ of $\entv$ starting from $\mu_0$ in
  $(\ProbabilitiesTwo{X},W_2)$, and if $f_0\in L^2(X,\mm)$ the functions
  $f_t$ coincide with the $L^2(X,\mm)$ gradient flow of Cheeger's
  energy $\C$, defined by (see also \eqref{def:Cheeger1} for an equivalent definition)
  \begin{equation}
\C(f):=\frac12 \inf\left\{\liminf_{h\to\infty}\int |\rmD  f_h|^2\,\d\mm: f_h\in {\rm Lip}(X),\,\,
\int_X|f_h-f|^2\,\d\mm\to 0\right\}.\label{def:Cheeger0}
\end{equation}
\end{itemize}
On the other hand, we believe that the``calculus" results described
in (iii) are of a wider interest for analysis in metric measure
spaces, beyond the application to $LSV$ spaces. Particularly important
is not only the identification of heat flows, but also the
identification of weak gradients that was previously
known only under doubling and Poincar\'e assumptions. The key new idea is to use the
heat flow and the rate of energy dissipation, instead of the usual
covering arguments, to prove the optimal approximation by Lipschitz
functions, see also Remark~\ref{rem:compachsh} and Remark~\ref{rem:compachsh2} for
a detailed comparison with the previous approaches ({\nc see also \cite{AGS11c} for the
extensions of these ideas to the Sobolev spaces $W^{1,p}(X,\sfd,\mm)$, $1<p<\infty$, and 
\cite{Ambrosio-DiMarino} for the space of functions of bounded variation).}

In connection with (ii), notice that the equivalence so far has been
proved in Euclidean spaces by Jordan-Kinderleher-Otto, in the
seminal paper \cite{JordanKinderlehrerOtto98}, in Riemannian
manifolds by Erbar and Villani \cite{Erbar10,Villani09}, 
in Hilbert spaces by \cite{Ambrosio-Savare-Zambotti09}, 
in Finsler spaces by Ohta-Sturm
\cite{Ohta-Sturm09} and eventually in Alexandrov spaces by
Gigli-Kuwada-Ohta \cite{GigliKuwadaOhta10}. In fact, the strategy
pursued in \cite{GigliKuwadaOhta10}, that we shall describe later
on, had a great influence on our work. The distinguished case when
the gradient flows are linear will be the object, in connection with
$LSV$ spaces, of a detailed investigation in
\cite{Ambrosio-Gigli-Savare11bis}.

We exploit as much as possible the variational formulation of
gradient flows on one hand (based on sharp energy dissipation rate
and the notion of descending slope) and the variational structure of
the optimal transportation problem to develop a theory that does not
rely on finite dimensionality and doubling properties; we are even
able to cover situations where the distance $\sfd$ is allowed to
take the value $+\infty$, as it happens for instance in optimal
transportation problems in Wiener spaces (see for instance
\cite{Feyel-Ustunel04,Fang-Shao-Sturm10}). We are also able to deal
with $\sigma$-finite measures $\mm$, provided they are representable
in the form $\rme^{\Wgh^2}\,\tmm$ with $\tmm(X)\leq 1$ and
$\Wgh:X\to [0,\infty)$ $\sfd$-Lipschitz weight function bounded from
above on compact sets.

In order to reach this level of generality, it is useful to separate
the roles of the topology $\tau$ of $X$ (used for the
measure-theoretic structure) and of the possibly extended distance
$\sfd$ involved in the optimal transport problem, introducing the
concept of Polish extended metric measure space $(X,\sfd,\tau,\mm)$.
Of course, the case when $\sfd$ is a distance inducing the Polish
topology $\tau$ is included. Since we assume neither doubling
properties nor the validity of the Poincar\'e inequalities, we can't
rely on Cheeger's theory \cite{Cheeger00}, developed precisely under
these assumptions. The only known connection between synthetic
curvature bounds and this set of assumptions is given in
\cite{Lott-Villani-Poincare}, where the authors prove that in
non-branching $LSV$ spaces the Poincar\'e inequality holds under the
so-called $CD(K,N)$ assumption ($N<\infty$), a stronger curvature
assumption which involves also the dimension. {\nc In a more
recent paper \cite{R2011} a version of the Poincar\'e inequality valid in
all $CD(K,\infty)$ spaces is proved with no non-branching assumption; 
 this version implies the classical Poincar\'e inequality whenever
the measure $\mm$ is doubling.}

Now we pass to a more detailed description of the results of the
paper, the problems and the related literature. In
Section~\ref{sec:preliminary} we introduce all the basic concepts
used in the paper: first we define extended metric spaces
$(X,\sfd)$, Polish extended spaces $(X,\sfd,\tau)$ (in our
axiomatization $\sfd$ and $\tau$ are not completely decoupled, see
(iii) and (iv) in Definition~\ref{dPolish}), absolutely continuous
curves, metric derivative $|\dot{x}_t|$, local Lipschitz constant
$|\rmD  f|$, one-sided slopes $|\rmD^\pm f|$. Then, we see how in
Polish extended spaces one can naturally state the optimal transport
problem with cost $c=\sfd^2$ in terms of transport plans
(i.e. probability measures in $X\times X$. {\nc Only in
Section~\ref{sec:weakbrenier} we discuss the formulation of the optimal transport problem}
in terms of geodesic transport plans, namely probability
measures with prescribed marginals at $t=0$, $t=1$ in the space
$\geo(X)$ of constant speed geodesics in $X$. In
Subsection~\ref{se:prelGF} we recall the basic definition of
gradient flow $(y_t)$ of an energy functional $E$: it is based on
the integral formulation of
the sharp energy dissipation rate
$$
-\frac{\d}{\d t}E(y_t)\geq \frac{1}{2}|\dot
y_t|^2+\frac{1}{2}|\rmD^-E(y_t)|^2
$$
which, under suitable additional assumptions (for instance the fact
that $|\rmD^-E|$ is an upper gradient of $E$, as it happens for
$K$-geodesically convex functionals), turns into an equality for
almost every time. These facts will play a fundamental role in our
analysis.

In Section~\ref{sec:hopflax} we study the fine properties of the
Hopf-Lax semigroup
\begin{equation}\label{eq:Qtfintro}
Q_tf(x):=\inf_{y\in X}f(y)+\frac{\sfd^2(x,y)}{2t},\quad\qquad
(x,t)\in X\times (0,\infty)
\end{equation}
in a extended metric space $(X,\sfd)$. Here the main technical
novelty, with respect to \cite{Lott-Villani07bis}, is the fact that
we do not rely on Cheeger's theory (in fact, no reference measure
$\mm$ appears here) to show in Theorem~\ref{prop:supersol} that
in length spaces $(x,t)\mapsto Q_tf(x)$ is a pointwise solution to
the Hamilton-Jacobi equation $\partial_tQ_tf+|\rmD  Q_t|^2/2=0$:
precisely, for given $x$, the equation does not hold for at most
countably many times $t$. This is achieved refining the estimates in
\cite[Lemma~3.1.2]{Ambrosio-Gigli-Savare08} and looking at the
monotonicity properties w.r.t. $t$ of the quantities
$$
\rmD^+(x,t):=\sup\limsup_{n\to\infty} \sfd(x,y_n),\qquad
\rmD^-(x,t):=\inf\liminf_{n\to\infty} \sfd(x,y_n)
$$
where the supremum and the infimum run among all minimizing
sequences $(y_n)$ in \eqref{eq:Qtfintro}. Although only the easier
subsolution property $\partial_tQ_tf+|\rmD  Q_t|^2/2\leq 0$ (which
does not involve the length condition) will play a crucial role in
the results of Sections~\ref{sec:identification1} and
\ref{sec:identification_flows}, another byproduct of this refined
analysis is a characterization of the slope of $Q_tf$ (see
Theorem~\ref{prop:supersol}) which applies, to some extent, also
to Kantorovich potentials (see \ref{sec:weakbrenier}).

In Section~\ref{sec:cheeger} we follow quite closely
\cite{Cheeger00}, defining the collection of relaxed gradients of
$f$ as the weak $L^2$ limits of $|\rmD  f_n|$, where $f_n$ are
$\sfd$-Lipschitz and $f_n\to f$ in $L^2(X,\mm)$ (the differences
with respect to \cite{Cheeger00} are detailed in
Remark~\ref{rem:compachsh}). The collection of all these weak limits
is a convex closed set in $L^2(X,\mm)$, whose minimal element is
called \emph{relaxed gradient}, and denoted by $\relgrad f$. One can
then
see that Cheeger's convex and lower semicontinuous functional
\eqref{def:Cheeger0} can be equivalently represented as 
\begin{equation}\label{def:Cheeger1}\C(f)
=\frac12\int_X\relgrad f^2\,\d\mm\end{equation}
(set to $+\infty$ if $f$ has no relaxed gradient) and get a canonical
gradient flow in $L^2(X,\mm)$ of $\C$ and a notion of Laplacian
$\Deltam $ associated to $\C$. As explained in
Remark~\ref{re:puoesserebanale} and Remark~\ref{re:laplnonlin}, this
construction can be trivial if no other assumption on
$(X,\sfd,\tau,\mm)$ is imposed, and in any case $\C$ is not
necessarily a quadratic form and the Laplacian, though
1-homogeneous, is not necessarily linear. Precisely because of this
potential nonlinearity we avoided the terminology ``Dirichlet
form'', usually associated to quadratic forms, in connection with
$\C$.\\ It is also possible to consider the one-sided slopes
$|\rmD^\pm f|$, getting one-sided relaxed gradients $|\rmD^\pm
f|_*$ and Cheeger's corresponding functionals $\C^\pm$; eventually,
but this fact is not trivial, we prove that the one-sided relaxed
functionals coincide with $\C$, see Remark~\ref{rem:morecheeger}.

Section~\ref{sec:differentchee} is devoted to the ``horizontal''
notion of modulus of gradient, that we call \emph{weak upper
gradient}, along the lines of \cite{Shanmugalingam00}: roughly
speaking, we say that $G$ is a weak upper gradient of $f$ if the
inequality $|f(\gamma_0)-f(\gamma_1)|\leq\int_\gamma G$ holds along
``almost all'' curves with respect to a suitable collection $\calT$
of probability measures concentrated on absolutely continuous
curves, see Definition~\ref{def:weak_upper_gradient} for the precise
statement. The class of weak upper gradients has good stability
properties that allow to define a minimal weak upper gradient, that
we shall denote by $\weakgradA f$, and to prove that  $\weakgradA
f\le \relgrad f$ $\mm$-a.e. in $X$ for all $f\in D(\C)$ if $\calT$
is concentrated on the class of all the absolutely continuous curves
with finite $2$-energy.

Section~\ref{sec:identification1} is devoted to prove the converse
inequality and therefore to show that in fact the two notions of
gradient coincide. The proof relies on the fine analysis of the rate
of dissipation of the entropy $\int_X h_t\log h_t\,\d\mm$ along the
gradient flow of $\C$, and on the representation of $h_t\mm$ as the
time marginal of a random curve. The fact that $h_t\mm$ (having a
priori only $L^2(X,\mm)$ regularity in time and Sobolev regularity
in space) can be viewed as an absolutely continuous curve with
values in $(\Probabilities{X},W_2)$ is a consequence of
Lemma~\ref{le:key}, inspired by
\cite[Proposition~3.7]{GigliKuwadaOhta10}. More precisely, the
metric derivative of $t\mapsto h_t\mm$ w.r.t.~the Wasserstein
distance can be estimated as follows:
\begin{equation}\label{eq:nizza2}
|\dot{h_t\mm}|^2\leq 4\int_X\relgrad{\sqrt{h_t}}^2\,\d\mm\qquad
\text{for a.e. $t\in (0,\infty)$.}
\end{equation}
The latter estimate, written in an integral form, follows by a
delicate approximation procedure, the  Kantorovich duality formula
and the fine properties of the Hopf-Lax semigroup we proved.

{\nc Further consequences of the identification of gradients and a deeper analysis of
the Laplacian on metric measure spaces, still with applications to $LSV$ spaces, have been 
obtained in the more recent paper \cite{Gigli12}.}

In Section~\ref{sec:gflowrele} we introduce the relative entropy
functional  and the Fisher information
$$\entv(\rho\mm)=\int_X\rho\log\rho\,\d\mm,\qquad
\mathsf F(\rho)=4\int_X \relgrad{\sqrt \rho}^2\,\d\mm,$$ and prove
 two crucial inequalities for the descending slope of $\entv$: the first one, still based on
Lemma~\ref{le:key},
provides the lower bound via the Fisher information
\begin{equation}
  \label{eq:46}
  \mathsf F(\rho)=4\int_X \relgrad{\sqrt
    \rho}^2\,\ \d\mm\le |\rmD^-\entv|^2(\mu)\quad\text{if }\mu=\rho\mm,
\end{equation}
and the second one, combining \cite[Theorem~20.1]{Villani09} with an approximation
argument, the upper bound when $\rho$ is $\sfd$-Lipschitz
(and satisfies further technical assumptions if $\mm(X)=\infty$)
\begin{equation}
\label{eq:488}
|\rmD^-\entv|^2(\mu)\leq 4\int_X |{D^-\sqrt
\rho}|^2\,\ \d\mm\quad\text{if }\mu=\rho\mm.
\end{equation}
The identification of the squared descending slope of $\entv$
(which is always a convex functional, as we show in \S
\ref{sec:steepest})
with the Fisher
information thus follows, whenever $|\rmD^-\entv|$ satisfies a lower
semicontinuity property, as in the case of LSV spaces.

In Section \ref{sec:identification_flows} we show how the uniqueness
proof written by the second author in \cite{Gigli10} for the case of finite reference measures  can be adapted, thanks to the tightness properties of
the relative entropy, to our more general framework: we prove
uniqueness of the gradient flow of $\entv$ first for flows with
uniformly bounded densities and then, assuming that
$|\rmD^-\entv|$ is an upper gradient, without any restriction on
the densities. In this way we obtain the key property that the
Wasserstein gradient flow of $\entv$, understood in the metric sense
of Subsection~\ref{se:prelGF}, has a unique solution for a given
initial condition with finite entropy. This uniqueness phenomenon
should be compared with the recent work \cite{Ohta-Sturm10}, where
it is shown that in $LSV$ spaces (precisely in Finsler spaces)
contractivity of the Wasserstein distance along the semigroup may
fail.

In Section~\ref{sec:duality} we prove the equivalence of the two
gradient flows, in the natural class where a comparison is possible,
namely nonnegative initial conditions $f_0\in L^1\cap L^2(X,\mm)$
(if $\mm(X)=\infty$ we impose also that $\int_X
f_0\Wgh^2\,\d\mm<\infty$). In the proof of this result, that
requires suitable assumptions on $|\rmD^-\entv|$, we follow the
new strategy introduced in \cite{GigliKuwadaOhta10}: while the
traditional approach aims at showing that the Wasserstein gradient
flow $\mu_t=f_t\mm$ solves a ``conventional'' PDE, here we show the
converse, namely that the gradient flow of Cheeger's energy provides
solutions to the Wasserstein gradient flow. Then, uniqueness (and
existence) at the more general level of Wasserstein gradient flow
provides equivalence of the two gradient flows. The key properties
to prove the validity of the sharp dissipation rate
$$
-\frac{\d}{\d t}\entv(f_t\mm)\geq \frac{1}{2}|\dot{
f_t\mm}|^2+\frac{1}{2}|\rmD^-\entv(f_t\mm)|^2,
$$
where $f_t$ is the gradient flow of $\C$, are
the slope estimate \eqref{eq:488} and the metric derivative estimate
\eqref{eq:nizza2}.

We also emphasize that some results of ours, as the uniqueness
provided in Theorem~\ref{thm:coincidence_infty} for flows with
bounded densities, or the full convergence as the time step tends to
$0$ of the Jordan-Kinderleher-Otto scheme in
Corollary~\ref{cor:conv}, require \emph{no} assumption on
 the space (except for an exponential volume growth condition) and
$|\rmD^-\entv|$, so that they are applicable even to spaces which
are known to be not $LSV$ or for which the lower semicontinuity of
$|\rmD^-\entv|$ fails or it is unknown, as Carnot groups endowed
with the Carnot-Carath\'eodory distance and the Haar measure.

In Section~\ref{sec:LSV} we show, still following to a large extent
\cite{Gigli10}, the crucial lower semicontinuity of
$|\rmD^-\entv|$ in $LSV$ spaces; this shows that all existence and
uniqueness results of Section~\ref{sec:duality} are applicable to
$LSV$ spaces and that the correspondence between the heat flows is
complete.

The paper ends, in the last section, with results that are important
for the development of a ``calculus'' with Kantorovich potentials.
They will play a key role in some proofs of
\cite{Ambrosio-Gigli-Savare11bis}. We included these results here
because their validity does not really depend on curvature
properties, but rather on their implications, namely the existence
of geodesic interpolations satisfying suitable $L^\infty$ bounds.

Under these assumptions, in Theorem~\ref{thm:brweak} we prove that the ascending slope
$|\rmD^+\varphi|$ is the minimal weak upper gradient for
Kantorovich potentials $\varphi$,
A nice byproduct of this proof is
a ``metric'' Brenier theorem, namely the fact that the transport
distance $\sfd(x,y)$ coincides for $\ggamma$-a.e. $(x,y)$ with
$|\rmD^+\varphi|(x)$ even when the transport plan $\ggamma$ is
multi-valued. In addition, $|\rmD^+\varphi|$ coincides $\mm$-a.e.
with the relaxed and weak upper gradients. To some extent, the
situation here is ``dual'' to the one appearing in the transport
problem with cost=Euclidean distance: in that situation, one knows
the direction of transport, without knowing the distance. In
addition, we obtain in Theorem~\ref{thm:brweak1} a kind of
differentiability property of $\varphi$ along transport geodesics.

Eventually, we want to highlight an important application to the
present paper to the theory of Ricci bounds from below for metric
measure spaces. It is well known that $LSV$
spaces, while stable under Gromov-Hausdorff convergence \cite{Sturm06I} and
consistent with the smooth Riemannian case, include also Finsler
geometries  \cite{Ohta09bis}. It is therefore natural to look for additional axioms,
still stable and consistent, that rule out these geometries, thus
getting a finer description of Gromov-Hausdorff limits of Riemannian
manifolds. In \cite{Ambrosio-Gigli-Savare11bis} we prove, relying in
particular on the results obtained in
Section~\ref{sec:identification1}, Section~\ref{sec:LSV} and
Section~\ref{sec:weakbrenier} of this paper, that $LSV$ spaces whose
associated heat flow is linear have this stability property. In
addition, we show that $LSV$ bounds and linearity of the heat flow
are equivalent to a single condition, namely the existence of
solutions to the Wasserstein gradient flow of $\entv$ in the
\emph{EVI} sense, implying nice contraction and regularization
properties of the flow; we call these \emph{Riemannian} lower bounds
on Ricci curvature. Finally, for this stronger notion we provide
good tensorization and localization properties.

\smallskip
\noindent {\bf Acknowledgement.} The authors acknowledge the support
of the ERC ADG GeMeThNES and warmly thank an anonymous reviewer for his extremely
detailed and constructive report.

\section{Preliminary notions}\label{sec:preliminary}

In this section we introduce the basic metric, topological and
measure-theoretic concepts used in the paper.

\subsection{Extended metric and Polish spaces}

In this paper we consider metric spaces whose distance function may
attain the value $\infty$, we call them \emph{extended metric
spaces}.
\begin{definition}[Extended {distance and extended} metric spaces]
An extended distance on $X$ is a map $\sfd:X^2\to[0,\infty]$
satisfying
\[
\begin{split}
\sfd(x,y)&=0\qquad\textrm{if and only if }x=y,\\
\sfd(x,y)&=\sfd(y,x)\qquad\forall x,\,y\in X,\\
\sfd(x,y)&\leq \sfd(x,z)+\sfd(z,y)\qquad\forall x,\,y,\,z\in X.
\end{split}
\]
If $\sfd$ is an extended distance on $X$, we call $(X,\sfd)$ an
extended metric space.
\end{definition}

Most of the definitions concerning metric spaces generalize verbatim
to extended metric spaces, since extended metric spaces can be
written as a disjoint union of metric spaces,
which are simply defined
as
\begin{equation}
  \label{eq:1}
  \class Xx:=\big\{y\in X:\sfd(y,x)<\infty\big\},\qquad
  x\in X.
\end{equation}
For instance it makes perfectly sense to speak about a complete or
length extended metric space.

\begin{definition}[$\sfd$-Lipschitz functions and Lipschitz constant]
We say that $f:X\to\R$ is $\sfd$-Lipschitz if there exists $C\geq 0$
satisfying
$$
|f(x)-f(y)|\leq C\sfd(x,y)\qquad \forall\, x,\,y\in X.
$$
The least constant $C$ with this property will be denoted by ${\rm
Lip}(f)$.
\end{definition}

In our framework the roles of the distance $\sfd$ (used to define
optimal transport) and of the topology are distinct. This justifies
the following definition. Recall that a topological space $(X,\tau)$
is said to be \emph{Polish} if $\tau$ is induced by a complete and
separable distance.

\begin{definition}[Polish extended spaces]\label{dPolish}
We say that $(X,\tau,\sfd)$ is a Polish extended space if:
\begin{itemize}
\item[(i)] {$\tau$ is a topology on $X$} and $(X,\tau)$ is Polish;
\item[(ii)] {$\sfd$ is an extended distance on $X$} and $(X,\sfd)$ is a complete extended metric space;
\item[(iii)] For $(x_h)\subset X$, $x\in X$, $\sfd(x_h,x)\to 0$ implies $x_h\to x$ w.r.t. to the
topology $\tau$;
\item[(iv)] $\sfd$ is lower semicontinuous in $X\times X$, with
respect to the $\tau\times\tau$ topology.
\end{itemize}
\end{definition}

In the sequel,
 when $\sfd$ is not explicitly mentioned,
all the topological notions (in particular the class of
 compact sets,
the class of Borel sets $\BorelSets{X}$, the class
$C_b(X)$ of bounded continuous functions and the class $\prob X$ of
Borel probability measures) are always referred to the topology
$\tau$, even when $\sfd$ is a distance. When $(X,\sfd)$ is separable
(thus any $\sfd$-open set is a countable union of $\sfd$-closed
balls, which are also $\tau$-closed by (iv)), then a subset of $X$
is $\sfd$-Borel if and only if it is $\tau$-Borel, but when
$(X,\sfd)$ is not separable $\BorelSets X$ can be a strictly smaller
class than the Borel sets generated by $\sfd$.

 The Polish condition on $\tau$ guarantees that all Borel probability
measures $\mu\in \prob X$ are tight, a property (shared with the
more general class of Radon spaces, see e.g.\
\cite[Def.~5.1.4]{Ambrosio-Gigli-Savare08}) which justifies the
introduction of the weaker topology $\tau$. In fact most of the
results of the present paper could be extended to Radon spaces, thus
including Lusin and Suslin topologies \cite{Schwartz73}.

Notice that the only compatibility conditions between the possibly
extended distance $\sfd$ and $\tau$ are (iii) and (iv). Condition
(iii) guarantees that convergence in $(\prob X,W_2)$, as defined in
Section~\ref{sprobt}, implies weak convergence, namely convergence
in the duality with $C_b(X)$. Condition (iv) enables us, when the
cost function $c$ equals $\sfd^2$, to use the standard results of
the Kantorovich theory (existence of optimal plans, duality, etc.)
and other useful properties, as the lower semicontinuity of the
length and the $p$-energy of a curve w.r.t.\ pointwise
$\tau$-convergence, or the representation results of
\cite{Lisini07}.

An example where the roles of the distance and the topology are
different is provided by bounded closed subsets of the dual of a
separable Banach space: in this case $\sfd$ is the distance induced by
the dual norm and $\tau$ is the weak$^*$ topology.
 In this case $\tau$ enjoys better compactness properties than
$\sfd$.

The typical example of Polish extended space is a separable Banach
space $(X,\|\cdot\|)$ endowed with a Gaussian probability measure
$\gamma$. In this case $\tau$ is the topology induced by the norm
and $\sfd$ is the Cameron-Martin extended distance induced by
$\gamma$ (see \cite{Bogachev98}): thus, differently from $(X,\tau)$,
$(X,\sfd)$ is not separable if ${\rm dim\,}X=\infty$.

It will be technically convenient to use also the class
$\BorelSetsStar{X}$ of \emph{universally} measurable sets (and the
associated universally measurable functions): it is the
$\sigma$-algebra of sets which are $\mu$-measurable for any
$\mu\in\prob X$.

\subsection{Absolutely continuous curves and slopes}

If $(X,\sfd)$ is an extended metric space, $J\subset\R$ is an open
interval, $p\in [1,\infty]$ and $\gamma:J\to X$, we say that
$\gamma$ belongs to $\AC pJX\sfd$ if
$$
\sfd(\gamma_s,\gamma_t)\leq\int_s^t g(r)\,\d r\qquad\forall s,\,t\in
J,\,\,s<t
$$
for some $g\in L^p(J)$. The case $p=1$ corresponds to
\emph{absolutely continuous} curves, whose space is simply denoted
by $\AC {}JX\sfd$. It turns out that, if $\gamma$ belongs to $\AC
pJX\sfd$, there is a minimal function $g$ with this property, called
\emph{metric derivative} and given for a.e. $t\in J$ by
$$
|\dot\gamma_t|:=\lim_{s\to t}\frac{\sfd(\gamma_s,\gamma_t)}{|s-t|}.
$$
See \cite[Theorem~1.1.2]{Ambrosio-Gigli-Savare08} for the simple
proof. We say that an absolutely continuous curve $\gamma_t$ has
\emph{constant speed} if $|\dot\gamma_t|$ is (equivalent to) a
constant.

Notice that, by the completeness of $(X,\sfd)$, $\AC pJX\sfd\subset
\CC {\bar J}X\tau$, the set of $\tau$-continuous curves $\gamma:\bar
J\to X$. For $t\in \bar J$ we define the evaluation map
$\e_t:\CC{\bar J} X\tau\to X$ by
\[
\e_t(\gamma):=\gamma_t.
\]
We endow $\CC{\bar J}X\tau$ with the sup extended distance
$$
\sfd^*(\gamma,\tilde\gamma):=\sup_{t\in \bar J}
\sfd(\gamma_t,\tilde\gamma_t)
$$
and with the {compact-open} topology $\tau^*$, whose fundamental
system of neighborhoods is
$$
\left\{\gamma\in\CC {\bar J}X\tau:\ \gamma(K_i)\subset
  U_i,\quad i=1,2,\ldots,n\right\},\qquad
  \text{$K_i\subset \bar J$ compact, $U_i\in\tau$, $n\ge 1$.}
$$
With these choices, it can be shown that $(\CC{\bar
J}X\tau,\tau^*,\sfd^*)$ inherits a Polish extended structure from
$(X,\tau,\sfd)$ {\nc if $\tau$ is induced by a distance $\rho$ in $X$ smaller than 
$\sfd$}. Also, with this topology it is clear that the
evaluation maps are continuous from $(\CC{\bar J}X\tau,\tau^*)$ to
$(X,\tau)$. Since for $p>1$ the $p$-energy
\begin{equation}
  \label{eq:13}
  \Energy p\gamma:=\int_J |\dot \gamma|^p\,\d t\quad\text{if
  }\gamma\in \AC pJX\sfd,\quad
  \Energy p\gamma:=\infty\quad \text{otherwise},
\end{equation}
is $\tau^*$-lower-semicontinuous thanks to (iv) of
Definition~\ref{dPolish}, $\AC pJX\sfd$ is a Borel subset of
$\CC{\bar J}X\sfd$. It is not difficult to check that $\AC {}JX\sfd$
is a Borel set as well; indeed, denoting $J=(a,b)$ and defining
$$
\mathsf{TV}\bigl(\gamma,(a,s)\bigr):=\sup\left\{\sum_{i=0}^{n-1}\sfd(\gamma_{t_{i+1}},\gamma_{t_i}):\
n\in\N,\,\,a<t_0<\cdots<t_n<s\right\}\qquad s\in (a,b],
$$
it can be immediately seen that
$\mathsf{TV}\bigl(\gamma,(a,s)\bigr)$ is lower semicontinuous in
$\gamma$ and nonincreasing in $s$. Also, a continuous $\gamma$ is
absolutely continuous iff the Stieltjes measure associated to
$\mathsf{TV}\bigl(\gamma,(a,\cdot)\bigr)$ is absolutely continuous
w.r.t. $\Leb{1}$; by an integration by parts, this can be
characterized in terms of $m_\eps(\gamma)\downarrow 0$ as
$\eps\downarrow 0$, where
\begin{equation}\label{eq:TV}
m_\eps(\gamma):=\sup\left\{ \int_a^b
\mathsf{TV}\bigl(\gamma,(a,s)\bigr)\psi'(s)\,\d s: \psi\in
C^1_c(a,b),\,\,\max|\psi|\leq 1,\,\,\int_a^b|\psi(s)|\,\d
s\leq\eps\right\}
\end{equation}
if $\mathsf{TV}\bigl(\gamma,(a,b)\bigr)$ is finite,
$m_\eps(\gamma)=+\infty$ otherwise. Since $m_\epsilon$ are Borel in
$\CC{\bar J}X\sfd$, thanks to the separability of $C^1_c(a,b)$
w.r.t. the $C^1$ norm, the Borel regularity of $\AC {}JX\sfd$
follows.

We call $(X,\sfd)$ \emph{a geodesic space} if for any $x_0,\,x_1\in
X$ with $\sfd(x_0,x_1)<\infty$ there exists a curve
$\gamma:[0,1]\to X$ satisfying $\gamma_0=x_0$, $\gamma_1=x_1$ and
\begin{equation}\label{defgeo}
\sfd(\gamma_s,\gamma_t)=|t-s|\sfd(\gamma_0,\gamma_1)\qquad\forall s,\,t\in
[0,1].
\end{equation}

We will denote by $\geo(X)$ the space of all constant speed
geodesics $\gamma:[0,1]\to X$, namely $\gamma\in\geo(X)$ if
\eqref{defgeo} holds.
Given $f:X\to\overline\R$ we define its effective domain $D(f)$ by
\begin{equation}\label{eq:domaineffective}
D(f):=\left\{x\in X:\ f(x)\in\R\right\}.
\end{equation}
Given $f:X\to\overline\R$ and $x\in D(f)$, we define \emph{the local
Lipschitz constant at $x$} by
$$
|\rmD  f|(x):=\limsup_{y\to x}\frac{|f(y)-f(x)|}{\sfd(y,x)}.
$$
We shall also need the one-sided counterparts of the local Lipschitz
constant, called respectively \emph{descending slope} and
\emph{ascending slope}:
\begin{equation}\label{eq:slopes}
|\rmD^- f|(x):=\limsup_{y\to
x}\frac{[f(y)-f(x)]^-}{\sfd(y,x)},\qquad |\rmD^+
f|(x):=\limsup_{y\to x}\frac{[f(y)-f(x)]^+}{\sfd(y,x)}.
\end{equation}

When $x\in D(f)$ is an isolated point of $X$, we set $|\rmD 
f|(x)=|\rmD^- f|(x)=|\rmD^+ f|(x):=0$, while all slopes are
conventionally set to $+\infty$ on $X\setminus D(f)$.

{\nc
Notice the change of notation with respect to previous papers on
similar topics (even by the same authors as the current one):
the local Lipschitz constant and the slopes are sometimes
denoted by $|\nabla f|,|\nabla^\pm f|$. Following \cite{Gigli12}, we
are proposing this switch since these quantities are defined in duality with the
distance and therefore they are naturally cotangent 
objects rather than tangent ones (this observation
  has been used in \cite{Gigli12} as basis for approaching integration
  by parts in metric measure spaces). From this perspective, the
  wording ``upper gradients'' and ``relaxed/weak upper gradients'' that we
  will introduce later on, might be a bit misleading, as the objects
  should rather be called ``(relaxed/weak) upper differentials''. Yet,
  the terminology of upper gradients is by now well established in the
  context of analysis in metric measure spaces, so that we
  will keep it and we will simply replace $\nabla$ by $\rmD$ to highlight the dual point of view.}

Notice that for all $x\in D(f)$ it holds
  \begin{equation}
    \label{eq:2}
    |\rmD  f|(x)=\max \bigl\{|\rmD^- f|(x),|\rmD^+
    f|(x)|\bigr\},\qquad
    |\rmD^- f|(x)=|\rmD^+(-f)|(x).
  \end{equation}
Also, for $f,\,g:X\to\overline\R$ it is not difficult to check that
\begin{subequations}
\begin{align}
\label{eq:subadd}
|\rmD (\alpha f+\beta g)|&\leq|\alpha||\rmD  f|+|\beta||\rmD  g|,\qquad\forall \alpha,\beta\in\R\\
 \label{eq:leibn}
|\rmD  (fg)|&\leq |f||\rmD  g|+|g||\rmD  f|
\end{align}
\end{subequations}
on $D(f)\cap D(g)$. Also, if $\nchi:X\to [0,1]$,
it holds
\begin{equation}\label{eq:subadd2}
|\rmD^\pm(\nchi f+(1-\nchi)g)|\leq\nchi|\rmD^\pm
f|+(1-\nchi)|\rmD^\pm g|+ |\rmD  \nchi| \,|f-g|.
\end{equation}
Indeed, adding the identities
\begin{subequations}
\begin{align}
\nchi(y) f(y)-\nchi(x)f(x)&=
\nchi(y)(f(y)-f(x))+f(x)(\nchi(y)-\nchi(x)),\nonumber \\
\tilde\nchi(y) g(y)-\tilde\nchi(x)g(x)&=
\tilde\nchi(y)(g(y)-g(x))+g(x)(\tilde\nchi(y)-\tilde\nchi(x))\nonumber
\end{align}
\end{subequations}
with $\tilde\nchi=1-\nchi$ one obtains
\begin{align*}
  \frac{(\nchi f+\tilde\nchi g)(y)-(\nchi f+\tilde\nchi g)(x)}{\sfd(y,x)}
  & =
  \nchi(y)\frac{f(y)-f(x)}{\sfd(y,x)}+\tilde\nchi(y)\frac{g(y)-g(x)}{\sfd(y,x)}
  \\&   +\frac {\nchi(y)-\nchi(x)}{\sfd(x,y)}\, (f(x)-g(x))
\end{align*}
from which the inequality readily follows by
 taking the positive or negative parts and
 letting $y\to x$. 

{\GGG We shall prove a further inequality, that will turn to be useful
  to prove contraction estimates for the gradient 
  flow of the Cheeger energy in Section \ref{sec:cheeger}.
  \begin{lemma}
    \label{le:contraction}
    Let $f,g:X\to \R$ be Lipschitz functions, 
    let $\phi:\R\to\R$ be a $C^1$ map with $0\le \phi'\le 1$, and
    let $\psi:[0,\infty)\to \R $ be a convex nondecreasing function.
    Setting 
    \begin{equation}
      \label{eq:109}
      \tilde f:=f+\phi(g-f),\qquad
      \tilde g:=g-\phi(g-f),
    \end{equation}
    we have
    \begin{equation}
      \label{eq:107}
      \psi(|\rmD  \tilde f|(x))+\psi(|\rmD \tilde g|(x))\le 
      \psi(|\rmD  f|(x))+\psi(|\rmD  g|(x))\quad\forevery x\in X.
    \end{equation}
  \end{lemma}
  \begin{proof}
    Let $(y_n)$ be a sequence in $X$ converging to $x$ such
    that 
    $|\rmD \tilde
    f|(x)=\lim_{n\to\infty}\frac{|\tilde f(x)-\tilde f(y_n)|}{\sfd(x,y_n)}$.
    Since $f$ and $g$ are Lipschitz, 
    up to extracting a further subsequence, we can assume that 
    \begin{displaymath}
      \lim_{n\up\infty}\frac{f(x)-f(y_n)}{\sfd(x,y_n)}=A,\qquad
      \lim_{n\up\infty}\frac{g(x)-g(y_n)}{\sfd(x,y_n)}=B.
    \end{displaymath}
    Moreover, by the definition of $\tilde f$ and the Lagrange theorem 
    there exists a convex combination $\xi_n$ of $g(x)-f(x)$ and 
    $g(y_n)-f(y_n)$ such that
    \begin{align*}
      \tilde f(x)-\tilde f(y_n)&=f(x)-f(y_n)+
      \phi(g(x)-f(x))-\phi(g(y_n)-f(y_n))\\&=
      f(x)-f(y_n)+\phi'(\xi_n)\big(g(x)-g(y_n)-(f(x)-f(y_n))\big).
    \end{align*}
    Notice that $|A|\le |\rmD  f|(x)$ and $|B|\le |\rmD  g|(x).$
    Dividing the previous inequality by $\sfd(x,y_n)$ 
    and passing to the limit as $n\to\infty$, since $\phi'(\xi_n)\to
    \alpha:=\phi'(g(x)-f(x))\in [0,1]$
    we get
    \begin{displaymath}
      |\rmD \tilde f|(x)=
      \big|A+\alpha(B-A)\big|
      \le (1-\alpha)|A|+\alpha |B|\le 
      (1-\alpha)|\rmD  f|(x)+\alpha |\rmD  g|(x).
    \end{displaymath}
    A similar argument for $\tilde g$ yields
    \begin{displaymath}
      |\rmD \tilde g|(x)\le (1-\alpha)|\rmD  g|(x)+\alpha |\rmD  f|(x).
    \end{displaymath}
    Since $\psi$ is convex and nondecreasing,
    a combination of the last two inequalities yields \eqref{eq:107}.
  \end{proof}
}

We shall also need the measurability of slopes,
ensured by the following lemma.

\begin{lemma}
  \label{le:measurability_of_slopes}
  If $f:X\to\overline\R$ is Borel, then its slopes
  $|\rmD^\pm f|$ (and therefore $|\rmD  f|$)
  are $\BorelSetsStar{X}$-measurable in $D(f)$.
  In particular, if $\gamma:[0,1]\to X$
  is a continuous curve with $\gamma_t\in D(f)$ for a.e.\
  $t\in [0,1]$, then the functions $|\rmD^\pm f|\circ \gamma$ are Lebesgue measurable.
\end{lemma}
\begin{proof}
  By \eqref{eq:2} it is sufficient to consider the case of the ascending
  slope and, since the functions
  \begin{displaymath}
    G_r(x):=\sup_{\{y:\
      0<\sfd(x,y)<r\}}\frac{(f(y)-f(x))^+}{\sfd(y,x)}
  \end{displaymath}
  {(with the convention $\sup\emptyset=0$, so that $G_r(x)=0$ for $r$
  small enough if $x$ is an isolated point)}
  monotonically converge to $|\rmD^+ f|$ on $D(f)$,
  it is sufficient to prove that $G_r$ is universally measurable for
  any $r>0$.
  For any $r>0$ and $\alpha\geq 0$ we see that the set
  $$
  \left\{x\in D(f):\ G_r(x)>\alpha\right\}
  $$
is the projection on the first factor of the Borel set
$$
\left\{(x,y)\in D(f)\times X:\ f(y)-f(x)>\alpha \sfd(x,y),\quad
0<\sfd(x,y)< r\right\},
$$
so it is a Suslin set (see \cite[Proposition~1.10.8]{Bogachev07})
and therefore it is universally measurable (see
\cite[Theorem~1.10.5]{Bogachev07}).

To check the last statement of the lemma it is sufficient to recall
\cite[Remark 32 (c2)]{Dellacherie-Meyer78} that a continuous curve
$\gamma$ is $(\BorelSetsStar{[0,1]}, \BorelSetsStar X)$ measurable,
since any set in $\BorelSetsStar X$ is measurable for all images of
measures $\mu\in\Probabilities{[0,1]}$ under $\gamma$.
\end{proof}

Finally, for completeness we include the simple proof of the fact
that $|\rmD^-f|=|\rmD^+ f|$ $\mm$-a.e. if $\sfd$ is finite, $f$
is $\sfd$-Lipschitz and $(X,\sfd,\mm)$ is a doubling metric measure
space. We will be able to prove a weaker version of this result even
in non-doubling situations, see Remark~\ref{rem:morecheeger}.

\begin{proposition}
If $\sfd$ is finite, and $(X,\sfd,\mm)$ is doubling, for all
$\sfd$-Lipschitz $f:X\to\R$, $|\rmD^-f|=|\rmD^+ f|$ $\mm$-a.e.
in $X$.
\end{proposition}
\begin{proof} Let $\alpha'>\alpha>0$ and consider the set
 $H:=\{|\rmD^- f|\leq\alpha\}$. Let $H_m$ be
 the subset of points {\nc $x\in H$} such that $f(x)-f(y)\leq
 \alpha'\sfd(x,y)$ for all $y$ satisfying $\sfd(x,y)<1/m$. By the
 doubling property {\nc \cite[Theorem~14.15]{Hailasz-Koskela00}}, the equality $H=\cup_mH_m$ ensures that
 $\mm$-a.e. $x\in H$ is a point of
 density 1 for some set $H_m$. If we fix $\bar x$ with this property, {\nc a corresponding $m$}
 and $\sfd(x_n,\bar x)\to 0$, we can estimate
 $$
 f(x_n)-f(\bar x)= f(x_n)-f(y_n)+f(y_n)-f(\bar x)\leq {\rm Lip}(f)\sfd(x_n,y_n)+
 \alpha'\sfd(y_n,\bar x)
 $$
 choosing $y_n\in H_m\cap B_{1/m}(\bar x)$. But, since the
 density of $H_m$ at $\bar x$ is 1 we can choose $y_n$ in such a way that
  $\sfd(x_n,y_n)=o(\sfd(x_n,\bar x))$. Indeed, if for
  some $\delta>0$ the ball $B_{\delta\sfd(x_n,\bar x)}(x_n)$ does
  not intersect $H_m$ for infinitely many $n$, the upper density of
  $X\setminus H_m$ in the balls {\nc $B_{(1+\delta)\sfd(x_n,\bar x)}(\bar x)$} is
  strictly positive. Dividing both sides by
 $\sfd(x_n,\bar x)$ the arbitrariness of the sequence $(x_n)$
 yields $|\rmD^+ f|(\bar x)\leq\alpha'$.

Since $\alpha$ and $\alpha'$ are arbitrary we conclude that
$|\rmD^+f|\leq|\rmD^- f|$ $\mm$-a.e. in $X$. The proof of the
converse inequality is similar.
\end{proof}

\subsection{Upper gradients}\label{sec:supg}

According to \cite{Cheeger00}, we say that a function $g:X\to
[0,\infty]$ is an \emph{upper gradient} of $f:X\to\overline\R$ if,
for any curve $\gamma\in \AC{}{(0,1)}{D(f)}\sfd$, $s\mapsto
g(\gamma_s)|\dot\gamma_s|$ is measurable in $[0,1]$ (with the
convention $0\cdot\infty=0$) and
\begin{equation}
\left|\int_{\partial\gamma}f\right|\leq\int_\gamma g,\label{eq:68}
\end{equation}
Here and in the following we write $\int_{\partial\gamma}f$ for
$f(\gamma_1)-f(\gamma_0)$ and $\int_\gamma
g=\int_0^1g(\gamma_s)|\dot\gamma_s|\,\d s$.

It is not difficult to see that if $f$ is a Borel and
$\sfd$-Lipschitz function then the two slopes and the local
Lipschitz constant are upper gradients. More generally, the
following remark will be useful.

\begin{remark}[When slopes are upper gradients along a curve]\label{rem:whenslopesare}{\rm
Notice that if one a priori knows that $t\mapsto f(\gamma_t)$ is
absolutely continuous along a given absolutely continuous curve
$\gamma:[0,1]\to D(f)$, then $|\rmD^\pm f|$ are upper gradients of
$f$ along $\gamma$. Indeed, $|\rmD^\pm(f\circ\gamma)|$ are bounded
from above by $|\rmD^\pm f|\circ\gamma|\dot\gamma|$ wherever the
metric derivative $|\dot\gamma|$ exists; then, one uses the fact
that at any differentiability point both slopes of $f\circ\gamma$
coincide with $|(f\circ\gamma)'|$.}\fr
\end{remark}

The next lemma is a refinement of
\cite[Lemma~1.2.6]{Ambrosio-Gigli-Savare08}; as usual, we adopt the
convention $0\cdot \infty=0$.

\begin{lemma}[Absolute continuity criterion]\label{lem:realanalysis}
Let $L\in L^1(0,1)$ be nonnegative and let $g:[0,1]\to [0,\infty]$
be a measurable map with $\int_0^1 L\,\d t>0$ and $\int_0^1 g(t)
L(t)\,\d t<\infty$. Let $w:[0,1]\to\R\cup\{-\infty\}$ be an upper
semicontinuous map, with $w>-\infty$ a.e. on $\{L\neq 0\}$,
satisfying
\begin{equation}\label{eq:hajlasz}
w(s)-w(t)\leq g(t)\biggl|\int_s^tL(r)\,\d r\biggr|\qquad\text{for
all $t\in\{w>-\infty\}$}
\end{equation}
and, for arbitrary $0\le a< b\le 1$,
\begin{equation}
  \label{eq:104}
  \int_a^b L\,\d t=0\qquad\Longrightarrow\qquad
  \text{$w$ is constant in $[a,b]$}.
\end{equation}
Then $\{w=-\infty\}$ is empty and $w$ is absolutely continuous in
$[0,1]$.
\end{lemma}
\begin{proof}
It is not restrictive to assume $\int_0^1 L(t)\,\d t=1$ and set
$\lambda:=L\Leb{1}\vert_{[0,1]}$.
%
%
We introduce the monotone, right continuous map $\sft:[0,1]\to
[0,1]$ pushing $\Leb 1$ onto $\lambda$: setting
\begin{displaymath}
  \sfx(t):=\int_0^t L(r)\,\d r=\lambda([0,t]) \quad\text{it holds}\quad
  \sft(x):=\sup\{t\in [0,1]:\ \sfx(t)\le x\},
\end{displaymath}
and considering the function $\tilde g:=g\circ\sft$ we easily get
\begin{equation}
  \label{eq:103}\int_{\sft(x)}^{\sft(y)}L(r)\,\d r=
   |x-y|\,\,\,\,\,
0\leq x\leq y\leq 1,\,\,\,\,\int_a^b g(t)L(t)\,\d
t=\int_{\sfx(a)}^{\sfx(b)}\tilde g(z)\,\d z
  \,\,\,\,\,a,\,b\in [0,1],
\end{equation}
so that, defining also $\tilde w:=w\circ\sft$, \eqref{eq:hajlasz}
becomes
\begin{equation}
  \label{eq:102}
  \tilde w(y)-\tilde w(x)\le \tilde g(x)|x-y|\quad\text{for all }x\in
  \{\tilde w>-\infty\}.
\end{equation}
Notice that $\tilde w$ is still upper semicontinuous: since it is
the composition of an upper semicontinuous function with the
increasing right continuous map $\sft$, we have just to check this
property at the jump set of $\sft$. If $x\in (0,1]$ satisfies
$\sft_-(x)=\lim_{y\uparrow x}\sft(y)<\sft(x)$, since $w$ is constant
in $[\sft_-(x),\sft(x)]$ we have
$$\limsup_{y\uparrow x}\tilde w(y)=\limsup_{s\uparrow
  \sft_-(x)}w(s)\le w(\sft_-(x))= w(\sft(x))= \tilde w(x).$$
In particular $\tilde w$ is bounded from above and choosing $y_0$
such that $\tilde w(y_0)>-\infty$ we get $\tilde w(x)\ge \tilde
w(y_0)-\tilde g(x)$ for every $x\in \{\tilde w>-\infty\}$, so that
$\tilde w$ is integrable. Since
$$|\tilde w(y)-\tilde w(x)|\le (\tilde g(x)+\tilde g(y))|x-y|\quad
\text{for every }x,y\in (0,1)\setminus \{\tilde w>-\infty\}$$
applying \cite[Lemma 2.1.6]{Ambrosio-Gigli-Savare08} we obtain that
$\tilde w\in W^{1,1}(0,1)$ and $|\tilde w'|\le 2\tilde g$ a.e.\ in
$(0,1)$.

Since $\tilde w\in W^{1,1}(0,1)$ there exists a continuous representative
$\bar w$ of $\tilde w$ in the Lebesgue equivalence class of $\tilde w$. Since
any point in $[0,1]$ can be approximated by points in the
coincidence set we obtain that $\tilde w\geq \bar {w}>-\infty$ in $[0,1]$.
We can apply \eqref{eq:102} to obtain (in the case $y<1$)
$$
\tilde w(y)-\frac{1}{h}\int_{y}^{y+h}\tilde w(x)\,\d x \leq
\int_y^{y+h}\tilde g(x)\,\d x\to 0\quad\text{as $h\down0$}.$$ Since
$\int_y^{y+h}\tilde w(x)\,\d x \sim h\bar w(y)$ as $h\downarrow 0$,
we obtain the opposite inequality $\tilde w(y)\le\bar w(y)$ for
every $y\in [0,1)$. In the case $y=1$ the argument is similar.

We thus obtain
$
  |\tilde w(x)-\tilde w(y)|\le 2\int_x^y \tilde g(z)\,\d r
$ for every $0\le x\le y\le 1$. Now, the fact that $w$ is constant
in any closed interval where $\sfx$ is constant ensures the validity
of the identity $w(s)=w(\sft(\sfx(s))$, so that $w(s)=\tilde
w(\sfx(s))$ and the second equality in \eqref{eq:103} yields
\begin{displaymath}
  |w(s)-w(t)|=|\tilde w(\sfx(s))-\tilde w(\sfx(t))|\le
  2\int_{\sfx(s)}^{\sfx(t)} \tilde g(z)\,\d z=2\int_s^t g(r)L(r)\,\d
  r\quad
  0\le s<t\le 1.
\end{displaymath}
\end{proof}

\begin{corollary}
  \label{cor:real_curves}
  Let $\gamma\in \AC{}{[0,1]}X\sfd$ and let $\varphi:X\to
  \R\cup\{-\infty\}$ be a $\sfd$-upper semicontinuous map such that
  $\varphi(\gamma_s)>-\infty$ a.e. in $[0,1]$.
  Let $g:X\to [0,\infty]$ be such that $g\circ\gamma|\dot\gamma|\in
  L^1(0,1)$ and
  \begin{equation}
    \label{eq:105}
    \varphi(\gamma_s)-\varphi(\gamma_t)\le g(\gamma_t)\left|\int_s^t
      |\dot\gamma|(r)\,\d r\right|\quad\text{for all $t$ such that $\varphi(\gamma_t)>-\infty$}.
  \end{equation}
  Then the map $s\mapsto\varphi(\gamma_s)$ is real valued and absolutely continuous.
\end{corollary}
The \emph{proof} is an immediate application of
Lemma~\ref{lem:realanalysis} with $L:=|\dot\gamma|$ and
$w:=\varphi\circ\gamma$; \eqref{eq:104} is true since $\gamma$ (and
thus $\varphi\circ\gamma$) is constant on every interval where
$|\dot\gamma|$ vanishes a.e.

{\nc Finally, we shall need the following criterion for Sobolev regularity, see
\cite{AGS11c} for the simple proof.

\begin{lemma}\label{lem:Fibonacci}
Let $q\in [1,\infty]$, $f:(0,1)\to\R$ measurable, $g\in L^q(0,1)$ nonnegative be satisfying
$$
|f(s)-f(t)|\leq\bigl|\int_s^t g(r)\,\d r\bigr|\qquad\text{for $\Leb{2}$-a.e. $(s,t)\in (0,1)^2$.}
$$
Then $f\in W^{1,q}(0,1)$ and $|f'|\leq g$ a.e. in $(0,1)$.
\end{lemma}
}

\subsection{The space $(\prob X,W_2)$}\label{sprobt}

Here we assume that  $(X,\tau,\sfd)$ is a Polish extended space. Given
$\mu,\,\nu\in\prob X$, we define the Wasserstein distance $W_2$
between them as
\begin{equation}
W_2^2(\mu,\nu):=\inf\int_{X\times X} \sfd^2(x,y)\,\d\ggamma(x,y),\label{eq:27}
\end{equation}
where the infimum is taken among all $\ggamma\in\prob{X\times X}$
such that
\[
\pi^1_\sharp\ggamma=\mu,\qquad\pi^2_\sharp\ggamma=\nu.
\]
Such measures are called admissible plans {(or couplings)} for the
pair $(\mu,\nu)$. {As usual, if $\mu\in \prob X$ and $T:X\to Y$ is a
$\mu$-measurable map with values in the topological space $Y$, the
push-forward measure $T_\sharp \mu\in \prob Y$ is defined by
$T_\sharp\mu(B):=\mu(T^{-1}(B))$ for every set $B\in\BorelSets Y$.}

We are not restricting ourselves to the space of measures with
finite second moments, so that it can possibly happen that
$W_2(\mu,\nu)=\infty$. Still, via standard arguments one can prove
that $W_2$ is an extended distance in $\prob X$. Also, we point out
that if we define
\[
\pro\mu(X):=\Big\{\nu\in\prob X\ :\ W_2(\mu,\nu)<\infty\Big\}
\]
for some $\mu\in\prob X$, then the space $(\pro\mu(X),W_2)$ is
actually a complete metric space (which reduces to the standard one
$(\probt X,W_2)$ if $\mu$ is a Dirac mass and $\sfd$ is finite).

Concerning the relation between $W_2$ convergence and weak
convergence, the implication
\begin{equation}\label{eq:impliWasse}
W_2(\mu_n,\mu)\to 0\qquad\Longrightarrow\qquad
\int_X\varphi\,\d\mu_n\to\int_X\varphi\,\d\mu\quad\forall\varphi\in
C_b(X)
\end{equation}
is well known if $(X,\sfd)$ is a metric space and $\tau$ is induced
by the distance $\sfd$, see for instance
\cite[Proposition~7.1.5]{Ambrosio-Gigli-Savare08}; the implication
remains true in our setting, with the same proof, thanks to the
compatibility condition (iii) of Definition~\ref{dPolish}.

Since $\sfd^2$ is $\tau$-lower semicontinuous,
 when $W_2(\mu,\nu)<\infty$
the infimum in the definition \eqref{eq:27} of $W_2^2$ is attained
and we call optimal all the plans $\ggamma$ realizing the minimum;
Kantorovich's duality formula holds:
\begin{equation}
\label{eq:dualitabase} \frac12 W_2^2(\mu,\nu)=\sup
\left\{\int_X\varphi\,\d\mu+\int_X\psi\,\d\nu:\
\varphi(x)+\psi(y)\leq \frac12 \sfd^2(x,y)\right\},
\end{equation}
where the functions $\varphi$ and $\psi$ in the supremum are
respectively $\mu$-measurable and $\nu$-measurable, and in $L^1$.
One can also restrict, without affecting the value of the supremum,
to bounded and continuous functions $\varphi$, $\psi$ (see
\cite[Theorem~6.1.1]{Ambrosio-Gigli-Savare08}).

Recall that the \emph{$c$-transform} $\varphi^c$
of $\varphi:X\to\R\cup\{-\infty\}$ is defined by
\[
\varphi^c(y):=\inf\left\{\frac{\sfd^2(x,y)}2-\varphi(x):\ x\in X\right\}
\]
and that $\psi$ is said to be \emph{$c$-concave} if
$\psi=\varphi^c$ for some $\varphi$.

{$c$-concave functions are always $\sfd$-upper
semicontinuous, hence Borel in the case when $\sfd$ is finite and
induces $\tau$.} More generally, it is not difficult to check that
\begin{equation}\label{eq:suslinserve}
\varphi\,\,\,\,\text{Borel}\qquad\Longrightarrow\qquad
\varphi^c\,\,\,\,\text{$\BorelSetsStar{X}$-measurable.}
\end{equation}
 The proof follows, as in Lemma~\ref{le:measurability_of_slopes},
from Suslin's theory: indeed, the set $\{\varphi^c<\alpha\}$ is the
projection on the second coordinate of the Borel set of points
$(x,y)$ such that $\sfd^2(x,y)/2-\varphi(x)<\alpha$, so it is a
Suslin set and therefore universally measurable.

If $\varphi(x),\,\psi(y)$ satisfy $\varphi(x)+\psi(y)\leq
\sfd^2(x,y)/2$, since $\varphi^c\geq\psi$ still satisfies
$\varphi+\varphi^c\leq \sfd^2/2$ and since we may restrict ourselves to
bounded continuous functions, we obtain
\begin{equation}
\label{eq:dualitabasebis} \frac{1}{2}W_2^2(\mu,\nu)=\sup
\left\{\int_X\varphi\,\d\mu+\int_X\varphi^c\,\d\nu:\ \varphi\in
C_b(X)\right\}.
\end{equation}
\begin{definition}[Kantorovich potential]
Assume that $\sfd$ is a finite distance.  We say that a map
$\varphi:X\to\R\cup\{-\infty\}$ is a Kantorovich potential relative
to an optimal plan $\ggamma$ if:
\begin{itemize}
\item[(i)] $\varphi$ is $c$-concave, not identically equal to $-\infty$ and Borel;
\item[(ii)] $\varphi(x)+\varphi^c(y)=\tfrac12\sfd^2(x,y)$ for $\ggamma$-a.e. $(x,y)\in X\times X$.
\end{itemize}
\end{definition}
Since $\varphi$ is not identically equal to $-\infty$ the function
$\varphi^c$ still takes values in $\R\cup\{-\infty\}$ and the
$c$-concavity of $\varphi$ ensures that $\varphi=(\varphi^c)^c$.
Notice that we are not requiring integrability of $\varphi$ and
$\varphi^c$, although condition (ii) forces $\varphi$ (resp.
$\varphi^c$) to be finite $\mu$-a.e. (resp. $\nu$-a.e.).

The existence of maximizing pairs in the duality formula can be a
difficult task if $\sfd$ is unbounded, and no general result is
known when $\sfd$ may attain the value $\infty$. For this reason we
restrict ourselves to finite distances $\sfd$ in the previous
definition and in the next proposition, concerning the main
existence and integrability result for Kantorovich potentials.

\begin{proposition}[Existence of Kantorovich potentials]
If $\sfd$ is finite and $\ggamma$ is an optimal plan with finite
cost, then Kantorovich potentials $\varphi$ relative to $\ggamma$
exist. In addition, if $\sfd(x,y)\leq a(x)+b(y)$ with $a\in
L^2(X,\mu)$ and $b\in L^2(Y,\nu)$, the functions $\varphi$,
$\varphi^c$ are respectively $\mu$-integrable and $\nu$-integrable
and provide maximizers in the duality formula
\eqref{eq:dualitabase}. In this case $\varphi$ is a Kantorovich
potential relative to any optimal plan $\ggamma$.
\end{proposition}
\begin{proof} Existence of $\varphi$ follows by a well-known argument, see for instance
{\nc \cite[Theorem~5.10]{Villani09},} \cite[Theorem~6.1.4]{Ambrosio-Gigli-Savare08}: one makes the
R\"uschendorf-Rockafellar construction of a $c$-concave function
$\varphi$ starting from a $\sigma$-compact and $\sfd^2$-monotone set
$\Gamma$ on which $\ggamma$ is concentrated. The last statement
follows by
$$
\frac{1}{2}W_2^2(\mu,\nu)=\int_X\varphi\,\d\mu+\int_X\varphi^c\,\d\nu\leq
\int_{X\times X}\frac{1}{2}\sfd^2\,\d\ggamma
$$
for any admissible plan $\ggamma$.
\end{proof}

\subsection{Geodesically convex functionals and gradient
flows}\label{se:prelGF}

Given an extended metric space $(Y,\sfd_Y )$ (in the sequel it will
mostly be a Wasserstein space) and $K\in\R$, a functional
$E:Y\to\R\cup\{+\infty\}$ is said to be $K$-geodesically convex if
for any $y_0,\,y_1\in D(E)$ with ${\sfd_Y}(y_0,y_1)<\infty$ there
exists $\gamma\in\geo(Y)$ such that $\gamma_0=y_0$, $\gamma_1=y_1$
and
\[
E(\gamma_t)\leq (1-t)E(y_0)+tE(y_1)-\frac K2t(1-t)
\sfd_Y^2(y_0,y_1)\qquad\forall t\in[0,1].
\]

A consequence of $K$-geodesic convexity is that the descending slope
defined in \eqref{eq:slopes} can be calculated as
\begin{equation}
\label{eq:slopesup} |\rmD^-E|(y)=\sup_{z\in
Y\setminus\{y\}}\left(\frac{E(y)-E(z)}{\sfd_Y (y,z)}+\frac K2
\sfd_Y(y,z)\right)^+,
\end{equation}
so that $|\rmD^-E|(y)$ is the smallest constant $S\geq 0$ such
that
\begin{equation}
  \label{eq:77}
  E(z)\ge E(y)-S\sfd_Y(z,y)+\frac K2\sfd_Y^2(z,y)\quad
  \text{for every }z\in \class Yy.
\end{equation}
We recall (see \cite[Corollary~2.4.10]{Ambrosio-Gigli-Savare08})
that for $K$-geodesically convex {\nc and lower semicontinuous functionals} the descending slope is
an upper gradient, as defined in Section~\ref{sec:supg}: in
particular
\begin{equation}
\label{eq:boundtuttecurve} E(y_t)\ge E(y_s)- \int_s^t |\dot
y_r|\,|\rmD^-E|(y_r)\,\d r \qquad\text{for every }s,t\in
[0,\infty),\ s<t
\end{equation}
for all locally absolutely continuous curves $y:[0,\infty)\to D(E)$.
A metric gradient flow for $E$ is a locally absolutely continuous
curve $y:[0,\infty)\to D(E)$ along which \eqref{eq:boundtuttecurve}
holds as an equality and moreover $|\dot y_t|=|\rmD^- E|(y_t)$ for
a.e. $t\in (0,\infty)$.

An application of Young inequality shows that gradient flows for
functionals can be characterized by the following definition.

\begin{definition}[$E$-dissipation inequality and metric gradient flow]\label{def:dissKconv}
Let $E:Y\to\R\cup\{+\infty\}$ be a functional. We say that a locally
absolutely continuous curve $[0,\infty)\ni t\mapsto y_t\in D(E)$
satisfies the \emph{$E$-dissipation inequality}
if
\begin{equation}\label{eq:edi}
E(y_0)\geq E(y_t)+\frac12\int_0^t|\dot y_r|^2\,\d
r+\frac12\int_0^t|\rmD^- E|^2(y_r)\,\d r\qquad\forall t\geq 0.
\end{equation}
$y$ is a gradient flow of $E$ starting from $y_0\in D(E)$ if
 \eqref{eq:edi} holds as an equality, i.e.\
\begin{equation}\label{eq:ede}
E(y_0)= E(y_t)+\frac12\int_0^t|\dot y_r|^2\,\d
r+\frac12\int_0^t|\rmD^- E|^2(y_r)\,\d r\qquad\forall t\geq 0.
\end{equation}
\end{definition}

By the remarks above, it is not hard to check that \eqref{eq:ede} is
equivalent to the $E$-dissipation inequality \eqref{eq:edi} whenever
$t\mapsto E(y_t)$ is absolutely continuous, in particular if
$|\rmD^- E|$ is an upper gradient of $E$ (as for $K$-geodesically
convex functionals).
 In this case \eqref{eq:ede} is equivalent to
\begin{equation}
  \label{eq:4}
  \frac \d{\d t}E(y_t)=-|\dot y_t|^2=-|\rmD^- E|^2(y_t)\quad \text{for a.e. $t\in
    (0,\infty)$.}
\end{equation}

  If $E:\R^d\to\R$ is a smooth
  functional, then a $C^1$ curve $(y_t)$ is a gradient flow according to the
  previous definition if and only if it satisfies $y_t'=-D E(y_t)$ for all
  $t\in (0,\infty)$, so that the metric definition reduces to the classical one when
  specialized to Euclidean spaces and to regular curves and functionals.

\section{Hopf-Lax semigroup in metric spaces}\label{sec:hopflax}

In this section we study the properties of the functions given by
Hopf-Lax formula in a metric setting and the relations with the
Hamilton-Jacobi equation. Here we only assume that $(X,\sfd)$ is an extended
metric space until Theorem~\ref{thm:subsol} (in particular,
$(X,\sfd)$ is not necessarily $\sfd$-complete or $\sfd$-separable)
and the measure structure $(X,\tau,\mm)$ does not play a role,
except in Proposition~\ref{prop:goodBorel} and
Proposition~\ref{prop:slopeKA}. {\nc Only in Theorem~\ref{prop:supersol} we
will also assume that our space is a length space.}

Let $(X,\sfd)$ be an extended metric space and
$f:X\to\R\cup\{+\infty\}$. We define
\begin{equation}\label{def:Qta}
F(t,x,y):=f(y)+\frac{\sfd^2(x,y)}{2t},
\end{equation}
and
\begin{equation}\label{def:Qtb}
Q_tf(x):=\inf_{y\in X}F(t,x,y)\qquad (x,t)\in X\times (0,\infty).
\end{equation}
The map $ (x,t)\mapsto Q_t f(x),\ X\times
(0,\infty)\to\overline{\R}$ is obviously $\sfd$-upper
semicontinuous. The behavior of $Q_tf$ is not trivial only in the
set
\begin{equation}\label{eq:defdf}
\mathcal D(f):=\left\{x\in X:\ \text{$\sfd(x,y)<\infty$ for some $y$
with $f(y)<\infty$}\right\}
\end{equation}
and we shall restrict our analysis to $\mathcal D(f)$, so that
$Q_tf(x)\in\R\cup\{-\infty\}$ for $(x,t)\in\mathcal D(f)\times
(0,\infty)$.
For $x\in\mathcal D(f)$ we set also
$$
t_*(x):=\sup\{t>0:\ Q_tf(x)>-\infty\}
$$
with the convention $t_*(x)=0$ if $Q_tf(x)=-\infty$ for all $t>0$.
Since $Q_tf(x)>-\infty$ implies $Q_sf(y)>-\infty$ for all $s\in
(0,t)$ and all $y$ at a finite distance from $x$, it follows that
$t_*(x)$ depends only on the equivalence class  $\class Xx$ of $x$,
see \eqref{eq:1}.

Finally, we introduce the functions $\rmD^+(x,t)$, $\rmD^-(x,t)$ as
\begin{equation}\label{eq:defdpm}
\rmD^+(x,t)
:=\sup_{(y_n)}\limsup_n \sfd(x,y_n),\qquad
\rmD^-(x,t)
:=\inf_{(y_n)}\liminf_n \sfd(x,y_n),
\end{equation}
where, in both cases, the $(y_n)$'s vary among all minimizing
sequences of $F(t,x,\cdot)$. It is easy to check (arguing as in
\cite[Lemma~2.2.1, Lemma~3.1.2]{Ambrosio-Gigli-Savare08}) that
$\rmD^+(x,t)$ is finite for $0<t<t_*(x)$ and that
\begin{equation}\label{eq:continuityQtf}
\lim_{i\to\infty}\sfd(x_i,x)=0,\quad\lim_{i\to\infty}t_i=t\in
(0,t_*(x))\qquad\Longrightarrow\qquad
\lim_{i\to\infty}Q_{t_i}f(x_i)=Q_tf(x),
\end{equation}
\begin{equation}\label{eq:boundds}
\sup\left\{\rmD^+(y,t):\ \sfd(x,y)\leq R,\,\,0<t<
t_*(x)-\eps\right\}<\infty \qquad\forall R>0,\,\,\eps>0.
\end{equation}
Simple diagonal arguments show that the supremum and the infimum in
\eqref{eq:defdpm} are attained.

Obviously $\rmD^-(x,\cdot)\leq \rmD^+(x,\cdot)$; the next proposition
shows that both functions are nonincreasing, and that they coincide
out of a countable set.

\begin{proposition}[Monotonicity of $D^\pm$]
For all $x\in\mathcal D(f)$ it holds
\begin{equation}\label{eq:basic_mono}
\rmD^+(x,t)\leq \rmD^-(x,s)<\infty,\qquad 0<t<s<t_*(x).
\end{equation}
As a consequence, $\rmD^+(x,\cdot)$ and $\rmD^-(x,\cdot)$ are both
nondecreasing in $(0,t_*(x))$ and they coincide at all points
therein with at most countably many exceptions.
\end{proposition}
\begin{proof}
Fix $x\in\mathcal D(f)$, $0<t<s<t_*(x)$ and choose minimizing
sequences $(x^n_t)$ and $(x^n_s)$ for $F(t,x,\cdot)$ and
$F(s,x,\cdot)$ respectively, such that $\lim_n \sfd(x,x^n_t)=\rmD^+(x,t)$
and $\lim_n \sfd(x,x^n_s)=\rmD^-(x,s)$. As a consequence, there exist the
limits of $f(x^n_t)$ and $f(x^n_s)$ as $n\to\infty$. The minimality
of the sequences gives
\[
\begin{split}
\lim_nf(x^n_t)+\frac{\sfd^2(x_t^n,x)}{2t}&\leq\lim_nf(x^n_s)+\frac{\sfd^2(x_s^n,x)}{2t}\\
\lim_nf(x^n_s)+\frac{\sfd^2(x_s^n,x)}{2s}&\leq\lim_nf(x^n_t)+\frac{\sfd^2(x_t^n,x)}{2s}.
\end{split}
\]
Adding up and using the fact that $\tfrac1t>\tfrac 1s$ we deduce
\[
\rmD^+(x,t)=\lim_n \sfd(x^n_t,x)\leq\lim_n \sfd(x^n_s,x)= \rmD^-(x,s),
\]
which is \eqref{eq:basic_mono}.
Combining this with the inequality $\rmD^-\leq \rmD^+$ we
immediately obtain that both functions are nonincreasing. At a point
of right continuity of $\rmD^-(x,\cdot)$ we get
$$
\rmD^+(x,t)\leq\inf_{s>t}\rmD^-(x,s)=\rmD^-(x,t).
$$
This implies that the two functions coincide out of a countable set.
\end{proof}
Next, we examine the semicontinuity properties of $D^\pm$: they
imply that points $(x,t)$ where the equality
$\rmD^+(x,t)=\rmD^-(x,t)$ occurs are continuity points for both $\rmD^+$ and
$\rmD^-$.

\begin{proposition}[Semicontinuity of $D^\pm$]
  \label{prop:semiD}
Let $x_n\dto x$ and $t_n\to t\in (0,t_*(x))$. Then
$$
\rmD^-(x,t)\leq\liminf_{n\to\infty}\rmD^-(x_n,t_n),\qquad
\rmD^+(x,t)\geq\limsup_{n\to\infty}D^{+}(x_n,t_n).
$$
In particular, for every $x\in X$ the map $t\mapsto \rmD^-(x,t)$ is
left continuous in $(0,t_*(x))$ and the map $t\mapsto \rmD^+(x,t)$ is
right continuous in $(0,t_*(x))$.
\end{proposition}
\begin{proof} For every $n\in\N$, let $(y_n^i)_{i\in \N}$ be a minimizing sequence for
$F(t_n,x_n,\cdot)$ for which the limit of $\sfd(y_n^i,x_n)$ as
$i\to\infty$ equals $\rmD^-(x_n,t_n)$. From \eqref{eq:boundds} we see
that we can assume that $\sup_{i,n}\sfd(y_n^i,x_n)$ is finite.
For all $n$ we have
\[
\lim_{i\to\infty}f(y_n^i)+\frac{\sfd^2(y_n^i,x_n)}{2t_n}=Q_{t_n}f(x_n).
\]
Moreover, the $\sfd$-upper semicontinuity of $(x,t)\mapsto Q_tf(x)$
gives that $\limsup_nQ_{t_n}f(x_n)\leq Q_tf(x)$. Since
$\sfd(y_n^i,x_n)$ is bounded we have
$\sup_i|\sfd^2(y_n^i,x_n)-\sfd^2(y_n^i,{x})|$ is infinitesimal,
hence by a diagonal argument we can find a sequence $n\mapsto i(n)$
such that
\[
{\limsup_{n\to\infty}}
f(y_{n}^{i(n)})+\frac{\sfd^2(y_n^{i(n)},x)}{2t}\le Q_tf(x),\quad
{\big|\sfd(x_n,y_n^{i(n)})-\rmD^-(x_n,t_n)\big|\le \frac 1n}.
\]
This implies that $n\mapsto y_n^{i(n)}$ is a minimizing sequence for
$F(t,x,\cdot)$, therefore
\[
\rmD^-(x,t)\leq\liminf_{n\to\infty}\sfd(x,y_n^{i(n)})
=\liminf_{n\to\infty}\sfd(x_n,y_n^{i(n)})=\liminf_{i\to\infty}\rmD^-(x_n,t_n).
\]
If we choose, instead, sequences $(y_n^i)_{i\in \N}$ on which the supremum in
the definition of $\rmD^+(x_n,t_n)$ is attained, we obtain the upper
semicontinuity property.
\end{proof}
Before stating the next proposition we recall that semiconcave
functions $g$ on an open interval are local quadratic perturbations
of concave functions; they inherit from concave functions all
pointwise differentiability properties, as existence of right and
left derivatives $\frac{\d^-}{\d t}g\geq\frac{\d^+}{\d t}g$, and similar.

\begin{proposition}[Time derivative of
$Q_tf$]\label{prop:timederivative} The map $(0,t_*(x))\ni t\mapsto
Q_tf(x)$ is locally Lipschitz and locally semiconcave. For all $t\in
(0,t_*(x))$ it satisfies
\begin{equation}\label{eq:Dini1}
  \frac{\d^-}{\d t}Q_tf(x)=
  -\frac{(\rmD^-(x,t))^2}{2t^2},\qquad
  \frac{\d^+}{\d t}Q_tf(x)=
  -\frac{(\rmD^+(x,t))^2}{2t^2}.
\end{equation}
In particular, $s\mapsto Q_sf(x)$ is differentiable at $t\in
(0,t_*(x))$ if and only if $\rmD^+(x,t)=\rmD^-(x,t)$.
\end{proposition}
\begin{proof}
Let $(x^n_t)$, $(x^n_s)$ be minimizing sequences for $F(t,x,\cdot)$
and $F(s,x,\cdot)$. We have
\begin{equation}
Q_sf(x)-Q_tf(x)\leq\liminf_{n\to\infty}
F(s,x,x^n_t)-F(t,x,x^n_t)=\liminf_{n\to\infty}
\frac{\sfd^2(x,x^n_t)}{2}\left(\frac1s-\frac1t\right),\label{eq:5}
\end{equation}
\begin{equation}
Q_sf(x)-Q_tf(x)\geq\limsup_{n\to \infty}
F(s,x,x^n_s)-F(t,x,x^n_s)=\limsup_{n\to\infty}
\frac{\sfd^2(x,x^n_s)}{2}\left(\frac1s-\frac1t\right).\label{eq:6}
\end{equation}
If  $s>t$
we obtain
\begin{equation}
\frac{(\rmD^-(x,s))^2}{2}\left(\frac1s-\frac1t\right) \le Q_sf(x)-Q_tf(x) \leq
\frac{(\rmD^+(x,t))^2}{2}\left(\frac1s-\frac1t\right);\label{eq:7}
\end{equation}
recalling that $\lim_{s\downarrow t}\rmD^-(x,s)=\rmD^+(x,t)$,
a division by $s-t$ and a limit as $s\downarrow t$ gives the identity for the right derivative in
\eqref{eq:Dini1}.
A similar argument, dividing by $t-s<0$ and passing to the limit as
$t\uparrow s$ yields the left derivative in \eqref{eq:Dini1}.

The local Lipschitz continuity follows by \eqref{eq:7} recalling
that $D^\pm(x,\cdot)$ are locally bounded functions; we easily get
the quantitative bound
\begin{equation}\label{eq:partialtquantitative}
\left\|\frac{\d}{\d t}Q_tf(x)\right\|_{L^\infty(\tau,\tau')} \leq
\frac{1}{2\tau^2}\|\rmD^+(x,\cdot)\|_{L^\infty(\tau,\tau')}
\qquad\text{for every }0<\tau<\tau'<t_*(x).
\end{equation}
Since the distributional derivative of the function $t\mapsto
[\rmD^+(x,t)]^2/(2 t^2)$ is locally bounded from below,
we also deduce
that $t\mapsto Q_t f$ is locally semiconcave.
\end{proof}

\begin{proposition}[Slopes and upper gradients of $Q_tf$]\label{prop:slopesqt}
For $x\in\mathcal D(f)$  it holds:
\begin{subequations}
\begin{align}
\label{eq:hjbss}
\qquad \qquad t\in (0,t_*(x))&&\Longrightarrow&&
|\rmD  Q_tf|(x)&\le \frac{\rmD^+(x,t)}t,\qquad\qquad\\
\label{eq:hjups} Q_tf(x)>-\infty&&\Longrightarrow&&
|\rmD^+
Q_tf|(x)&\leq \frac{\rmD^-(x,t)}t.
\end{align}
\end{subequations}
In addition, for all $t\in (0,t_*(x))$, $\rmD^-(\cdot,t)/t$ is an upper
gradient of $Q_tf$ restricted to $X_{[x]}=\{y: \sfd(x,y)<\infty\}$.
\end{proposition}
\begin{proof} Let us first prove that for arbitrary $x,\,y$ be at finite distance
with $Q_t f(y)>-\infty$ we have the estimate
\begin{equation}
  \label{eq:9}
  Q_tf(x)-Q_tf(y)\leq
  \sfd(x,y)\Big(\frac {\rmD^-(y,t)}t+ \frac{\sfd(x,y)}{2t}\Big).
\end{equation}
It is sufficient to take a minimizing sequence $(y_n)$ for $F(t,y,\cdot)$ on which the
infimum in the definition of $\rmD^-(y,t)$ is attained, obtaining
\begin{align*}
  Q_tf(x)-Q_tf(y)&\leq \liminf_{n\to\infty}
  F(t,x,y_n)-F(t,y,y_n)=\liminf_{n\to\infty}\frac{\sfd^2(x,y_n)}{2t}-\frac{\sfd^2(y,y_n)}{2t}
  \notag\\
  &\leq
  \liminf_{n\to\infty}\frac{\sfd(x,y)}{2t}\big(\sfd(x,y_n)+\sfd(y,y_n)\big)
   \leq
    \frac{\sfd(x,y)}{2t}\big(\sfd(x,y)+2\rmD^-(y,t)\big).
\end{align*}
Dividing both sides of \eqref{eq:9} by $\sfd(x,y)$ and taking the
$\limsup$ as $y\to x$ we get \eqref{eq:hjbss} for the descending
slope, since Proposition \ref{prop:semiD} yields the
upper-semicontinuity of $\rmD^+$. The implication \eqref{eq:hjups}
follows by the same argument, by inverting the role of $x$ and $y$
in \eqref{eq:9} and still taking the $\limsup$ as $y\to x$ after a
division by $\sfd(x,y)$. The complete inequality in \eqref{eq:hjbss}
follows by \eqref{eq:2}.

We conclude with the proof of the upper gradient property. Let $t\in
(0,t_*(x))$, let $\gamma:[0,1]\to X_{[x]}$ be an absolutely
continuous curve with constant speed (this is not restrictive, up to
a reparameterizazion), and notice that {\nc 
$s\mapsto Q_tf(\gamma_s)$ is upper semicontinuous in $[0,1]$} whereas
Proposition~\ref{prop:semiD} shows the upper-semicontinuity (and
thus the measurability) of $s\mapsto \rmD^-(\gamma_s,t)$, {\nc while
\eqref{eq:boundds} shows that it is also bounded.} By applying
\eqref{eq:9} with $x=\gamma_{s'}$, $y=\gamma_{s}$ we can use
Corollary~\ref{cor:real_curves} to obtain that $s\mapsto
Q_tf(\gamma_s)$ is absolutely continuous. Coming back to
 \eqref{eq:9} we obtain that $|\frac{\d}{d s} Q_tf(\gamma_s)|\leq
\rmD^-(\gamma_s,t)/t$ for a.e. $s\in [0,1]$.
\end{proof}

\begin{theorem}[Subsolution of HJ]\label{thm:subsol}
For $x\in\mathcal D(f)$ and $t\in (0,t_*(x))$ the right and left
derivatives $\frac{\d^\pm}{\d t}Q_tf(x)$ satisfy
$$
  \frac{\d^+}{\d t}Q_tf(x)+\frac{|\rmD 
    Q_tf|^2(x)}{2}\leq 0,\qquad
  \frac{\d^-}{\d t}Q_tf(x)+\frac{|\rmD^+
    Q_tf|^2(x)}{2}\leq 0.
$$
In particular
\begin{equation}\label{eq:hjbsus}
\frac{\d}{\d t}Q_tf(x)+\frac{|\rmD  Q_tf|^2(x)}{2}\leq 0
\end{equation}
with at most countably many exceptions in $(0,t_*(x))$.
\end{theorem}
\begin{proof}
The first claim is a direct consequence of
Propositions~\ref{prop:timederivative} and \ref{prop:slopesqt}. The
second one \eqref{eq:hjbsus} follows by the fact that the larger
derivative, namely the left one, coincides with
$-[\rmD^-(x,t)]^2/(2t^2)$, and then with $-[\rmD^+(x,t)]^2/(2t^2)$ with at
most countably many exceptions. The latter is smaller than $-|\rmD 
Q_tf|^2(x)/2$ by \eqref{eq:hjbss}.
\end{proof}

We just proved that in an arbitrary extended metric space the
Hopf-Lax formula produces subsolutions of the Hamilton-Jacobi
equation. Our aim now is to prove that, if $(X,\sfd)$ is a
length
space, then the same formula provides also supersolutions.

We say that $(X,\sfd)$ is a \emph{length} space if for all
$x,\,y\in X$ the infimum of the length ${\cal L}(\gamma)$ of continuous curves $\gamma$ 
joining $x$ to $y$ is equal to $\sfd(x,y)$. We remark that under this
assumption it can be proved that the Hopf-Lax formula produces a
semigroup (see for instance the proof in \cite{Lott-Villani07bis}),
while in general only the inequality $Q_{s+t}f\leq Q_s(Q_tf)$ holds.

\begin{theorem}[Solution of HJ and agreement of slopes]
\label{prop:supersol} Assume that $(X,\sfd)$ is a length space. Then
for all $x\in\mathcal D(f)$ and $t\in (0,t_*(x))$ it holds
\begin{equation}
  \label{eq:10}
  |\rmD^- Q_tf|(x)=|\rmD  Q_t f|(x)=\frac{\rmD^+(x,t)}t,
\end{equation}
so that equality holds in \eqref{eq:hjbss}. In particular, the right
time derivative of $Q_tf$ satisfies
\begin{equation}
  \label{eq:12}
  \frac{\d^+}{\d t}Q_tf(x)+\frac{|\rmD  Q_tf|^2(x)}{2}=0\quad
  \text{for every $t\in (0,t_*(x)),$}
\end{equation}
and equality holds in \eqref{eq:hjbsus}, with at most countably many
exceptions.
\end{theorem}
\begin{proof}
Let $(y_i)$ be a minimizing sequence for $F(t,x,\cdot)$ on which the
supremum in the definition of $\rmD^+(x,t)$ is attained, {\nc so that $\sfd(x,y_i)\to \rmD^+(x,t)$.} Let
$\gamma^i:[0,1]\to X$ be {\nc constant speed} curves connecting $x$ to $y_i$
whose lengths $\mathcal L(\gamma^i){\nc\geq\sfd(x,y_i)}$ converge to $\rmD^+(x,t)$. For
every $s\in (0,1)$ we have
\[
\begin{split}
\limsup_{i\to\infty} Q_tf(x)-Q_tf(\gamma^i_s)&\geq
\limsup_{i\to\infty}F(t,x,y_i)-F(t,\gamma^i_s,y_i)\\
&=\limsup_{i\to\infty}\frac{\sfd^2(x,y_i)-\sfd^2(\gamma^i_s,y_i)}{2t},
\end{split}
\]
and our assumption on the $\gamma^i$'s ensures that
{\nc
\[
\lim_{i\to\infty}\frac{\sfd(x,\gamma^i_s)}{s\sfd(x,y_i)}=
\lim_{i\to\infty}\frac{s{\cal L}(\gamma^i)}{s\sfd(x,y_i)}=1,\qquad
\lim_{i\to\infty}\frac{\sfd(\gamma^i_s,y_i)}{(1-s)\sfd(x,y_i)}=
\lim_{i\to\infty}\frac{(1-s){\cal L}(\gamma^i)}{(1-s)\sfd(x,y_i)}=1
\]
}
for all $s\in (0,1)$. Therefore we obtain
\[
\begin{split}
\limsup_{i\to\infty}\frac{Q_tf(x)-Q_tf(\gamma^i_s)}{\sfd(x,\gamma^i_s)}
&\geq\limsup_{i\to\infty}\frac{\big(\sfd(x,y_i)-\sfd(\gamma^i_s,y_i)\big)
\big(\sfd(x,y_i)+\sfd(\gamma^i_s,y_i)\big)}{2 t \sfd(x,\gamma^i_s)}
\\
&=\frac{(2-s)\rmD^+(x,t)}{2t}
\qquad
\text{for all $s\in (0,1)$}.
\end{split}
\]
With a diagonal argument we find {\nc $i(k)\to\infty$
such that
$$
\limsup_{k\to\infty}
\frac{Q_tf(x)-Q_tf(\gamma^{i(k)}_{1/k})}{\sfd(x,\gamma^{i(k)}_{1/k})}\geq
\frac{\rmD^+(x,t)}{t}.
$$
Since $\sfd(x,\gamma^{i(k)}_{1/k})={\cal L}(\gamma^{i(k)})/k\to 0$}
we deduce
\[
|\rmD^-Q_tf|(x)\geq \frac{\rmD^+(x,t)}{t}.
\]
Thanks to \eqref{eq:hjbss} and to the inequality $|\rmD^-Q_t|\leq
|\rmD  Q_t|$, this proves that $|\rmD^-Q_tf|(x) =|\rmD  Q_t
f|(x)=\rmD^+(x,t)/t$.


Taking Proposition~\ref{prop:timederivative} into
account we obtain {\eqref{eq:12}} and that the Hamilton-Jacobi
equation is satisfied at all points $x$ such that
$\rmD^+(x,t)=\rmD^-(x,t)$.
\end{proof}
When $f$ is bounded the maps $Q_tf$ are easily seen to be bounded
and $\sfd$-Lipschitz. It is immediate to see that
\begin{equation}
  \label{eq:108}
  \inf_X f\le \inf_X Q_tf\le \sup_X Q_tf\le \sup_X f.
\end{equation}
A quantitative global estimate we shall need later on is:
\begin{equation}\label{eq:lipquantitative}
{\rm Lip}(Q_tf)\leq 2\sqrt{\frac{{\rm osc}(f)}{t}}, \qquad
\text{where}\quad {\rm osc}(f):=\sup_X f-\inf_X f.
\end{equation}
It can be derived noticing that
choosing a minimizing sequence
$(y_n)_{n\in \N}$ for $F(t,x,\cdot)$ attaining the supremum in \eqref{eq:defdpm}, the energy comparison
$$\frac{(\rmD^+(x,t))^2}{2t}-{\rm osc}(f)\le \lim_{n\to\infty}
f(y_n)+\frac{\sfd^2(x,y_n)}{2t}-f(x)=Q_tf(x)-f(x)\le 0$$ yields
\begin{equation}\label{eq:lipquantitative1}
\rmD^+(x,t)\leq \sqrt{2t\, {\rm osc}(f)}.
\end{equation}
 Since $\rmD^-(x,t)\le \rmD^+(x,t)$, setting
 $R:=(\sqrt 2-1)\sqrt {2t\,{\rm osc}(f)}$, \eqref{eq:9} and simple calculations
yield
\begin{displaymath}
  \frac{Q_tf(x)-Q_tf(y)}{\sfd(x,y)}\le 2\Big(\frac{{\rm
      osc}(f)}t\Big)^{1/2}\quad\text{if }0<\sfd(x,y)\le R,
\end{displaymath}
and, since ${\rm osc}(Q_t f)\le {\rm osc}(f)$ by \eqref{eq:108},
\begin{displaymath}
  \frac{Q_tf(x)-Q_tf(y)}{\sfd(x,y)}\le \frac{{\rm
      osc}(Q_t f)}{ R}\le 2\Big(\frac{{\rm
      osc}(f)}t\Big)^{1/2} \quad\text{if }\sfd(x,y)\ge R.
\end{displaymath}
The constant $2$ in \eqref{eq:lipquantitative} can be reduced to
$\sqrt 2$ if $X$ is a length space: it is sufficient to combine
\eqref{eq:lipquantitative1} with \eqref{eq:hjbss}.

We conclude this section with a simple observation, a technical
lemma, where also a Polish structure is involved, and with some
relations between slope of Kantorovich potentials and Wasserstein
distance.

\begin{remark}[Continuity of $Q_t$ at $t=0$]\label{rem:goodBorel} {\rm
If $(X,\tau,\sfd)$ is an Polish extended space and $\varphi$ is
bounded and $\tau$-lower semicontinuous, then
$Q_t\varphi\uparrow\varphi$ as $t\downarrow 0$. This is a simple
consequence of assumption (iii) in Definition~\ref{dPolish}. \fr
}\end{remark}

\begin{proposition}\label{prop:goodBorel}
Let $(X,\tau,\sfd)$ be an Polish extended space.
\begin{itemize}
\item[(i)] if $K\subset X$ is compact, $\psi\in C(K)$, $M\geq\max\psi$ and
\begin{equation}\label{eq:compactreduction}
\varphi(x)=\begin{cases}
\psi(x) &\text{if $x\in K$},
\\
M &\text{if $x\in X\setminus K,$}
\end{cases}
\end{equation}
then $Q_t\varphi$ is $\tau$-lower semicontinuous in $X$ for all
$t>0$;
\item[(ii)] if $\mathcal D(f)=X$, $t_*(x)\geq T>0$ for all $x\in X$ and
$Q_t\varphi$ is Borel measurable for all $t>0$ then $\tfrac{\\rmD^+}{\d
t}Q_t\varphi(x)$ is Borel measurable in $X\times (0,T)$ and the
slopes
$$(x,t)\mapsto |\rmD^+ Q_t\varphi|(x),\qquad (x,t)\mapsto |\rmD^- Q_t\varphi|(x)
$$ are $\BorelSetsStar{X\times
(0,T)}$-measurable in $X\times (0,T)$.
\end{itemize}
\end{proposition}
\begin{proof} (i) The proof is straightforward, using the identity
{\nc
$$
Q_t\varphi(x)=\min\biggl\{\min_{y\in K}\psi(y)+\frac{1}{2t}\sfd^2(x,y),M\biggr\}.
$$
}

(ii) A simple time
discretization argument also shows that $(x,t)\mapsto Q_t\varphi(x)$
is Borel measurable. 
The Borel measurability of $\frac{\d^+}{\d t}Q_t\varphi(x)$ is a
simple consequence of the continuity of $t\mapsto Q_t\varphi(x)$,
together with the Borel measurability of $Q_t\varphi$. Then, the proof of the measurability of slopes
follows as in Lemma~\ref{le:measurability_of_slopes}.
\end{proof}

In the next proposition we consider the ascending slope of
Kantorovich potentials, for finite distances $\sfd$.
\begin{proposition}[Slope and approximation of Kantorovich
potentials]\label{prop:slopeKA} {\nc Assume that $\sfd$ is a finite distance},  
let $\mu,\,\nu\in\Probabilities{X}$ with $W_2(\mu,\nu)<\infty$ and let $\ggamma\in\prob{X\times X}$ be
an optimal plan with marginals $\mu,\,\nu$. If $\varphi$ is a
Kantorovich potential relative to $\ggamma$, we have
\begin{equation}\label{eq:wbre1}
|\rmD^+\varphi|(x)
\leq \sfd(x,y) \quad\text{for $\ggamma$-a.e.
$(x,y)$.}
\end{equation}
In particular $|\rmD^+\varphi|\in L^2(X,\mu)$ and
$\int_X|\rmD^+\varphi|^2\,\d\mu\leq W_2^2(\mu,\nu)$.
\end{proposition}
\begin{proof}
We set $f:=-\varphi^c$, so that from $\varphi=(\varphi^c)^c$ we have
$\varphi=Q_1f$. In addition, the definition of Kantorovich potential
tells us that $\varphi(x)=f(y)+\sfd^2(x,y)/2$ for $\ggamma$-a.e.\
$(x,y)$, so that
\begin{equation}
  \label{eq:100}
  \rmD^-(x,1)\leq \sfd(x,y)\quad\text{for $\ggamma$-a.e.
    $(x,y)$.}
\end{equation}
Taking \eqref{eq:hjups} into account we obtain
\eqref{eq:wbre1}.
\end{proof}

In general the inequality $\int_X|\rmD^+\varphi|^2\,\d\mu\leq
W_2^2(\mu,\nu)$ can be strict, as the following simple example shows:

\begin{example}\label{ex:strict}{\rm
Let $X=[0,1]$ endowed with the Euclidean distance, $\mu_0=\delta_0$
and $\mu_t=t^{-1}\nchi_{[0,t]}\Leb{1}$ for $t\in (0,1]$. Then
clearly $(\mu_t)$ is a constant speed geodesic connecting $\mu_0$ to
$\mu_1$ and the corresponding Kantorovich potential is
$\varphi(x)=x^2/2-x$, so that $\int| D^+\varphi|^2\,\d\mu_0=0$,
while $W_2^2(\mu_0,\mu_1)=1/3$.}
\end{example}

\section{Relaxed gradient, Cheeger's energy, and its $L^2$-gradient flow}\label{sec:cheeger}

In this section we assume that $(X,\tau,\sfd)$ is a Polish extended
space. Furthermore, $\mm$ is a nonnegative, Borel and
$\sigma$-finite measure on $X$. Recall that
\begin{equation}
  \label{eq:17}
  \text{there exists a bounded Borel function $\vartheta:X\to
  (0,\infty)$ such that $\int_X \vartheta\,\d\mm\le 1$.}
\end{equation}
Notice that $\mm$ and the finite measure $\tmm:=\vartheta\mm$ share
the same class of negligible sets. In the following we will often
assume that $\mm$ and $\vartheta$ satisfy some further structural
conditions, which will be described as they occur. For future
references, let us just state here our strongest assumption in
advance: we will often assume that $\vartheta$ has the form
$\rme^{-V^2}$, where
\begin{equation}
  \label{eq:75}
  \begin{gathered}
    \text{$V:X\to
      [0,\infty)$ is a Borel $\sfd$-Lipschitz map, }\\
    \text{it is bounded on each
      compact set $K\subset X$, and}\quad
    \int_X \rme^{-V^2}\,\d\mm\le 1.
    \end{gathered}
\end{equation}
When $\tau$ is the topology induced by the finite distance $\sfd$,
then the facts that $V$ is Borel and bounded on compact sets are
obvious consequences of the $\sfd$-Lipschitz property. In this case
a simple choice is $\Wgh(x)=\sqrt{\kappa/2}\,\sfd(x,x_0)$ for some
$x_0\in X$ and $\kappa>0$. It is not difficult to check that
\eqref{eq:75} is then equivalent to
\begin{equation}
  \label{eq:78}
  \exists \,\kappa>0:\quad
  m(r)\le \rme^{\frac \kappa2 r^2}\quad\text{where}\quad
  m(r):=\mm\big(\{x\in X:\sfd(x,x_0)<r\}\big).
\end{equation}
In fact, for every $h>0$
 \begin{equation}
    \label{eq:41}
    \int_X \rme^{-\frac h2 \sfd^2(x,x_0)}\,\d\mm=
    \int_X \int_{r>\sfd(x,x_0)}h\, r\,\rme^{-\frac h 2r^2}\,\d r\,\d\mm(x)=
    \int_0^\infty h \,r \,m(r)\,\rme^{-\frac h2 r^2}\,\d r.
  \end{equation}
  Since $r\mapsto m(r)$ is nondecreasing, if the last integral in
  \eqref{eq:41} is less than $1$ for $h:=\kappa$, then Chebichev inequality yields
  $m(r)\rme^{-\frac12\kappa r^2}\le 1$; on the other hand, if
  \eqref{eq:78} holds, then there exists  $h>\kappa$ sufficiently big
  such that the integral in \eqref{eq:41} is less than $1$, so that
  \eqref{eq:75} holds.

\subsection{Minimal relaxed gradient}
The content of this subsection is inspired by Cheeger's work
\cite{Cheeger00}. We are going to relax the integral of the squared
local Lipschitz constant of Lipschitz functions with respect to the
$L^2(X,\mm)$ topology. By Lemma~\ref{le:measurability_of_slopes},
$|\rmD  f|$ is $\BorelSetsStar{X}$-measurable whenever $f$ is
$\sfd$-Lipschitz and Borel.

\begin{proposition}\label{prop:densitydlip}
Let $(X,\tau,\sfd)$ be an Polish extended space and let $\mm$ be a
nonnegative, Borel measure in $(X,\tau)$ satisfying the following
condition (weaker than \eqref{eq:75}):
\begin{equation}
  \label{eq:19}
  \forall\, K\subset X\text{  compact}\quad \exists r>0:\quad
  \mm\big(\{x\in X:\sfd(x,K)\le r\}\big)<\infty.
\end{equation}
Then the class of bounded, Borel and $\sfd$-Lipschitz functions
$f\in L^2(X,\mm)$ with $|\rmD  f|\in L^2(X,\mm)$ is dense in
$L^2(X,\mm)$.
\end{proposition}
\begin{proof} It suffices to approximate functions $\varphi:X\to\R$ such that for
some compact set $K\subset X$
$$
  \varphi\restr{K}\in C^0(K),\quad
  \varphi\equiv 0\quad\text{in }X\setminus K.
$$
By taking the positive and negative part, we can always assume that
$\varphi$ is, e.g., nonnegative. We can thus define
$$
  \varphi_n(x):=\sup_{y\in K}\bigl[\varphi(y)-n\sfd(x,y)\Big]^+.
$$
It is not difficult to check that $\varphi_n$ is upper
semicontinuous, nonnegative, $n$-Lipschitz and bounded above by
$S:=\max_K\varphi\ge 0$; moreover
$$
  \varphi_n(x)=|\rmD  \varphi_n|(x)=0\quad\text{if }\sfd(x,K)> S/n.
$$
If $r>0$ is given by \eqref{eq:19}, choosing $n>S/r$ we get that
$\varphi_n,|\rmD \varphi_n|$ are supported in the set $\{x\in
X:\sfd(x,K)\le r\}$ of finite measure, so that they belong to
$L^2(X,\mm)$; since $S\ge \varphi_n(x)\ge \varphi(x)$ and
$\varphi_n(x)\downarrow\varphi(x)$ for every $x\in X$, we conclude.
\end{proof}

\begin{definition}[Relaxed gradients]\label{def:genuppergrad} We say that $G\in L^2(X,\mm)$ is a
relaxed gradient of $f\in L^2(X,\mm)$ if there exist Borel
$\sfd$-Lipschitz functions $f_n\in L^2(X,\mm)$ such that:
\begin{itemize}
\item[(a)] $f_n\to f$ in $L^2(X,\mm)$ and $|\rmD  f_n|$
weakly converge to $\tilde{G}$ in $L^2(X,\mm)$;
\item[(b)] $\tilde{G}\leq G$ $\mm$-a.e. in $X$.
\end{itemize}
We say that $G$ is the minimal relaxed gradient of $f$ if its
$L^2(X,\mm)$ norm is minimal among relaxed gradients. We shall denote
by $\relgrad f$ the minimal relaxed gradient.
\end{definition}

The definition of minimal relaxed gradient is well posed; indeed,
thanks to \eqref{eq:subadd} and to the reflexivity of $L^2(X,\mm)$,
the collection of relaxed gradients of $f$ is a convex set, possibly
empty. Its closure follows by the following lemma, which also shows
that it is possible to obtain the minimal relaxed gradient as
\emph{strong} limit in $L^2$.

\begin{lemma}[Closure and strong approximation of the minimal
  relaxed gradient]\label{lem:strongchee}
  \
  \begin{itemize}
  \item[(a)]
    If {$G\in L^2(X,\mm)$ is a relaxed gradient of $f\in
      L^2(X,\mm)$ then} there exist Borel $\sfd$-Lipschitz functions
    $f_n$ converging to $f$ in $L^2(X,\mm)$ and
    $G_n\in L^2(X,\mm)$ strongly convergent to $\tilde G$ in
    $L^2(X,\mm)$ with $|\rmD  f_n|\le G_n$ and $\tilde G\le G$.
  \item[(b)]
    If $G_n\in L^2(X,\mm)$ is a relaxed gradient of $f_n\in L^2(X,\mm)$
    and $f_n\weakto f$, $G_n\weakto G$ weakly in $L^2(X,\mm)$,
    then $G$ is a relaxed gradient of $f$.
  \item[(c)]
    In particular, the collection of all the relaxed gradients of $f$ is closed in
    $L^2(X,\mm)$ and there exist {\nc bounded} Borel $\sfd$-Lipschitz functions
    $f_n\in L^2(X,\mm)$ such that
    \begin{equation}
      \label{eq:16}
      f_n\to f,\quad
      |\rmD  f_n|\to \relgrad f\quad\text{strongly in }L^2(X,\mm).
    \end{equation}
  \end{itemize}
\end{lemma}
\begin{proof} (a) Since $G$ is a relaxed gradient, we can find Borel $d$-Lipschitz
functions $g_i\in L^2(X,\mm)$ such that $g_i\to f$ in $L^2(X,\mm)$
and $|\rmD  g_i|$ weakly converges to $\tilde G\le G$ in
$L^2(X,\mm)$; by Mazur's lemma we can find a sequence of convex
combinations $G_n$ of $|\rmD  g_i|$, starting from an index
$i(n)\to\infty$, strongly convergent to $\tilde G$ in $L^2(X,\mm)$;
the corresponding convex combinations of $g_i$, that we shall denote
by $f_n$, still converge in $L^2(X,\mm)$ to $f$ and $|\rmD  f_n|$
is bounded from above by $G_n$.

(b) Let us prove now the weak closure in $L^2(X,\mm)\times
L^2(X,\mm)$ of the set
$$S:=\big\{(f,G)\in L^2(X,\mm)\times L^2(X,\mm):
\text{$G$ is a relaxed gradient for $f$}\big\}.
$$
Since $S$ is convex, it is sufficient to prove that $S$ is strongly
closed. If $S\ni(f^i,G^i)\to (f,G)$ strongly in $L^2(X,\mm)\times
L^2(X,\mm)$, we can find sequences of Borel $d$-Lipschitz functions
$(f^i_n)_n\in L^2(X,\mm)$ and of nonnegative functions $(G^i_n)_n\in
L^2(X,\mm)$ such that
$$
  f^i_n\stackrel{n\to\infty}\longrightarrow f^i, \quad
  G^i_n \stackrel{n\to\infty}\longrightarrow  \tilde G^i
  \text{ strongly in $L^2(X,\mm)$,}
  \quad
  |\rmD  f^i_n|\le G^i_n,\quad
  \tilde G^i\le G^i.
$$
Possibly extracting a suitable subsequence, we can assume that $\tilde
G^i\weakto \tilde G$ weakly in $L^2(X,\mm)$ with $\tilde G\le G$;
{\nc
by a standard diagonal argument 
we can find an increasing sequence $i\mapsto n(i)$ such that
$f^i_{n(i)}\to f$, $G^i_{n(i)}\weakto \tilde G$ in $L^2(X,\mm)$ and $|\rmD  f^i_{n(i)}|$ is bounded in
$L^2(X,\mm)$.
By the reflexivity of $L^2(X,\mm)$ we can also assume, possibly extracting a further subsequence,
that $|\rmD  f^i_{n(i)}|\weakto H$.}
It follows that $H\le \tilde G\le G$ so that $G$ is a relaxed gradient
for $f$.

(c) Let us consider now the minimal relaxed gradient $G:=\relgrad f$
and let $f_n$, $G_n$ be sequences in $L^2(X,\mm)$ as in the first
part of the present Lemma. Since $|\rmD  f_n|$ is uniformly bounded
in $L^2(X,\mm)$ it is not restrictive to assume that it is weakly
convergent to some limit $H\in L^2(X,\mm)$ with $0\le H\le \tilde
G\le G$. This implies at once that  $H=\tilde G=G$ and $|\rmD 
f_n|$ weakly converges to $\relgrad f$ (because any limit point in
the weak topology of $|\rmD  f_n|$ is a relaxed gradient with minimal
norm) and that the convergence is strong, since
\begin{displaymath}
  \limsup_{n\to\infty}\int_X |\rmD  f_n|^2\,\d\mm\leq
  \limsup_{n\to\infty}\int_X G_n^2\,\d\mm=\int_X G^2\,\d\mm=
  \int_X H^2\,\d\mm.
\end{displaymath}
{\nc Finally, replacing $f_n$ by suitable truncations $\tilde f_n$, made in such a way that
$\tilde f_n\to f$ in $L^2(X,\mm)$, we can achieve the boundedness property retaining
the strong convergence of $|\rmD  \tilde f_n|$ to $\relgrad{f}$, since $|\rmD \tilde f_n|\leq |\rmD  f_n|$
and any weak limit point of $|\rmD \tilde f_n|$ is a relaxed gradient.}
\end{proof}
{\GGG The minimal relaxed gradient satisfies 
  a ``Leibnitz'' rule:
  if $f,g\in L^2(X,\mm)\cap
  L^\infty(X,\mm)$ have a relaxed gradient, then their product 
  $fg$ has a relaxed gradient as well, with
  \begin{equation}
    \label{eq:weak-leibn}
    \relgrad {(fg)}\le |f|\,\relgrad g+|g|\,\relgrad f.
  \end{equation}
  It is sufficient to approximate $f,g$ by two sequences 
  $f_n,g_n$ of bounded Lipschitz functions as 
  in (c) of
  lemma \ref{lem:strongchee} (notice that 
  $f_n,g_n$ can be assumed uniformly bounded  by truncation)
  and then pass to the limit in \eqref{eq:leibn}.}

The distinguished role of the minimal relaxed gradient
is also illustrated by the following lemma.

\begin{lemma}[Locality]\label{lem:locality}
Let $G_1,\,G_2$ be relaxed gradients of $f$. Then $\min\{G_1,G_2\}$ and
$\nchi_BG_1+\nchi_{X\setminus B}G_2$,  $B\in\BorelSets{X}$, are
relaxed gradients of $f$ as well. In particular, for any relaxed gradient
$G$ of $f$ it holds
\begin{equation}\label{eq:facileee}
\relgrad  f\leq G\qquad\text{$\mm$-a.e. in $X$.}
\end{equation}
\end{lemma}
\begin{proof} It is sufficient to prove that if $B\in\BorelSets{X}$, then
$\nchi_{X\setminus B}G_1+\nchi_BG_2$ is a relaxed gradient of $f$. By
approximation, taking into account the closure of the class of
relaxed gradients, we can assume with no loss of generality that
$X\setminus B$ is a compact set, so that
the $\sfd$-Lipschitz function
$$
\rho(y):=\inf\left\{\sfd(y,x):\ x\in X\setminus B\right\}
$$
is $\tau$-lower semicontinuous and therefore
$\BorelSets{X}$-measurable. Notice that, because of condition (iii)
in Definition~\ref{dPolish}, $\rho$ is strictly positive in $B$ and
null on $X\setminus B$. Therefore it will be sufficient to show
that, setting $\nchi_r:=\min\{1,\rho/r\}$, $\nchi_rG_1+(1-\nchi_r)G_2$
is a relaxed gradient for all $r>0$.

Let now $f_{n,i}$, $i=1,\,2$, be Borel, $\sfd$-Lipschitz and
$L^2(X,\mm)$ functions converging to $f$ in $L^2$ as $n\to\infty$
with $|\rmD  f_{n,i}|$ weakly convergent to $ \tilde G_i\le G_i$,
and set $f_n:=\nchi_r f_{n,1}+(1-\nchi_r)f_{n,2}$. Then
\eqref{eq:subadd2} immediately gives that $\nchi_r
G_1+(1-\nchi_r)G_2
 \ge \nchi_r \tilde G_1+(1-\nchi_r)\tilde G_2$ is a relaxed gradient.

For the second part of the statement we argue by contradiction: let
$G$ be a relaxed gradient of $f$ and assume that {there exists a Borel
set $B$ with $\mm(B)>0$ on which $G<\relgrad f$}. Consider the
relaxed gradient $G\nchi_B+\relgrad f\nchi_{X\setminus B}$: its $L^2$
norm is strictly less than the $L^2$ norm of $\relgrad f$, which is
a contradiction.
\end{proof}
By \eqref{eq:facileee}, for $f$ Borel and $\sfd$-Lipschitz we get
\begin{equation}
\label{eq:facile} \relgrad  f\leq |\rmD  f|\qquad\text{$\mm$-a.e.
in $X$.}
\end{equation}
A direct byproduct of this characterization of $\relgrad f$ is its
invariance under multiplicative perturbations of $\mm$ of the form
$\theta\,\mm$, with
\begin{equation}
  \label{eq:8}
   0<c\leq \theta \leq C<\infty\quad
  \text{$\mm$-a.e.\ on $X$}.
\end{equation}
Indeed, the class of relaxed gradients is invariant under these
perturbations.

\begin{theorem} \label{thm:cheeger} Cheeger's functional
\begin{equation}\label{def:Cheeger}
\C(f):=\frac{1}{2}\int_X \relgrad f^2\,\d\mm,
\end{equation}
set equal to $+\infty$ if $f$ has no relaxed slope,
is convex and lower semicontinuous in $L^2(X,\mm)$.
If \eqref{eq:19} holds, then its domain is dense in $L^2(X,\mm)$.
\end{theorem}
\begin{proof} A simple byproduct of condition \eqref{eq:subadd} is that
$\alpha F+\beta G$ is a relaxed gradient of $\alpha f+\beta g$ whenever
$\alpha,\,\beta$ are nonnegative constants and $F,\,G$ are relaxed
gradients of $f,\,g$ respectively. Taking $F=\relgrad f$ and
$G=\relgrad g$ yields
{\nc\begin{equation}
  \label{eq:82}
  \relgrad{(\alpha f+\beta g)}\le |\alpha| \relgrad f+|\beta|\relgrad g\quad
  \text{for every }f,g\in D(\C),\ \alpha,\beta\in\R.
\end{equation}}
This proves the convexity of $\C$, while lower semicontinuity
follows by (b) of Lemma~\ref{lem:strongchee}.
\end{proof}

\begin{remark}[The Sobolev space
  $W^{1,2}_*(X,\sfd,\mm)$]\label{re:sobolev}
\upshape
As a simple consequence of the lower semicontinuity of the Cheeger's
functional, it can be proved that the domain of $\C$ endowed with
the norm
\[
\|f\|_{W_*^{1,2}}:=\sqrt{\|f\|^2_2+\|\relgrad f\|^2_2},
\]
is a Banach space (for the proof, just adapt the arguments in 
\cite[Theorem~2.7]{Cheeger00}). Call $W^{1,2}_*(X,\sfd,\mm)$ this space. {\nc Notice that for the moment we don't know whether this space coincides with the Sobolev space $W^{1,2}(X,\sfd,\mm)$ introduced in \cite{Cheeger00} (see Remark~\ref{rem:compachsh} below), which is the standard one used in the context of analysis in metric measure spaces,  whence the distinguished notation that we will keep until we will prove in Theorem~\ref{thm:density_energy} that this new definition coincides with the existing one.}

It is important to remark that, in general,
$W^{1,2}_*(X,\sfd,\mm)$ is \emph{not} an Hilbert space. This is the
case, for example, of the metric measure space
$(\R^d,\|\cdot\|,\Leb{d})$ where $\|\cdot\|$ is any norm not coming
from an inner product. The fact that $W_*^{1,2}(X,\sfd,\mm)$ may fail
to be Hilbert is strictly related to the potential lack of linearity
of the heat flow, see also Remark~\ref{re:laplnonlin} below (for
computations in smooth spaces with non linear heat flows see
\cite{Ohta-Sturm09}). Also, the reflexivity of $W_*^{1,2}$ and the
density of Lipschitz functions in $W^{1,2}_*$ norm seem to be
difficult problems at this level of generality, while it is known
that both these facts are true in doubling spaces satisfying a local
Poincar\'e inequality, see \cite{Cheeger00}.\fr
\end{remark}

\begin{remark}[Cheeger's original functional]\label{rem:compachsh} \upshape Our definition of $\C$ can be
compared with the original one in \cite{Cheeger00}: the relaxation
procedure is similar, but the approximating functions $f_n$ are not
required to be Lipschitz and $|\rmD  f_n|$ are replaced by upper
gradients $G_n$ of $f_n$. Obviously, this leads to a \emph{smaller}
functional, that we shall denote by $\underline{\sf Ch}_*$; this
functional can still be represented by the integration of a local
object, smaller $\mm$-a.e. than $\relgrad{f}$, that we shall denote
by $|\rmD  f|_C$. {\nc Then the Sobolev space $W^{1,2}(X,\sfd,\mm)$ is defined as the domain of $\underline{\sf Ch}_*$ endowed with the norm
\[
\|f\|_{W^{1,2}}:=\sqrt{\|f\|^2_2+\||\rmD  f|_C\|^2_2}.
\]
This is the definition of Sobolev space adopted as the standard in metric measure spaces (and agrees with the one of Newtonian space 
given by Shanmugalingam in \cite{Shanmugalingam00}, see Remark~\ref{rem:compachsh2}). The inequality $|\rmD  f|_C\leq \relgrad f$ $\mm$-a.e. yields
\[
W_*^{1,2}(X,\sfd,\mm)\subset W^{1,2}(X,\sfd,\mm).
\]
Relating $W_*^{1,2}(X,\sfd,\mm)$ to $W^{1,2}(X,\sfd,\mm)$, and hence $\C$ to $\underline{\sf Ch}_*$, amounts
to find, for any $f\in L^2(X,\mm)$ and any upper gradient $G$ of
$f$, a sequence of Lipschitz functions $f_n$ such that $f_n\to f$ in
$L^2(X,\mm)$ and
\begin{equation}\label{eq:june7}
\limsup_{n\to\infty}\int_X|\rmD  f_n|^2\,\d\mm\leq
\int_XG^2\,\d\mm.
\end{equation}
It is well known, see \cite{Cheeger00}, that this approximation is
possible (even in strong $W^{1,2}$ norm) if Poincar\'e and doubling
hold with upper gradients in the right hand side.}

A byproduct of our identification result, see
Theorem~\ref{thm:rel=weak}  {\nc and Theorem~\ref{thm:density_energy}}, is the fact that
$\underline{\sf Ch}_*=\C$, i.e. that the approximation
\eqref{eq:june7} with Lipschitz functions and their corresponding
slopes instead of upper gradients is possible, without \emph{any}
regularity assumption on $(X,\sfd,\mm)$, besides \eqref{eq:75}. 
Also, in the case when $\sfd$ is a distance, taking into account the locality properties
of the weak gradients, the result can be extended to locally
finite measures.
\fr
\end{remark}

\begin{proposition}[Chain rule]\label{prop:chainrule}
If $f\in L^2(X,\mm)$ has a relaxed gradient, the following properties
hold:
\begin{itemize}
\item[(a)] for any $\Leb{1}$-negligible Borel set $N\subset\R$ it holds $\relgrad f=0$ $\mm$-a.e. on
$f^{-1}(N)$;
\item[(b)] $\relgrad f=\relgrad g$ $\mm$-a.e. on $\{f-g=c\}$
  for all constants $c\in\R$ and $g\in L^2(X,\mm)$ with $\C(g)<\infty$;
\item[(c)] $\phi(f)\in D(\C)$ and $\relgrad {\phi(f)}\leq |\phi'(f)|\relgrad f$ for any
Lipschitz function $\phi$ on  an interval $J$ containing the image
of $f$ (with $0\in J$ and $\phi(0)=0$ if $\mm$ is not finite);
\item[(d)] $\phi(f)\in D(\C)$ and $\relgrad {\phi(f)}= \phi'(f)\relgrad f$ for
any nondecreasing and Lipschitz function $\phi$ on an interval $J$
containing the image of $f$ (with $0\in J$ and $\phi(0)=0$ if $\mm$
is not finite); {\GGG
\item[(e)] If $f,\,g\in D(\C)$ and $\phi:\R\to\R$ is 
  a nondecreasing contraction (with $\phi(0)=0$ if
  $\mm(X)=\infty)$, then
  \begin{equation}
    \label{eq:111}
    \relgrad{(f+\phi(g-f))}^2+
    \relgrad{(g-\phi(g-f))}^2\le \relgrad f^2+\relgrad g^2\quad
    \text{$\mm$-a.e.\ in $X$}.
  \end{equation}
}
\end{itemize}
\end{proposition}
\begin{proof}
{\bf (a)} We claim that for $\phi:\R\to\R$ continuously
differentiable and Lipschitz on the image of $f$ it holds
\begin{equation}
\label{eq:chainc1} \relgrad {\phi(f)}\leq |\phi'\circ f|\relgrad
f,\qquad\text{$\mm$-a.e. in $X$},
\end{equation}
for any $f\in D(\C)$. To prove this, observe that the pointwise
inequality $|\rmD \phi(f)|\leq |\phi'\circ f||\rmD  f|$ trivially
holds for $f\in L^2(X,\mm)$ Borel and $\sfd$-Lipschitz. The claim
follows by an easy approximation argument,
 thanks to \eqref{eq:16} of Lemma~\ref{lem:strongchee};
when $\mm$ is not finite, we also require $\phi(0)=0$ in order to be
sure that $\phi\circ f\in L^2(X,\mm)$.

Now, assume that $N$ is compact. In this case, let $A_n\subset\R$ be
open sets such that $A_n\downarrow N$ {\nc and $\Leb{1}(A_1)<\infty$}. Also, let $\psi_n:\R\to[0,1]$
be a continuous function satisfying $\nchi_{N}\leq\psi_n\leq
\nchi_{A_n}$, and define $\phi_n:\R\to\R$ by
\[
\left\{\begin{array}{ll}
\phi_n(0)&=0,\\
\phi_n'(z)&=1-\psi_n(z).
\end{array}
\right.
\]
The sequence $(\phi_n)$ uniformly converges to the identity map, and
each $\phi_n$ is 1-Lipschitz and $C^1$. Therefore $\phi_n\circ f$
converge to $f$ in $L^2$. Taking into account that $\phi_{n}'=0$ on $N$
and \eqref{eq:chainc1} we deduce
\[
\begin{split}
\int_X \relgrad f^2\,\d\mm
&\leq\liminf_{n\to\infty}\int_X\relgrad {\phi_n(f)}^2\,\d\mm
\leq\liminf_{n\to\infty}\int_X|\phi'_n\circ f|^2\relgrad f^2\,\d\mm\\
&=\liminf_{n\to\infty}\int_{X\setminus f^{-1}(N)}|\phi'_n\circ
f|^2\relgrad f^2\,\d\mm\leq\int_{X\setminus f^{-1}(N)}\relgrad
f^2\,\d\mm.
\end{split}
\]
It remains to deal with the case when $N$ is not compact. In this
case we consider the finite measure $\mu:=f_\sharp\tmm$,
where $\tmm=\vartheta\mm$ is the finite measure defined as in \eqref{eq:17}.  Then there exists
an increasing sequence $(K_n)$ of compact subsets of $N$ such that
$\mu(K_n)\uparrow\mu(N)$. By the result for the compact case we know
that $\relgrad f=0$ $\mm$-a.e.\ on $\cup_nf^{-1}(K_n)$, and by
definition
 of push forward and the fact that $\tmm$ and $\mm$ have the
same negligible subsets, we know that $\mm(f^{-1}(N\setminus
\cup_nK_n))=0$.
\\* {\bf (b)} By
(a) the claimed property is true if $g$ is identically 0. In the
general case we notice that $\relgrad {(f-g)}+\relgrad  g$ is a
relaxed gradient of $f$, hence on $\{f-g=c\}$ we conclude that
$\mm$-a.e. it holds $\relgrad f\leq \relgrad g$. Reversing the roles
of $f$ and $g$ we conclude.\\* {\bf (c)} By (a) and Rademacher
Theorem we know that the right hand side is well defined, so that
the statement makes sense (with the convention to define
$|\phi'\circ f|$ arbitrarily at points $x$ such that $\phi'$ does
not exist at $f(x)$). Also, by \eqref{eq:chainc1} we know that the
thesis is true if $\phi$ is $C^1$. For the general case, just
approximate $\phi$ with a sequence $(\phi_n)$ of equi-Lipschitz and
$C^1$ functions, such that $\phi_n'\to \phi'$ a.e. on the image of
$f$.
\\* {\bf (d)} Arguing as in (c)
we see that it is sufficient to prove the claim under the further
assumption that $\phi$ is $C^1$, thus we assume this regularity.
Also, with no loss of generality we can assume that $0\leq\phi'\leq
1$. We know that $(1-\phi'(f)){\relgrad f}$ and $\phi'(f)\relgrad f$
are relaxed gradients of $f-\phi(f)$ and $f$ respectively. Since
$$
\relgrad f\leq \relgrad {(f-\phi(f))}+\relgrad {\phi(f)}\leq
\Big( (1-\phi'(f))+\phi'(f)\Big)\relgrad f=\relgrad f
$$
it follows that all inequalities are equalities $\mm$-a.e. in $X$.
\\*{\GGG 
{\bf (e)} Applying Lemma \ref{lem:strongchee} we find
two optimal sequences $(f_n),\,(g_n)$ of bounded
Lipschitz functions satisfying \eqref{eq:16}
(w.r.t.~$f$ and $g$ respectively).
When $\phi$ is of class $C^1$, 
passing to the limit in the inequality
\eqref{eq:107} of Lemma \ref{le:contraction} written for
$f_n$ and $g_n$ we easily get \eqref{eq:111}.
In the general case, we first approximate $\phi$ by a sequence
$\phi_n$ of nondecreasing contraction of class $C^1$ converging to
$\phi$ pointwise (and satisfying the condition $\phi_n(0)=0$ when
$\mm(X)=\infty$) and then pass to the limit in \eqref{eq:111} 
written for $\phi_n$.}
\end{proof}

{\nc Taking the locality property of Proposition~\ref{prop:chainrule}} into account, we can extend the relaxed gradient
from $L^2(X,\mm)$ to the class of $\mm$-measurable maps $f$ whose
truncates $f_N:=\min\{N,\max\{f,-N\}\}$ belong to $D(\C)\subset
L^2(X,\mm)$ for any integer $N$ in the following way:
\begin{equation}\label{eq:extrelgrad}
\relgrad f:=\relgrad{f_N}\qquad\text{$\mm$-a.e. on $\{|f|<N\}$.}
\end{equation}
Accordingly, we can extend Cheeger's functional \eqref{def:Cheeger}
as follows:
\begin{equation}\label{eq:extchee}
{\tilde{\sf Ch}}_*(f):=
\begin{cases}
\frac{1}{2}\int_X \relgrad f^2\,\d\mm& \text{if $f_N\in D(\C)$ for
all $N\geq 1$}
\\ +\infty&\text{otherwise.}
\end{cases}
\end{equation}
It is obvious that ${\tilde{\sf Ch}}_*$ is convex and, when
$\mm(X)<\infty$, it is sequentially lower semicontinuous with
respect to convergence $\mm$-a.e. in $X$: we shall see that this
property holds even when $\mm$ is not finite but satisfies
\eqref{eq:75}. We shall use this extension when we will compare
relaxed and weak upper gradient, see Theorem~\ref{thm:rel=weak}.

Here it is useful to introduce the \emph{Fisher information
functional}:

\begin{definition}[Fisher information]\label{def:Fisher}
We define the Fisher information $\mathsf{F}(f)$ of a Borel function
$f:X\to [0,\infty)$ as
  \begin{equation}
    \label{eq:91}
    \mathsf{F}(f):= 4\int_X \relgrad{\sqrt f}^2\,\d\mm=8\,
    \C(\sqrt f),
  \end{equation}
if $\sqrt{f}\in D(\C)$ and we define $\mathsf{F}(f)=+\infty$
otherwise.
\end{definition}

\begin{lemma}[Properties of $\mathsf F$]
  \label{le:Fisher}
  For every Borel function $f:X\to [0,\infty)$ we have the
  equivalence
  \begin{equation}
    \label{eq:90}
    f\in D(\mathsf F)\quad\Longleftrightarrow
    \quad
    f,\,\relgrad f\in L^1(X,\mm),\quad
    \int_{\{f>0\}}\frac{\relgrad f^2}f\,\d\mm<\infty,
  \end{equation}
  and in this case it holds
  \begin{equation}
    \label{eq:42}
    \mathsf F(f)
    =\int_{\{f>0\}}\frac{\relgrad f^2}f\,\d\mm.
  \end{equation}
  In addition, the functional $\mathsf F$ is convex and sequentially lower semicontinuous with
  respect to the weak topology of $L^1(X,\mm)$.
\end{lemma}
\begin{proof}
  By the definition of extended relaxed gradient it is sufficient to consider the case
  when $f$ is bounded.  The right implication in \eqref{eq:90} is an
  immediate consequence of Proposition \ref{prop:chainrule} with
  $\phi(r)=r^2$.  The  reverse
  one still follows by applying the same
  property to $\phi_\eps(r)=\sqrt{r+\eps}-\sqrt\eps$, $\eps>0$, and then passing
  to the limit as $\eps\downarrow 0$.

  The strong lower semicontinuity in $L^1(X,\mm)$ is an immediate consequence
  of the lower semicontinuity of the Cheeger's energy in
  $L^2(X,\mm)$. The convexity of $\mathsf F$ follows
  by the representation of $\mathsf F$ given in \eqref{eq:42},
  the convexity of $g\mapsto\relgrad{g}$ stated in \eqref{eq:82},
  and the convexity of the function $(x,y)\mapsto y^2/x$ in
  $(0,\infty)\times\R$. Since $\mathsf F$ is convex, its weak lower
  semicontinuity in $L^1(X,\mm)$ is a consequence of the strong one.
\end{proof}

We conclude this section with a result concerning general
multiplicative perturbations of the measure $\mm$. Notice that the
choice $\theta=\rme^{-V^2}$ with $V$ as in \eqref{eq:75} implies
\eqref{eq:15} for arbitrary $r>0$.

\begin{lemma}[Invariance with respect to multiplicative perturbations
  of $\mm$]
  \label{le:invariance}
  Let $\mm'=\theta\,\mm$ be another $\sigma$-finite Borel
  measure whose density $\theta$ satisfies the following condition:
  for every $K$ compact in $X$ there exist $r>0$ and positive constants
  $c(K),\,C(K)$ such that
  \begin{equation}
    \label{eq:15}
    0<c(K)\leq\theta\leq C(K)<\infty\quad
  \text{$\mm$-a.e.\ on $K(r):=\{x\in X:\sfd(x,K)\le r\}$}.
\end{equation}
  Then the
  relaxed gradient $\relgrad{f}'$ induced by $\mm'$ coincides $\mm$-a.e.
  with $\relgrad f$ for every $f\in W_*^{1,2}(X,\sfd,\mm)\cap
  W_*^{1,2}(X,\sfd,\mm')$. If moreover
  there exists $r>0$ such that \eqref{eq:15} holds
  for every compact set $K\subset X$ then
  \begin{equation}
    \label{eq:56}
    f\in W_*^{1,2}(X,\sfd,\mm),\quad
    f,\relgrad f\in L^2(X,\mm')\quad\Longrightarrow\quad
    f\in W_*^{1,2}(X,\sfd,\mm').
  \end{equation}
\end{lemma}
\begin{proof} Let us first notice that the role of $\mm$ and $\mm'$ in
  \eqref{eq:15} can be inverted, since also $\mm$ is absolutely continuous
  w.r.t.\ $\mm'$ (\eqref{eq:15} yields $\mm(K)=0$ for every compact
  set $K$ with $\mm'(K)=0$) and therefore its density $\d\mm/\d\mm'=\theta^{-1}$
  w.r.t.\ $\mm'$ still satisfies \eqref{eq:15}.

  Let us prove that $\relgrad f\le \relgrad f'$:
  we argue by contradiction and we suppose that for
  some $f\in W_*^{1,2}(X,\sfd,\mm)\cap W_*^{1,2}(X,\sfd,\mm')$ the strict inequality
  $\relgrad f>\relgrad f'$ holds in a Borel set $B$ with $\mm'(B)>0$. {\nc Since
  $\mm'$ is $\sigma$-finite we can assume $\mm'(B)<\infty$.

  By the finiteness of $\chi_B\mm'$} we can find a compact set $K\subset B$
  with $\mm'(K)>0$ (and therefore $\mm(K)>0$ by \eqref{eq:15})
  and $r>0$ such that \eqref{eq:15} holds.
  Introducing a Lipschitz {\nc nonincreasing} function $\phi_r:\R\to [0,1]$ such that
  $\phi_r(v)\equiv 1$ in $[0,r/3]$ and $\phi_r(v)\equiv 0$ in
  $[2r/3,\infty)$,
  we consider the corresponding functions
  $\nchi_r(x):=\phi_r(\sfd(x,K))$,
  which are {\nc upper} semicontinuous, $\sfd$-Lipschitz, and satisfy
  $\nchi_r(x)=|\rmD \nchi_r(x)|=0$ for every $x$ with $\sfd(x,K)>2r/3$.

  Applying Lemma~\ref{lem:strongchee} we find
  a sequence of Borel and $\sfd$-Lipschitz function $f_n\in
  L^2(X,\mm)$ satisfying \eqref{eq:16}.
  It is easy to check that $f_n':=\nchi_r\,f_n$ is a sequence of Borel
  $\sfd$-Lipschitz functions which converges strongly to
  $f':=\nchi_r\, f$ in $L^2(X,\mm')$ by \eqref{eq:15}.
  Moreover, since
  \begin{equation}
    \label{eq:57}
    |\rmD  f_n'|\le \nchi_r |\rmD  f_n|+|f_n|
    \Lip(\nchi_r)\qquad\text{and}\qquad
    |\rmD  f_n'|\equiv 0\quad\text{on the open set }X\setminus
    K(r),
  \end{equation}
  $|\rmD  f_n'|$ is clearly uniformly bounded in $L^2(X,\mm')$ by \eqref{eq:15}, so
  that up to subsequence, it weakly converges to some function {\nc $G'\ge \relgrad {f'}'$.}
  Since $|\rmD  f_n'|=|\rmD  f_n|$ in a $\sfd$-open set containing $K$,
  \eqref{eq:16} yields $G'=\relgrad f$ $\mm'$-a.e.\ in $K$ so that
  $\relgrad {f'}'\le \relgrad f$ $\mm'$-a.e.\ in $K$. {\nc Locality then gives
  $\relgrad {f}'\le \relgrad f$ $\mm'$-a.e.\ in $K$.}
  Inverting the role of $\mm$ and $\mm'$, we can also prove the converse
  inequality $\relgrad f'\le \relgrad f$.

  In order to prove \eqref{eq:56}, let $K_n$ be an sequence of compact
  sets such that $\nchi_{K_n}\up1$
  as $n\to\infty$ $\mm$-a.e.\ in $X$ (recall that the finite measure
  $\tmm=\vartheta\mm$ defined by \eqref{eq:17} is tight); by the previous argument
  and \eqref{eq:15} (which now, by assumption, holds uniformly with respect to $K_n$)
  we find a sequence $\nchi_n(x):=\phi_r(\sfd(x,K_n))$ uniformly
  $\sfd$-Lipschitz such that
  $\nchi_n f\in W_*^{1,2}(X,\sfd,\mm')$. Since $\nchi_n f$ converges strongly
  to $f$ in $L^2(X,\mm')$ and \eqref{eq:57} yields $\relgrad{(\nchi_n f)}\le
  \relgrad{f}+\frac 3r|f|$, we deduce that
  $\relgrad{(\nchi_n f)}=\relgrad{(\nchi_n f)}'$ is uniformly bounded in $L^2(X,\mm')$;
  applying (b) of Lemma~\ref{lem:strongchee} we conclude.
\end{proof}

\begin{remark}\label{re:puoesserebanale}{\rm Although the content of this
section makes sense in a general metric measure space,
it should be remarked that if no additional assumption is made
it may happen that the constructions presented here are trivial.\\
Consider for instance the case of the interval $[0,1]\subset\R$
endowed with the Euclidean distance and a probability measure $\mm$
concentrated on the set $\{q_n\}_{n\in\N}$ of rational points in
$(0,1)$. For every $n\ge 1$ we consider an open set
$A_n\supset\Q\cap (0,1)$ with Lebesgue measure less than $1/n$ and
the $1$-Lipschitz function $j_n(x)=\Leb{1}(([0,x]\setminus A_n)$,
locally constant in $A_n$. If $f$ is any $L$-Lipschitz function in
$[0,1]$, then $f_n(x):=f(j_n(x))$ is still $L$-Lipschitz and
satisfies
\begin{displaymath}
  \int_{[0,1]}|\rmD  f_n|^2(x)\,\d\mm(x)=0.
\end{displaymath}
Since $j_n(x)\to x$, $f_n\to f$ strongly in $L^2([0,1];\mm)$ as
$n\to\infty$ and we obtain that $\C(f)=0$. Hence Cheeger's
functional is identically 0 and the corresponding gradient flows
that we shall study in the sequel are simply the constant
curves.\\

Another simple example is $X=[0,1]$ endowed with the Lebesgue
measure $\mm$ and the distance $\sfd(x,y):=|y-x|^{1/2}$. It is easy
to check that $|\rmD  f|(x)\equiv0$ for every $f\in C^1([0,1])$
(which is in particular $\sfd$-Lipschitz), so that a standard
approximation argument yields $\C(f)=0$ for every $f\in
L^2([0,1];\mm)$.}\fr
\end{remark}

\begin{subsection}{Laplacian and $L^2$ gradient flow of Cheeger's
energy}\label{sec:cheegergflow}

In this subsection we assume, besides $\sigma$-finiteness, that the
measure $\mm$ satisfies the condition in \eqref{eq:19} (weaker than
\eqref{eq:75}), so that the domain of $\C$ is dense in $L^2(X,\mm)$
by Proposition~\ref{prop:densitydlip}.

The Hilbertian theory of gradient flows (see for instance
\cite{Brezis73}, \cite{Ambrosio-Gigli-Savare08}) can be applied to
Cheeger's functional \eqref{def:Cheeger} to provide, for all $f_0\in
L^2(X,\mm)$, a locally Lipschitz map $t\mapsto f_t= \sfH_t(f_0)$
from $(0,\infty)$ to $L^2(X,\mm)$, with $f_t\to f_0$ as $t\downarrow
0$, whose derivative satisfies
\begin{equation}\label{eq:ODE}
\frac{\d}{\d t}f_t\in -\partial^-\C(f_t)\qquad\text{for a.e. $t\in
(0,\infty)$.}
\end{equation}

Recall that the subdifferential $\partial^-\C$ of convex analysis is
the multivalued operator in $L^2(X,\mm)$ defined at all $f\in D(\C)$ by
the family of inequalities
\begin{equation}
  \label{eq:18}
  \ell\in \partial^-\C(f)\quad\Longleftrightarrow\quad
  \int_X \ell(g-f)\,\d\mm\le \C(g)-\C(f)\quad\text{for every }g\in L^2(X,\mm).
\end{equation}
The map $\Heat t:f_0\mapsto f_t$ is uniquely determined by
\eqref{eq:ODE} and defines a semigroup of contractions in
$L^2(X,\mm)$. Furthermore, we have the regularization estimate
\begin{equation}\label{eq:regularizing}
\C(f_t)\leq\inf \left\{\C(g)+\frac{1}{2t}\int_X|g-f_0|^2\,\d\mm:\
g\in W_*^{1,2}(X,\sfd,\mm)\right\}.
\end{equation}
Another important regularizing effect of gradient flows lies in the
fact that, for every $t>0$, the right derivative $\tfrac{\d^+}{\d t}
f_t$ exists and it is actually the element with minimal $L^2(X,\mm)$
norm in $\partial^-\C(f_t)$. This motivates the next definition:

\begin{definition}[$(\sfd,\mm)$-Laplacian]
The Laplacian $\Deltam f$ of $f\in L^2(X,\mm)$ is defined for those
$f$ such that $\partial^-\C(f)\neq\emptyset$. For those $f$,
$-\Deltam f$ is the element of minimal $L^2(X,\mm)$ norm in
$\partial^-\C(f)$.
\end{definition}

The domain of $\Deltam$ will be denoted by $D(\Deltam)$ 
{\GGG and it is a dense subset of $D(\C)$ (in particular, it is also
  dense in $L^2(X,\mm)$), see for instance \cite[Prop.~2.11]{Brezis73}}. There
is no risk of confusion with the notation \eqref{eq:domaineffective}
introduced for extended real valued maps; in this connection, notice
that convexity and lower semicontinuity of $\C$ ensure the identity
$D(\Deltam)=D(|\rmD^-\C|)$, see
\cite[Proposition~1.4.4]{Ambrosio-Gigli-Savare08}. We can now write
$$
\frac{\d^+}{\d t}f_t=\Deltam f_t\qquad\text{for every $t\in
(0,\infty)$}
$$
for gradient flows $f_t$ of $\C$, in agreement with the classical
case. However, not all classical properties remain valid, as
illustrated in the next remark.

\begin{remark}[Potential lack of linearity]\label{re:laplnonlin}{\rm
It should be observed that, in general, the Laplacian we just
defined is \emph{not} a linear operator: the potential lack of
linearity is strictly related to the fact that the space
$W_*^{1,2}(X,\sfd,\mm)$ needs not be Hilbert, see also
Remark~\ref{re:sobolev}. However, the Laplacian (and the
corresponding gradient flow $\sfH_t$) is always $1$-homogeneous,
namely
$$\Deltam (\lambda f)=\lambda \Deltam f,
\quad\sfH_t(\lambda f)=\lambda\sfH_t(f)\quad
\text{for all $\lambda\in\R$.}$$ This is indeed a property true
for the subdifferential of any $2$-homogeneous, {\nc convex and
lower semicontinuous} functional $\Phi$; to
prove it, if $\lambda\neq 0$ (the case $\lambda=0$ being trivial)
and $\xi\in\partial^-\Phi(x)$ it suffices to multiply the
subdifferential inequality $\Phi(\lambda^{-1}
y)\geq\Phi(x)+\langle\xi,\lambda^{-1} y-x\rangle$ by $\lambda^2$ to
get $\lambda\xi\in\partial^-\Phi(\lambda x)$.

{\GGG When $\mm(X)<\infty$ the invariance property $\C(f+c)=\C(f)$
  for every $c\in \R$ also yields
  $$\Deltam (f+c)=\Deltam f,
  \quad\sfH_t(f+c)=\sfH_t(f)+c\quad\text{for all $c\in\R$.}\qquad\blacksquare$$}}
\end{remark}

\begin{proposition}[Some properties of the Laplacian]
\label{prop:deltaineq} For all $f\in D(\Deltam)$, $g\in D(\C)$ it
holds
\begin{equation}
\label{eq:delta1}
-\int_X g\Deltam f\,\d\mm
  \leq \int_X \relgrad
g\relgrad f\,\d\mm.
\end{equation}
Also, let $f\in D(\Deltam)$ and $\phi:J\to\R$ Lipschitz, with $J$
closed interval containing the image of $f$ ($\phi(0)=0$ if
$\mm(X)=\infty$). Then
\begin{equation}
\label{eq:delta2} -\int_X\phi(f)\Deltam f\,\d\mm=\int_X \relgrad
f^2\phi'(f)\,\d\mm.
\end{equation}
{\GGG
Finally, for every $f,\,g\in D(\Deltam)$ and for every
Lipschitz nondecreasing map  $\phi:\R\to\R$ with $\phi(0)=0$,
we have
\begin{equation}
  \label{eq:112}
  \int_X \big(\Deltam g-\Deltam f\big)\phi(g-f)\,\d\mm\le0.
\end{equation}
}
\end{proposition}
\begin{proof}
Since $-\Deltam f\in\partial^-\C(f)$ it holds
\[
\C(f)-\int_X \eps g\Deltam f\,\d\mm\leq \C(f+\eps g)\qquad\forall
\eps\in\R.
\]
For $\eps>0$, $\relgrad f+\eps \relgrad g$ is a relaxed gradient of
$f+\eps g$. Thus it holds $2\C(f+\eps g)\leq\int_X(\relgrad
f+\eps\relgrad g)^2\,\d\mm$ and therefore
\[
-\int_X\eps g\Deltam f\leq \frac12\int_X\Big((\relgrad f+\eps\relgrad g)^2-
\relgrad f^2\Big)\,\d\mm=\eps\int_X\relgrad f\relgrad g\,\d\mm+o(\eps).
\]
Dividing by $\eps$, letting $\eps\downarrow 0$ we get
\eqref{eq:delta1}.

For the second part we recall that, by the chain rule,
$\relgrad {(f+\eps \phi(f))}=(1+\eps \phi'(f))\relgrad f$ for $|\eps|$ small
enough. Hence
\[
\C(f+\eps \phi(f))-\C(f)= \frac{1}{2}\int_X\relgrad f^2\Big((1+\eps
\phi'(f))^2-1\Big)\,\d\mm=\eps\int_X\relgrad f^2 \phi'(f)\,\d\mm+o(\eps),
\]
which implies that for any $v\in \partial^-\C(f)$ it holds $\int_Xv
\phi(f)\,\d\mm=\int_X\relgrad f^2\phi'(f)\,\d\mm$, and gives the thesis
with $v=-\Deltam f$.

{\GGG Concerning \eqref{eq:112}, we set $h=\phi(g-f)$
and recal that $h\in D(\C)$, so that \eqref{eq:18} yields
for every $\eps>0$
\begin{align*}
  -\eps\int_X \big(\Deltam f-\Deltam g\big)h\,\d\mm
  &= 
  -\eps\int_X \Deltam f\,h\,\d\mm-\eps\int_X \Deltam g\,(-h)\,\d\mm
  \\&\le 
  \C(f+\eps h)-\C(f)+\C(g-\eps h)-\C(g).
\end{align*}
Choosing $\eps>0$ so small that $\eps\phi$ is a contraction, 
we conclude thanks to \eqref{eq:111}.
}
\end{proof}

\begin{theorem}[Comparison principle, convex entropies and contraction]
  \label{prop:maxprin}
Let $f_t=\sfH_t(f_0),\ g_t=\sfH_t(g_0)$ 
be the gradient flows of $\C$ starting from $f_0,g_0\in
L^2(X,\mm)$ respectively.
\begin{enumerate}[(a)]
\item
  \label{maxprin:a}
  Assume that $f_0\leq C$ (resp. $f_0\geq c$).
  Then $f_t\leq C$ (resp.
  $f_t\geq c$) for every $t\geq 0$.
  Similarly, 
  {\GGG if 
  $f_0\leq g_0+C$ 
  for some constant $C\in \R$, then $f_t\leq g_t+C$.} 
\item
  \label{maxprin:c}
  If $e:\R\to[0,\infty]$ is a
  convex lower semicontinuous function and
  $E(f):=\int_X e(f)\,\d\mm$ is the associated convex and lower semicontinuous
  functional in $L^2(X,\mm)$ it holds
  \begin{equation}
    \label{eq:23}
    E(f_t)\le E(f_0)\quad\text{for every }t\ge0,
  \end{equation}
  {\GGG
    and
  \begin{equation}
    \label{eq:crandall-tartar} E(f_t-g_t)\le E(f_0-g_0)
    \quad\text{for every }t\ge0.
  \end{equation}
  In particular, for every $p\in [1,\infty]$, if $f_0\in L^p(X,\mm)$, then also
  $f_t\in L^p(X,\mm)$ and the semigroup $\sfH_t:L^2(X,\mm)\to L^2(X,\mm)$
  satisfies the contraction property
  \begin{equation}
    \label{eq:85}
    \|\sfH_t(f_0)-\sfH_t(g_0)\|_{L^p(X,\mm)}\le
    \|f_0-g_0\|_{L^p(X,\mm)}\quad \forall\, f_0,\,g_0\in
    L^2(X,\mm)\cap L^p(X,\mm).
  \end{equation}
  }
\item
  \label{maxprin:b}
  If $e'$ is locally Lipschitz in $\R$
  and $E(f_0)<\infty$, then we have
  \begin{equation}
    \label{eq:25}
    E(f_t)+\int_0^t\int_X e''(f_s)\relgrad{f_s}^2\,\d\mm\,\d s=E(f_0)\qquad\forall t\geq 0.
  \end{equation}
\item
  \label{maxprin:d}
 When $\mm(X)<\infty$ we have
  \begin{equation}
    \label{eq:24}
    \int_X f_t\,\d\mm=\int_X f_0\,\d\mm\quad\text{for every }t\ge0.
  \end{equation}
\end{enumerate}
\end{theorem}
\begin{proof}
  {\GGG Notice that (\ref{maxprin:a}) is a particular case
    of (\ref{maxprin:c}), simply by choosing $e(r):=\max\{r-C,0\}$
    (or, respectively, $e(r):=\max\{c-r,0\}$), $r\in \R$.}

{\GGG Concerning (\ref{maxprin:c}), 
  \eqref{eq:23} corresponds to \eqref{eq:crandall-tartar} when
$g_0\equiv0$.

In order to prove \eqref{eq:crandall-tartar} 
let us assume first that $e'$ is bounded 
and globally Lipschitz in $\R$ (with $e(0)=0$ if $\mm(X)=\infty$).
Under this assumption, we know that since the maps $t\mapsto f_t$ 
and $t\mapsto g_t$ are locally
Lipschitz in $(0,\infty)$ with values in $L^2(X,\mm)$, the same is
true for the map $t\mapsto e(f_t-g_t)$ so that 
\begin{equation}
  \label{eq:113}
  \frac{\d}{\d t
  }e(f_t-g_t)=e'(f_t-g_t)\frac{\d}{\d t}(f_t-g_t)=e'(f_t-g_t)(\Deltam
  f_t-\Deltam g_t).
\end{equation}
Thus the map
$t\mapsto E(f_t-g_t)$ is locally Lipschitz in $L^2(X,\mm)$, so that using
\eqref{eq:112} of Proposition~\ref{prop:deltaineq} 
we obtain \eqref{eq:crandall-tartar}.

A standard approximation, by first truncating $e'$ and then replacing
$e(r)$ 
by its Yosida approximation,
yields the same result for general $e$.}

{\GGG In order to prove (\ref{maxprin:b}),
as before {\nc we assume first that $e'$ is bounded 
and globally Lipschitz in $\R$.}}
notice that
by (a) we know that the image of $f_t$ is contained in the same
interval containing the image of $f_0$. We can also assume, {\nc by
truncation of $e'$}, that this interval is closed, bounded, and that $e'$ is
Lipschitz in it (the interval also contains $0$ if $\mm(X)=\infty$
Under this assumption, 
{\GGG \eqref{eq:113} with $g_t=g_0=0$ 
and}
\eqref{eq:delta2} with $\phi=e'$ yield
\begin{equation}
\frac \d{\d t}\int_X e(f_t)\,\d\mm=\int_X e'(f_t)\Deltam
f_t\d\,\mm=-\int_X e''(f_t)\relgrad {f_t}^2\,\d\mm.\label{eq:36}
\end{equation}
{\nc In the general case, if $t_0$ is a minimum point of $e$ ($t_0=0$ in the case $\mm(X)=\infty$) 
we monotonically approximate $e$ from below by the convex functions $e_k$ defined by
$e_k(t)=e(t_0)+\int_{t_0}^t w_k(s)\,\d s$, with
$w_k=\min\{N,\max\{e',-N\}\}$.}

In order to prove (\ref{maxprin:d}) we notice that $\mm(X)<\infty$ allows to the
choice of $g=\pm 1$ in \eqref{eq:delta1}, to obtain $\int_X\Deltam
h\,\d\mm=0$ for all $h\in D(\Deltam)$. Hence \eqref{eq:24} follows
by
\begin{displaymath}
  \frac\d{\dt}\int_X f_t\,\d\mm=\int_X \Deltam f_t\,\d\mm=0.
\end{displaymath}
\end{proof}
\end{subsection}

\subsection{Increasing family of measures and variational
  approximation of Cheeger's energy}
\label{sec:varapp}
In this section we study a monotone approximation scheme for the Cheeger's energy
and its
gradient flow, which turns to be quite useful  when
$\mm(X)=\infty$ and one is interested to extend the validity
of suitable estimates,
which can be more easily obtained in the case of measures with finite
total mass.

Let us consider an increasing sequence of $\sigma$-finite, Borel measures
$\mm^0\le \mm^1\le \cdots\le \mm^k\le \mm^{k+1}\le \cdots$
converging to the limit measure $\mm$ in the sense that
\begin{equation}
  \label{eq:58}
  \lim_{k\to\infty}\mm^k(B)=\mm(B)\quad\text{for every }B\in\BorelSets X.
\end{equation}
Let us assume that, as in \eqref{eq:15}, $\mm\ll\mm^0$ with density
$\displaystyle\theta=\frac{\d\mm}{\d\mm^0}$ satisfying
\begin{equation}
  0<c(K)\leq\theta\leq C(K)<\infty\quad
  \text{$\mm^0$-a.e.\ on $K(r):=\{x\in X:\sfd(x,K)\le r\}$}
\label{eq:33}
\end{equation}
for any compact set $K\subset X$, with {\nc $r>0$ independent of $K$}. Notice that the
measures $\mm^k$ share the same collection of negligible sets and of
measurable functions. We denote by $\mathcal H^k:=L^2(X,\mm^k)$ and
by $\C^k$ the Cheeger's energy associated to $\mm^k$ in
$W_*^{1,2}(X,\sfd,\mm^k)\subset \mathcal H^k$, extended to $+\infty$ in
$\mathcal H^0\setminus W_*^{1,2}(X,\sfd,\mm^k)$. We have $\mathcal
H^{k+1}\subset \mathcal H^k\subset \mathcal H^0$, with continuous
inclusion and, by Lemma~\ref{le:invariance}, $\C^k\le \C^{k+1}$.
\begin{proposition}[$\Gamma$-convergence]
  \label{prop:Gamma}
  Let $(\mm^k)$ be an increasing sequence of $\sigma$-finite measures
  satisfying \eqref{eq:58} and \eqref{eq:33}.
  If $f^k\in \mathcal H^k$ weakly converge in $\mathcal H^0$ to $f$ with
  $S:=\limsup_k\int_X |f^k|^2\,\d\mm^k<\infty$
  then $f\in L^2(X,\mm)$,
  \begin{equation}
    \label{eq:47}
    \liminf_{k\to\infty}
    \int_X |f^k|^2\,\d\mm^k\ge \int_X
    |f|^2\,\d\mm,\qquad
    \liminf_{k\to\infty}\C^k(f^k)\ge \C(f),
  \end{equation}
  and
  \begin{equation}
    \label{eq:51}
    \lim_{k\to\infty}\int_X f^k\, g\,\d\mm^k= \int_X
    f\,g\,\d\mm\quad\text{for every }g\in L^2(X,\mm).
  \end{equation}
  Finally, if $S\leq\int_X|f|^2\,\d\mm$ then
    \begin{equation}
      \label{eq:52}
      f^k\to f\quad\text{strongly in }\mathcal H^0\quad\text{and}\quad
      \lim_{k\to\infty}\int_X |f^k|^2\,\d\mm^k= \int_X|f|^2\,\d\mm.
    \end{equation}
 \end{proposition}
 \begin{proof}
  \eqref{eq:47} is an easy consequence of the monotonicity of $\mm^k$,
  the lower semicontinuity of the $L^2$-norm with respect to weak
  convergence, and \eqref{eq:56} of Lemma~\ref{le:invariance}.

  In order to check \eqref{eq:51} notice that for every $g\in
  L^2(X,\mm)$
   and every $\upsilon>0$
  \begin{equation}
    \label{eq:49}
    \int_X f^k\, g\,\d\mm^k=
    \frac 12\int_X (\upsilon f^k+\upsilon^{-1} g)^2\,\d\mm^k-
    \frac{\upsilon^2} 2\int_X |f^k|^2\,\d\mm^k-
    \frac 1{2\upsilon^2}\int_X |g|^2\,\d\mm^k,
  \end{equation}
  so that, taking the limit as $k\to\infty$,
  \begin{align*}
    \liminf_{k\to\infty}\int_X f^k\, g\,\d\mm^k&\ge
    \frac 12\int_X (\upsilon f+\upsilon^{-1} g)^2\,\d\mm-
    \frac{\upsilon^2} 2S-
    \frac 1{2\upsilon^2}\int_X |g|^2\,\d\mm
    \\&=
    \int_X fg\,\d\mm+\frac {\upsilon^2}2\Big(\int_X |f|^2\,\d\mm-S\Big).
  \end{align*}
  Passing to the limit as $\upsilon\downarrow0$ and applying the same
  inequality with $-g$ in place of $g$ we get \eqref{eq:51}.
  Finally, \eqref{eq:52} follows easily by \eqref{eq:51} and the inequality
  $S\leq\int_X|f|^2\,\d\mm$, passing to the limit in
  \begin{displaymath}
    \int_X |f^k-f|^2\,\d\mm^k=-2\int_X f^kf\,\d\mm^k+\int_X
    |f^k|^2\,\d\mm^k+\int_X |f|^2\,\d\mm^k.
  \end{displaymath}
  \end{proof}
Let us now consider the gradient flow $\sfH^k_t$ generated by $\C^k$
in $\mathcal H^k$ and the ``limit'' semigroup $\sfH_t$ generated by
$\C$ in $\mathcal H=L^2(X,\mm)\subset \mathcal H^0$. Since any
element $f_0$ of $\mathcal H$ belongs also to $\mathcal H^k$, the
evolution $f^k_t:=\sfH^k_t(f_0)$ is well defined for every $k$ and
it is interesting to prove the convergence of $f^k_t$ to
$f_t=\sfH_t(f_0)$ as $k\to\infty$ in the larger space $\mathcal
H^0$.
\begin{theorem}
  \label{thm:monotoneHeat} {\nc Assume that \eqref{eq:19} holds.}
  Let $f_0\in L^2(X,\mm)\subset\mathcal H^0$ and let
  $f^k_t=\sfH^k_t(f_0)\in \mathcal H^0$ be the heat flow in
  $L^2(X,\mm^k)$, $f_t:=\sfH_t(f_0)\in L^2(X,\mm)$.
  Then for every $t\ge0$ we have
  \begin{equation}
    \label{eq:28}
    \lim_{k\to\infty}f^k_t=f_t\quad\text{strongly in
    }\mathcal H^0,\quad
    \lim_{k\to\infty}\int_X |f_t^k|^2\,\d\mm^k=
    \int_X |f_t|^2\,\d\mm.
  \end{equation}
\end{theorem}
\begin{proof}
  The following classical argument combines the $\Gamma$-convergence result of
  the previous proposition with resolvent
  estimates; the only technical issue here is that the gradient flows
  are settled in
  Hilbert spaces $\mathcal H^k$ also depending on $k$.

  Let us fix $\lambda>0$ and let us consider the family of resolvent operators
  $J^k_\lambda:\mathcal H^k\to
  \mathcal H^k$ which to every $f^k\in \mathcal H^k$ associate the
  unique minimizer $f^k_\lambda$ of
  \begin{equation}
    \label{eq:29}
    \mathcal C^k_\lambda(g;f^k):=\C^k(g)+\frac \lambda2\int_X |g-f^k|^2\,\d\mm^k.
  \end{equation}
  We first prove that if
  $\limsup_k\int_X|f^k|^2\,\d\mm\leq\int_X|f|^2\,\d\mm$
  {\GGG and $f^k\weakto f$ in $\mathcal H^0$}
  then
  $f^k_\lambda:=J^k_\lambda(f^k)$ converge
  to $f_\lambda:= J_\lambda(f)$ as $k\to\infty$ according to
  \eqref{eq:52}.
  \upshape
  In fact we know that for every $g\in W_*^{1,2}(X,\sfd,\mm)$
  \begin{displaymath}
    \C^k(f^k_\lambda)+\frac \lambda2\int_X
    |f^k_\lambda-f^k|^2\,\d\mm^k\le
    \C^k(g)+\frac \lambda2\int_X
    |g-f^k|^2\,\d\mm^k.
  \end{displaymath}
  By the assumption on $f^k$ {\GGG and \eqref{eq:52} of Proposition
    \ref{prop:Gamma},} the right hand side of the previous
  inequality converges to $\C(g)+\frac \lambda 2 \int_X
  |g-f|^2\,\d\mm$.
  Since the sequence $(f^k_\lambda)$ is uniformly bounded in
  $\mathcal H^0=L^2(X,\mm_0)$,
  up to extracting a suitable
  subsequence we can assume that $f^k_\lambda$ weakly converge to
  some limit $\tilde f$ in $\mathcal H^0$; \eqref{eq:47} yields
  \begin{align*}
    \C(\tilde f)+\frac \lambda2\int_X
    |\tilde f-f|^2\,\d\mm
    &\le \liminf_{k\to\infty}\C^k(f^k_\lambda)+\frac \lambda2\int_X
    |f^k_\lambda-f^k|^2\,\d\mm^k\\
    &\le
    \C(g)+\frac \lambda 2 \int_X
    |g-f|^2\,\d\mm=\mathcal C_\lambda(g;f),
  \end{align*}
  for every $g\in W_*^{1,2}(X,\sfd,\mm)$.
  We deduce that
  $\tilde f=J_\lambda f$ is the unique minimizer of $g\mapsto \mathcal
  C_\lambda(g;f)$.
  In particular the whole sequence converge weakly to $J_\lambda f$ in
  $\mathcal H^0$
  and moreover
  \begin{equation}
    \label{eq:31}
    \limsup_{k\to\infty}\int_X |f^k_\lambda-f^k|^2\,\d\mm^k\le
    \int_X |f_\lambda-f|^2\,\d\mm,
  \end{equation}
  so that we can apply Proposition~\ref{prop:Gamma}
  and obtain \eqref{eq:52} for
  the sequence $(f^k_\lambda)$.

  Iterating this resolvent convergence property, we get the same result for
  the operator $(J^k_\lambda)^n$ obtained by $n$ iterated
  compositions of $J^k_\lambda$, for every $n\in\N$.
  By the general estimates for gradient flows, choosing
  $\lambda:=n/t$, we know that
  $$
    \int_X |H_t f_0-(J_{n/t})^nf_0|^2\,\d\mm\le \frac tn \C(f_0),\quad
    \int_X |H^k_t f_0-(J^k_{n/t})^nf_0|^2\,\d\mm^k\le \frac tn \C^k(f_0)\le
    \frac tn\C(f_0).
  $$
  Since for every $n$ and every $t>0$ we have
  $\lim_{k\to\infty}(J^k_{n/t})^n f=(J_{n/t})^nf$ strongly in
  $\mathcal H^0$,
  combining the previous estimates we get the first convergence
  property of \eqref{eq:28} when $\C(f_0)<\infty$.
  Since the domain of $\C$ is dense in $L^2(X,\mm)$ and $\sfH_t$ is a
  contraction
  semigroup, a simple approximation argument yields the general case
  when
  $f_0\in L^2(X,\mm)$.
  Passing to the limit as $k\to\infty$ in the identities
  \begin{equation}
    \label{eq:53}
    \frac 12\int_X |f^k_t|^2\,\d\mm^k+2\int_0^t \C^k(f^k_s)\,\d s=
    \frac 12\int_X |f_0|^2\,\d\mm^k,
  \end{equation}
  and taking into account the corresponding identity for $\mm$ and
  $\C$ and the lower semicontinuity property \eqref{eq:47} for $\C^k$,
  we obtain the second limit of \eqref{eq:28}.
\end{proof}

\subsection{Mass preservation and entropy dissipation when
  $\mm(X)=\infty$}
\label{sec:masspreservation} Let us start by deriving useful
``moment-entropy estimates'', in the case of a measure $\mm$ with
finite mass.
\newcommand{\zlambda}{z}
\begin{lemma}[Moment-entropy estimate]
  \label{le:moment}
  Let $\mm$ be a finite measure, let $\Wgh:X\to[0,\infty)$ be
  a Borel and $\sfd$-Lipschitz function,
  {\GGG let $\zlambda>0$,} 
  let $f_0\in L^2(X,\mm)$ be nonnegative with
  \begin{equation}
    \label{eq:88}
    {\GGG \zlambda} \int_X \rme^{-\Wgh^2}\,\d\mm\le 
    \int_X f_0\,\d\mm,\qquad
    \int_X \Wgh^2\,f_0\,\d\mm<\infty,
  \end{equation}
  and let $f_t=\Heat t(f_0)$ be the solution of \eqref{eq:ODE}.
  Then the map $t\mapsto \int_X \Wgh^2\,f_t\,\d\mm$ is locally absolutely
  continuous in $[0,\infty)$
  and for every $t\ge0$
  \begin{align}
    \label{eq:34}
    \int_X \Wgh^2\,f_t\,\d\mm&\le
    \rme^{4{\rm Lip}^2(\Wgh) t}\int_X f_0\Big(\log f_0+2\Wgh^2
    {\GGG       -\log \zlambda}
    \Big)\,\d\mm,
    \\
    \label{eq:45}
    \int_0^t\int_{\{f_s>0\}} \frac{\relgrad {f_s}^2}{f_s}\,\d\mm\,\d s&\le
    2\rme^{4{\rm Lip}^2(\Wgh) t}\int_X f_0\Big(\log
    f_0+2\Wgh^2
    {\GGG-\log \zlambda}
    \Big)\,\d\mm.
  \end{align}
\end{lemma}
\begin{proof}
  {\GGG By the $1$-homogeneity of $\Heat t$ it is sufficient to
    consider the case $\zlambda=1$.}
  We set $L={\rm Lip}(\Wgh)$ 
  and
  \begin{equation}
    \label{eq:35}
    M^2(t):=\int_X \Wgh^2\,f_t\,\d\mm,\quad
    E(t):=\int_X f_t\log f_t\,\d\mm,\quad
    F^2(t):=\int_{\{f_t>0\}}
    \frac{\relgrad{f_t}^2}{f_t}\,\d\mm.
  \end{equation}
  Applying \eqref{eq:25} to 
  {\GGG $(f_t+\eps)=\Heat t(f_0+\eps)$} and letting $\eps\downarrow
  0$ we get $F\in L^2(0,T)$ for every $T>0$ with
  \begin{equation}
    \label{eq:37}
    \frac \d{\d t}E(t)=-F^2(t)\quad\text{a.e.\ in }(0,T).
  \end{equation}
  The convexity inequality $r\log r\ge r-r_0+ r\log r_0$ with $r=f_t$,
  $r_0=\rme^{-\Wgh^2}$, and the conservation of the total mass
  \eqref{eq:24} and \eqref{eq:88} yield
  for every $t\ge 0$
  $$
    E(t)\ge \int_X (f_t-\rme^{-\Wgh^2})\,\d\mm-M^2(t)=
    \int_X (f_0-\rme^{-\Wgh^2})\,\d\mm-M^2(t)\ge -M^2(t).
  $$
  We introduce now the truncated weight $\Wgh_k(x)=\min(\Wgh(x),k)$
  and the corresponding functional $M_k^2(t)$ defined as in
  \eqref{eq:35}.
  Since the map
  $t\mapsto M_k^2(t)$ is Lipschitz continuous
  we get for a.e.\ $t>0$
  \begin{equation}
    \label{eq:39}
    \Big|\frac \d{\d t}M_k^2(t)\Big|=\Big|
    \int_X \Wgh_k^2\,\Deltam f_t\,\d\mm\Big|\le
    2\int_X \relgrad{f_t}\relgrad{\Wgh_k}\Wgh_k\,\d\mm
    \le 2 L\,F(t)\,M_k(t).
  \end{equation}
  We deduce that
  \begin{displaymath}
    M_k(t)\le M_k(0)+L\,\int_0^t F(s)\,\d s\le
    M(0)+L\,\int_0^t F(s)\,\d s,
  \end{displaymath}
  so that $M_k(t)$ is uniformly bounded. Passing to the limit in
  (an integral form of) \eqref{eq:39}
  as $k\to\infty$ by monotone convergence,
  we obtain the same differential inequality for $M$
$$
    \Big|\frac \d{\d t}M^2(t)\Big|\le 2L\,F(t)\,M(t).
$$
  Combining with \eqref{eq:37} we obtain
$$
    \frac \d{\d t}\big(E+2 M^2\big)+F^2\le 4 L\,F\,M\le
    F^2+4 L^2\,M^2.
$$
  Since $E+2M^2
  \ge M^2$, Gronwall lemma yields
$$
    M^2(t)\le E(t)+2M^2(t)
    \le \big(E(0)+2M^2(0
    \big)\rme^{4L^2t},
$$
  i.e.\ \eqref{eq:34}.
  Integrating now \eqref{eq:37} we get $
    \int_0^t F^2(s)\,\d s\le E(0)-E(t)\le E(0)+M^2(t)$,
  which yields \eqref{eq:45}.
\end{proof}

We want now to extend the validity of \eqref{eq:24}, \eqref{eq:34}
and \eqref{eq:45} to the case when $\mm(X)=\infty$, at least when
\eqref{eq:75} holds. Notice that this assumption also includes the
cases when $\mm(X)\in (0,\infty)$.
\begin{theorem}
  \label{thm:infinitemass}
  If $\mm$ is a $\sigma$-finite measure satisfying \eqref{eq:75},
  then the gradient flow $\sfH_t$ of the Cheeger's energy is mass
  preserving (i.e.\ \eqref{eq:24} holds).
  Moreover, for every nonnegative $f_0\in L^2(X,\mm)$ with
  \begin{equation}
    \label{eq:89}
    \int_X \Wgh^2\,f_0\,\d\mm<\infty,\quad
    \int_X f_0\,\d\mm<\infty
  \end{equation}
  the solution $f_t=\Heat t(f_0)$ of \eqref{eq:ODE} satisfies
  \eqref{eq:34} and \eqref{eq:45} with $\zlambda:=\int_X f_0\,\d\mm$ for every $t\ge0$.
\end{theorem}
\begin{proof} {\nc Since $f_0\in L^2(X,\mm)$ and $\int_X \Wgh^2 f_0\,\d\mm<\infty$, using
\eqref{eq:75} it is easy to check that $f_0|\log f_0|\in L^1(X,\mm)$ (see also Lemma~\ref{lem:changeref} below).}

Let us first prove mass preservation and \eqref{eq:34},
\eqref{eq:45} for a nonnegative initial datum satisfying
\eqref{eq:89}. The proof is based on a simple approximation result.
We set $\mm^0:=\rme^{-\Wgh^2}\mm$, $\Wgh_k:=\min(\Wgh,k)$ and
$\mm^k:=\rme^{\Wgh_k^2}\mm^0= \rme^{\Wgh_k^2-\Wgh^2}\mm$, so that
that $\mm^k$ is an increasing family of finite measures satisfying
conditions \eqref{eq:58}, by monotone convergence. In addition,
since by \eqref{eq:75} $\Wgh$ is $\sfd$-Lipschitz and bounded from
above on compact sets, \eqref{eq:33} holds.

We define $f^k_t=\sfH^k_t(f_0)$ as in the
Theorem~\ref{thm:monotoneHeat} and $z_k:=\int_Xf_0\,\d\mm^k$. We
apply \eqref{eq:24} to obtain that $\int_X f^k_t\,\d\mm^k=z_k$ for
all $t\geq 0$; then, since {\GGG $\int_X \rme^{-V^2}\,\d\mm^k\le 1$,} we can apply the estimates of
Lemma~\ref{le:moment} with $\mm:=\mm^k$ 
to
obtain
\begin{eqnarray}
  \label{eq:54}
  {\GGG \int_X \Wgh^2\,f^k_t\,\d\mm^k
  \le
  \rme^{4
    {\rm Lip}^2(V) t} \int_X f_0(\log f_0+2
  \Wgh^2-\log z_k)\,\d\mm^k.}
\end{eqnarray}
Since, thanks to \eqref{eq:52}, $f^k_t\to f_t$ strongly in $L^2(X,\mm^0)$ as $k\to\infty$, we
get up to subsequences $f^k_t\to f_t$ $\mm$-a.e., so that Fatou's
lemma and the monotonicity of $\mm^k$ yield
\begin{displaymath}
  \int_X \Wgh^2\,f_t\,\d\mm\le
  \liminf_{k\to\infty} \int_X \Wgh^2\,f^k_t\,\d\mm^k,
\end{displaymath}
and \eqref{eq:34} follows by \eqref{eq:54}, {\GGG the monotone
  convergence of $\mm^k$ and the limit $z_k\uparrow z$}.

Let us consider now $A_h:=\{x\in X:\Wgh(x)\le h\}$ and observe that
\eqref{eq:75} and \eqref{eq:51} yield
$$\mm(A_h)\le \int_X \rme^{h^2-V^2}\,\d\mm\le \rme^{h^2}
<\infty,\quad
\int_{A_h} f_t\,\d\mm=
\lim_{k\to\infty}\int_{A_h}f^k_t\,\d\mm^k.
$$
From \eqref{eq:54} we obtain for every $t>0$ a constant $C$
satisfying $h^2 \int_{X\setminus A_h}f^k_t\,\d\mm^k\le C$ for every
$h>0$, so that
\begin{align*}
  \int_X f_t\,\d\mm\ge
  \int_{A_h} f_t\,\d\mm=
  \lim_{k\to\infty}\int_{A_h}f^k_t\,\d\mm^k\ge
  {\GGG z}-\limsup_{k\to\infty}\int_{X\setminus A_h}f^k_t\,\d\mm^k\ge
  1-C/h^2.
\end{align*}
Since $h$ is arbitrary and the integral of $f_t$ does not exceed
${\GGG z}$
by \eqref{eq:23}, we showed that $\int_X f_t\,\d\mm={\GGG z}$. Finally,
\eqref{eq:45} follows now by the lower semicontinuity \eqref{eq:47}
of the Cheeger's energy from the corresponding estimate for $f^k_t$,
recalling \eqref{eq:42}.

Let us now consider an initial datum $f_0\in L^2(X,\mm)$ with
arbitrary sign and vanishing outside some $A_h$, so that $|f_0|$
satisfies \eqref{eq:89} (up to a multiplication for a suitable
constant). The comparison principle yields
$\big|\sfH^k_t(f_0)\big|\le \sfH^k_t(|f_0|)$, so that for every
$t>0$ there exists a constant $C$ such that $h^2\int_{X\setminus
A_h}|f^k_t|\,\d\mm^k\le C$. Since $\int_X f^k_t\,\d\mm^k=\int_X
f_0\,\d\mm^k$ by \eqref{eq:24}, we thus have
\begin{align*}
  \Big|\int_X (f_t-f_0)\,\d\mm\Big|&\le
  \int_{X\setminus A_h} |f^k_t|\,\d\mm^k+\int_{X\setminus A_h}|f_t|\,\d\mm+
  \Big|\int_{A_h} f_t\,\d\mm-\int_{A_h} f^k_t\,\d\mm^k\Big|\\&+
  \int_X |f_0|\,\d(\mm-\mm^k).
\end{align*}
Passing to the limit in the previous inequality first as
$k\to\infty$, taking \eqref{eq:51} into account, and then as
$h\to\infty$ we obtain that the integral of $f_t$ is constant in
time. As in the proof of Theorem~\ref{prop:maxprin}(\ref{maxprin:d}) we can show
that $\sfH_t$ satisfies the contraction estimate \eqref{eq:85} for
$p=1$ and arbitrary couples of initial data vanishing outside $A_h$.
Approximating any $f_0\in L^2(X,\mm)\cap L^1(X,\mm)$ by the sequence
$\nchi_{A_h}f_0$ we can easily extend the contraction property and
the mass conservation to arbitrary initial data.
\end{proof}

\begin{remark}
  \label{rem:grigoryan}
  \upshape
  It is interesting to compare the mass preservation property of
  Theorem \ref{thm:infinitemass} relying on \eqref{eq:75}
  with the well known results
  for the Heat flow on a smooth, complete, finite dimensional, Riemannian
  manifold $(X,\sfd,\mm)$, where
  $\sfd$ (resp.\ $\mm$) is the induced Riemannian distance (resp.\
  volume measure).  In this case, a sufficient condition
  \cite[Theorem 9.1]{Grigoryan99} is
  \begin{equation}
    \label{eq:55}
    \int_{r_0}^\infty \frac{r}{\log\big(m(r)\big)}\,\d
    r=\infty,\quad
    \text{for some }\ r_0>0,\quad
    m(r):=\mm\big(\{x:\sfd(x,x_0)<r\}\big),
  \end{equation}
  which is obviously a consequence of \eqref{eq:78}.
  On the other hand, \eqref{eq:78} is always satisfied if
  the Ricci curvature of $X$ is bounded from below:
  more generally \eqref{eq:78} holds in metric spaces
  satisfying the $CD(K,\infty)$ condition, see Section \ref{sec:LSV}
  and \cite[Theorem 4.24]{Sturm06I}.
\end{remark}

\begin{proposition}[Entropy dissipation]
\label{le:diss} Let $\mm$ be a $\sigma$-finite measure satisfying
\eqref{eq:75}, let $f_0\in {\GGG L^1(X,\mm)\cap }L^2(X,\mm)$ be a nonnegative initial
datum with $\int_X
\Wgh^2\,f_0\,\d\mm<\infty$, and let $(f_t)$ be the corresponding
gradient flow of Cheeger's energy. Then the map $t\mapsto
\int_Xf_t\log f_t\,\d\mm$ is locally absolutely continuous in
$(0,\infty)$ and it holds
\begin{equation}\label{eq:eqdissrate}
\frac \d{\d t}\int_X f_t\log f_t\,\d\mm=
-\int_{\{f_t>0\}}\frac{\relgrad {f_t}^2}{f_t}\,\d\mm\qquad\text{for
a.e. $t$}\in(0,\infty).
\end{equation}
\end{proposition}
\begin{proof} The case when $\mm(X)<\infty$ can be easily deduced by
  Proposition~\ref{prop:maxprin}(\ref{maxprin:b}).
  If $\mm(X)=\infty$ we first consider regularized $C^{1,1}(0,\infty)$ and convex
  entropies $e_\eps$, $0<\eps<\rme^{-1}$, of $e(r):=r\log r$:
   \begin{displaymath}
    \begin{cases}
      e_\eps(r)=(1+\log\eps)r=e'(\eps)r&\text{ in }[0,\eps],\\
      e_\eps(r)=r\log r+\eps=e(r)-e(\eps)+\eps e'(\eps)&\text{ in }[\eps,\infty).
    \end{cases}
  \end{displaymath}
  Notice that $e_\eps'(r)=\max\{e'(r),e'(\eps)\}\le (1+\log r)^+$ because of our choice of $\eps$;
  since $(1+\log r)^+\le r$, we deduce that $e(r)\le e_\eps(r)\le
  \frac 12 r^2$ and $e_\eps(r)\downarrow e(r)$ as $\eps\downarrow0$.

  We can now define a convex and $C^{1,1}(\R)$ function by
  setting $\tilde e_\eps(r):=e_\eps(r)-(1+\log\eps)r$ for $r\ge0$ and
  $\tilde e_\eps(r)\equiv 0$ for $r<0$; applying \eqref{eq:25}
  (by the previous estimates $\int_X \tilde e_\eps(f_0)\,\d\mm<\infty$)
  and recalling that the integral of $f_t$ is constant for every $t\ge0$
  we obtain
  \begin{displaymath}
    \int_X e_\eps(f_t)\,\d\mm+\int_0^t\int_{\{f_t>\eps\}}
    \frac{\relgrad{f_t}^2}{f_t}\,\d\mm\,\d t=
    \int_X e_\eps(f_0)\,\d\mm.
  \end{displaymath}
  Passing to the limit as $\eps\down0$ and recalling the uniform
  bounds \eqref{eq:34} and \eqref{eq:45}, guaranteed by
  Theorem~\ref{thm:infinitemass}, we conclude.
  \end{proof}

\begin{remark}\label{rem:onesidedcheeger}
{\rm Although these facts will not play a role in the paper, we
emphasize that it is also possible to define one-sided Cheeger
energies $\C^+(f)$, $\C^-(f)$, by relaxing respectively the
ascending and descending slopes of Borel and $\sfd$-Lipschitz
functions w.r.t. $L^2(X,\mm)$ convergence. We still have the
representation
$$
\C^+(f)=\frac12\int_X|\rmD^+ f|^2_*\,\d\mm,\qquad
\C^-(f)=\frac12\int_X|\rmD^- f|^2_*\,\d\mm
$$
for suitable one-sided relaxed gradients $|\rmD^\pm f|_*$ with
minimal norm and it is easily seen that the functionals $\C^\pm$ are
convex and lower semicontinuous in $L^2(X,\mm)$.

Obviously $\C\geq\max\{\C^+,\C^-\}$ and
$\C^+(f)=\C^-(-f)$. Lemma~\ref{lem:strongchee} still holds with the
same proof and, using (\ref{eq:subadd2}), locality can be proved for
the one-sided relaxed gradients as well, so that $|\rmD^\pm
f|_*\leq|\rmD^\pm f|$ $\mm$-a.e. for $f$ Borel and
$\sfd$-Lipschitz. We shall see in the Section~\ref{sec:identification1} that if $\mm$
satisfies \eqref{eq:75} then for every Borel function $|\rmD^\pm
f|_*=\relgrad f$ and $\C=\C^+=\C^-$.}\fr
\end{remark}

\section{Weak upper gradients}\label{sec:differentchee}

In this subsection we define a new notion for the ``weak norm of the
gradient'' (which we will call ``minimal weak upper gradient'') of a
real valued functions $f$ on an extended metric space $(X,\sfd)$ and
we will show that this new notion essentially coincides with the
relaxed gradient. The approach that we use here is inspired by the
work \cite{Shanmugalingam00}, i.e. rather than proceeding by
relaxation, as we did for $\relgrad f$, we ask the fundamental
theorem of calculus to hold along ``most'' absolutely continuous
curves, in a sense that we will specify soon. Our definition of null
set of curves is different from \cite{Shanmugalingam00}, natural in
the context of optimal transportation, and leads to an a priori
larger class of null sets, see Remark~\ref{rem:comparenullsets};
also, another difference is that we obtain Sobolev regularity (and
not absolute continuity) along every curve, so that our theory does
not depend on the choice of precise representatives in the Lebesgue
equivalence class. In Remark~\ref{rem:compachsh2} we compare more
closely the two approaches and show, as a nontrivial consequence of
our identification results, that they lead to the same Sobolev
space.

The advantages of working with a direct definition, rather than
proceeding by relaxation, can be appreciated by looking at
Lemma~\ref{le:conditioned_UG}, where we prove absolute continuity of
functionals $t\mapsto\int_X\phi(f_t)\,\d\mm$ even along curves
$t\mapsto f_t\mm$ that are absolutely continuous in the Wasserstein
sense, compare with Proposition~\ref{le:diss} for $L^2$-gradient
flows; we can also  compute the minimal weak upper gradient for
Kantorovich potentials, as we will see in
Section~\ref{sec:weakbrenier}.

We assume in this section that $(X,\tau,\sfd)$ is an extended Polish
space and that $\mm$ is a $\sigma$-finite Borel measure in $X$
representable in the form $\rme^{V^2}\tmm$ with $\tmm(X)\leq 1$ and
$V:X\to [0,\infty)$ Borel and $\sfd$-Lipschitz. Recall that the
$p$-energy of an absolutely continuous curve has been defined in
\eqref{eq:13}, as well as the collection of curves of finite
$p$-energy $\AC p{(0,1)}X\sfd,$ which we will consider as a Borel
subset of $\CC{[0,1]}X\tau$ (and in particular a Borel subset of a
Polish space).

\subsection{Test plans, Sobolev functions along a.c.\ curves, and weak upper
  gradients}\label{subs:testplan}

Recall that the evaluation maps $\rme_t:\CC{[0,1]}X\tau\to X$ are
defined by $\rme_t(\gamma):=\gamma_t$. We also introduce the
restriction maps ${\rm restr}_t^s: \CC\bJ X\tau\to\CC\bJ X\tau$,
$0\le t\le s\le 1$, given by
\begin{equation}
{\rm restr}_t^s(\gamma)_r:=\gamma_{((1-r)t+rs)},\label{eq:93}
\end{equation}
so that ${\rm restr}_t^s$ restricts the curve $\gamma$ to the
interval $[t,s]$ and then ``stretches'' it on the whole of $[0,1]$.

\begin{definition}[Test plans]
  \label{def:testplan}
  We say that a probability measure $\ppi\in \prob{\CC{[0,1]}X\tau}$
  is a test plan if it is concentrated on $\AC{}{(0,1)}X\sfd$,
  i.e.~$\ppi\big(\CC\bJ X\tau\setminus\AC{}{(0,1)}X\sfd\big)=0$,
  and
  \begin{equation}
    \label{eq:96}
    \text{$(\e_t)_\sharp\ppi\ll\mm$\quad for all\quad $t\in\bJ$}.
  \end{equation}
  A collection $\calT$ of test plans is stretchable
  if
  \begin{equation}
    \label{eq:97}
    \ppi\in \calT\quad\Longrightarrow\quad
    ({\rm
      restr}_t^s)_\sharp\ppi\in \calT\quad\text{for every }0
    \le t\le s\le 1.
  \end{equation}
  \end{definition}

  We will often impose additional quantitative assumptions
  on test plans, besides \eqref{eq:96}. The most important one,
  which we call \emph{bounded compression},  provides
  a locally uniform upper bound on the densities of $(\rme_t)_\sharp \ppi$.
  More precisely, a test plan $\ppi$ has bounded compression on the sublevels of
  $\Wgh$ if for all $M\geq 0$ there exists $C=C(\ppi,M)\in [0,\infty)$
  satisfying
  \begin{equation}\label{eq:incompre2}
    (\e_t)_\sharp\ppi(B\cap\{\Wgh\leq M\})\leq
    C(\ppi,M)\,\mm(B)\qquad\forall B\in\BorelSets{X},\ t\in [0,1].
  \end{equation}
  The above condition \eqref{eq:incompre2}
  depends not only on $\mm$,
  but also on $V$. For finite measures $\mm$ it will be understood
  that we take $\Wgh$ equal to a constant, so that \eqref{eq:incompre2}
  does not depend on the value of the constant.

  Taking \eqref{eq:incompre2} into account,
  typical examples of
  stretchable collections $\calT$
  are the families of all the test plans with
  bounded compression which are concentrated on
  absolutely continuous curves, or on the curves of finite
  $2$-energy, or on the geodesics in $X$.

\begin{definition}[Negligible sets of curves]
{\nc Let $\calT$ be a stretchable collection of test plans and let $P$ be a statement about
absolutely continuous curves $\gamma:[0,1]\to X$. We say that $P$ holds for $\calT$-almost every 
(absolutely continuous) curve if for any $\ppi\in\calT$
the set
$$
\left\{\gamma:\ \text{$P(\gamma)$ does not hold }\right\}
$$
is contained in a $\ppi$-negligible Borel set. }
\end{definition}

In the next remark we compare our definition with the more classical
notion of ${\rm Mod}_2$-null set of absolutely continuous curve used
in \cite{Shanmugalingam00}.

\begin{remark}\label{rem:comparenullsets}
\upshape
Recall that, for a collection $\Gamma$ of absolutely
continuous curves in $(X,\sfd)$, the $2$-modulus ${\rm
Mod}_2(\Gamma)$ is defined by
$$
{\rm Mod}_2(\Gamma):=\inf\left\{\int_X g^2\,\d\mm:\ \text{$g\geq 0$
Borel, $\int_\gamma g\geq 1$ for all $\gamma\in\Gamma$}\right\}.
$$
If $\calT$ denotes the class of plans with bounded compression
defined by \eqref{eq:incompre2}, it is not difficult to show that
Borel and ${\rm Mod}_2$-null sets of curves are $\calT$-negligible.
Indeed, if $\ppi\in\calT$ has (with no loss of generality) finite
2-action and is concentrated on curves contained in $\{V\leq M\}$
and $\int_\gamma g\geq 1$ for all $\gamma\in\Gamma$, we can
integrate w.r.t. $\ppi$ and then minimize w.r.t. $g$ to get
$$
\bigl[\ppi(\Gamma)\bigr]^2\leq C(\ppi,M){\rm
Mod}_2(\Gamma)\int\int_0^1|\dot\gamma|^2\,\d s\,\d\ppi(\gamma).
$$
Proving equivalence of the two concepts seems to be difficult, also
because one notion is independent of parameterization, while the
other one (because of the bounded compression condition) takes into
account also the way curves are parameterized. \fr
\end{remark}

\begin{definition}[Weak upper gradients]\label{def:weak_upper_gradient}

Let $\calT$ be a stretchable collection of test plans. Given
$f:X\to\R$, a $\mm$-measurable function $G:X\to[0,\infty]$ is a $\calT$-weak upper
gradient of $f$ (or a weak upper gradient w.r.t.\ $\calT$) if
\begin{equation}
\label{eq:inweak} \biggl|\int_{\partial\gamma}f\biggr|\leq
\int_\gamma G{\nc<\infty}\qquad\text{for
  $\calT$-almost every $\gamma\in \AC{}{(0,1)}X\sfd$.}
\end{equation}
\end{definition}
{\nc Although the definition of weak upper gradient makes sense for
functions, rather than Lebesgue equivalence classes, this concept enjoys natural
invariance properties w.r.t. modifications in $\mm$-negligible sets, see Proposition~\ref{prop:invarianza}
below.} Notice that the measurability of $s\mapsto G(\gamma_s)$ in $[0,1]$
for $\calT$-almost every $\gamma$ is a direct consequence of the
$\mm$-measurability of $G$: indeed, if $\tilde{G}$ is a Borel
modification of $G$, $A\supset\{G\neq\tilde{G}\}$ is a
$\mm$-negligible Borel set and $\ppi$ is a test plan we have by
\eqref{eq:96} that $\ppi(\{\gamma_t\in
A\})=(\rme_t)_\sharp\ppi(A)=0$ for every $t\in [0,1]$, so that
\begin{displaymath}
 0=\int_0^1 \ppi(\{\gamma_t\in A\})\,\d t=
  \int_0^1 \int \nchi_{\{\gamma_t\in A\}}\,\d\ppi(\gamma)\,\d t=
  \int \Big(\int_0^1 \nchi_{\{\gamma_t\in A\}}\,\d t\Big)\d\ppi(\gamma).
\end{displaymath}
$\int_0^1\nchi_{\{\gamma_t\in A\}}\,\dt$ is therefore null for
$\ppi$-a.e. $\gamma$. For any curve $\gamma$ for which the integral
is null $G(\gamma_t)$ coincides a.e. in $[0,1]$ with the Borel map
$\tilde{G}(\gamma_t)$. 

\begin{remark}[Slopes of
  $\sfd$-Lipschitz functions are weak upper gradients]\label{rem:lipweak}{\rm
As we explained in Remark~\ref{rem:whenslopesare}, if $f:X\to\R$ is
Borel and $\sfd$-Lipschitz, then the local Lipschitz constant
$|\rmD  f|$ and the one-sided slopes are upper gradients. Therefore
they are also weak upper gradients w.r.t.\ any stretchable
collection of test plans sense we just defined. Notice that the
$\mm$-measurability of the slopes is ensured by
Lemma~\ref{le:measurability_of_slopes}.}\fr
\end{remark}

\begin{definition}[Sobolev regularity along almost every curve]\label{def:Sobolev}
We say that a $\mm$-measurable function
$f:X\to\R$ is \emph{Sobolev} along  $\calT$-almost every curve if,
for $\calT$-almost every curve $\gamma$,
$f\circ\gamma$ coincides a.e. in $[0,1]$ and in $\{0,1\}$ with an
absolutely continuous map $f_\gamma:[0,1]\to\R$.
\end{definition}

{\nc
In the next proposition we prove that existence and summability of a $\calT$-weak upper gradient
yields Sobolev regularity along $\calT$-almost every curve. In an earlier version of this paper, this
property was imposed a priori, and not derived as a consequence, see also Remark~\ref{re:restr} below.}

{\nc
\begin{proposition} \label{prop:Rovereto}
Assume that $\calT$ is a stretchable collection of
test plans and that $G:X\to [0,\infty]$ is a $\calT$-weak upper gradient of 
a $\mm$-measurable function $f:X\to\R$. Then $f$ is Sobolev along $\calT$-almost every curve and 
 \begin{equation}\label{eq:pointwisewug}
      \biggl|\frac{\d}{\dt}f_\gamma\biggr|\leq
      G\circ\gamma|\dot\gamma|\quad\text{a.e. in $[0,1]$, for $\calT$-almost every
        $\gamma\in\AC{}{(0,1)}X\sfd$.}
    \end{equation}
\end{proposition}
\begin{proof} The stretchable condition \eqref{eq:97} yields for every $t<s$ in $[0,1]$
    \[
    |f(\gamma_s)-f(\gamma_t)|\leq \int_t^s G(\gamma_r)|\dot\gamma_r|\,\d
    r \qquad\text{for $\calT$-almost every $\gamma$.}
    \]
Let $\ppi\in\calT$: by Fubini's theorem applied
    to the product measure $\Leb2\times\ppi$ in $(0,1)^2\times
    C([0,1];X)$, it follows that for $\ppi$-a.e. $\gamma$ the function
     $f$ satisfies
    \[
    |f(\gamma_s)-f(\gamma_t)|\leq \Bigl|\int_t^s G(\gamma_r)|\dot\gamma_r|\,\d
    r \Bigr|\qquad\text{for $\Leb{2}$-a.e. $(t,s)\in (0,1)^2$.}
    \]
    An analogous argument shows that 
    \begin{equation}
      \label{eq:222}
      \left\{
    \begin{aligned}
      \textstyle |f(\gamma_s)-f(\gamma_0)|&\textstyle 
      \leq \int_0^s
      G(\gamma_r)|\dot\gamma_r|\,\d r\\
      \textstyle |f(\gamma_1)-f(\gamma_s)|&\textstyle \leq \int_s^1
      G(\gamma_r)|\dot\gamma_r|\,\d r
    \end{aligned}\right.
    \qquad\text{for $\Leb{1}$-a.e. $s\in (0,1)$.}
\end{equation}
 Since $G\circ \gamma|\dot \gamma|\in L^1(0,1)$ for
    $\ppi$-a.e.\ $\gamma$,  
    by Lemma~\ref{lem:Fibonacci} it follows that $f\circ\gamma\in W^{1,1}(0,1)$
    for $\ppi$-a.e. $\gamma$ and (understanding the derivative of $f\circ\gamma$ as the distributional one)
    \begin{equation}\label{eq:pointwisewug1}
      \biggl|\frac{\d}{\dt}(f\circ\gamma)\biggr|\leq
      g\circ\gamma|\dot\gamma|\quad\text{a.e. in $(0,1)$, for
        $\ppi$-a.e. $\gamma$.}
    \end{equation}
  Since $\ppi$ is arbitrary, we conclude that $f\circ\gamma\in
  W^{1,1}(0,1)$ for $\calT$-a.e. $\gamma$, and therefore it admits an
  absolutely continuous representative $f_\gamma$ for which \eqref{eq:pointwisewug} holds; moreover,
  by \eqref{eq:222}, it is immediate to check that $f(\gamma(t))=f_\gamma(t)$ for $t\in \{0,1\}$ and $\calT$-a.e.\ $\gamma$.
\end{proof}
}

{\nc
\begin{remark}[Equivalent formulation]
  \label{re:restr}{\rm
  By a similar reasoning we obtain an equivalent formulation of the weak upper gradient property, in the case
  when the collection of test plans $\calT$ is stretchable and $f$ is Sobolev along
  $\calT$-almost every curve: a function $G$ satisfying $\int_\gamma G<\infty$ for $\calT$-almost every curve $\gamma$ is a weak upper gradient
  w.r.t. $\calT$ of $f$ if and only if, for $\calT$-almost every curve $\gamma$, the function $f_\gamma$ of Definition~\ref{def:Sobolev} satisfies 
  \eqref{eq:pointwisewug}. 
    \fr   }
\end{remark}}

\subsection{Calculus with weak upper gradients}
\label{subs:calculus_weak}

\begin{proposition}[Locality]\label{prop:locweak}
Let $\calT$ be a stretchable collection of test plans, let $f:X\to\R$ be $\mm$-measurable
and let $G_1,\,G_2$ be weak upper gradients of $f$
 w.r.t.\ $\calT$. Then $\min\{G_1,G_2\}$ is a $\calT$-weak upper gradient of
 $f$.
\end{proposition}
\begin{proof} {\nc We know from Proposition~\ref{prop:Rovereto} that $f$ is Sobolev along
$\calT$-almost every curve.}
Then, the claim is a direct consequence of Remark~\ref{re:restr} and \eqref{eq:pointwisewug}.
\end{proof}

The notion of weak upper gradient enjoys natural invariance properties with respect to
$\mm$-negligible sets:

\begin{proposition}[Invariance under modifications in $\mm$-negligible sets]\label{prop:invarianza}
Let
 $\calT$ be a\\
stretchable collection of test plans, let $f,\,\tilde f:X\to\R$ and
$G,\,\tilde G:X\to[0,\infty]$ be such that both $\{f\neq \tilde f\}$
and $\{G\neq \tilde G\}$ are $\mm$-negligible. Assume that $G$ is a
$\calT$-weak upper gradient of $f$. Then $\tilde G$ is a $\calT$-weak upper
gradient of $\tilde f$.
\end{proposition}
\begin{proof}
Fix a test plan $\ppi$:
it is sufficient to prove that the sets
$\bigl\{\gamma:\ f(\gamma_t)\neq \tilde f(\gamma_t)\bigr\}$,
$t=0,\,1$ and the set $\bigl\{\gamma:\ \int_\gamma G\neq\int_\gamma
\tilde G\bigr\}$
are contained in $\ppi$-negligible Borel sets.

For the first two sets the proof is obvious, because
$(\e_t)_\sharp\ppi\ll\mm$, which implies that if $A$ is a
$\mm$-negligible Borel set containing $\{f\neq\tilde{f}\}$ we have
$\ppi(\{\gamma:\ \gamma_t\in A\})=({\e_t})_\sharp\ppi(A)=0$. For the
third one we choose as $A$ a $\mm$-negligible Borel set containing
$\{G\neq\tilde{G}\}$ and we use the argument described immediately
after Definition~\ref{def:weak_upper_gradient}. 
\end{proof}

Thanks to the previous proposition we can also consider extended
real valued $f$ (as Kantorovich potentials), provided the set
$N=\{|f|=\infty\}$ is $\mm$-negligible: as a matter of fact the
curves $\gamma$ which intersect $N$ at $t=0$ or $t=1$ are
negligible, hence $\int_{\partial\gamma}f$ is defined for almost
every $\gamma$.

\begin{definition}[Minimal weak upper gradient]
  Let $\calT$ be a stretchable collection of test plans
  and let {\nc $f:X\to\R$ be a $\mm$-measurable function with a weak $\calT$-upper gradient.}
  The $\calT$-minimal weak upper gradient $\weakgradA f$ of $f$
  is the $\calT$-weak upper gradient characterized, up to
$\mm$-negligible sets, by the property
\begin{equation}\label{eq:defweakgrad}
  \weakgradA f\leq G\qquad\text{$\mm$-a.e. in $X$, for every $\calT$-weak upper gradient $G$ of $f$.}
\end{equation}  
 \end{definition}

Uniqueness of the minimal weak upper gradient is obvious. For
existence, let us consider a minimizing sequence $(G_n)$ for the problem
{\nc $$
\inf\left\{\int_X {\rm tan}^{-1}(G)\vartheta\,\d\mm:\ \text{$G$ weak $\calT$-upper gradient of $f$}\right\}
$$
with $\vartheta$ is as in \eqref{eq:17}.}
We immediately see, thanks to Proposition~\ref{prop:locweak}, that
we can assume with no loss of generality that $G_{n+1}\leq G_n$.
Hence, by monotone convergence, the function $\weakgradA f:=\inf_n G_n$ is a
{\nc $\calT$-weak upper gradient} of $f$ and $\int_X {\rm tan}^{-1}(G)\vartheta\,\d\mm$ is minimal at $G=\weakgradA f$. This
minimality, in conjunction with Proposition~\ref{prop:locweak},
gives \eqref{eq:defweakgrad}.

\begin{remark} [Comparison with Newtonian
spaces]\label{rem:compachsh2} \upshape Shanmugalingam introduced in
\cite{Shanmugalingam00} the Newtonian space $N^{1,2}(X,\sfd,\mm)$ of
all functions $f:X\to\R$ such that $\int f^2\,\d\mm<\infty$ and the
inequality
\begin{equation}\label{eq:june6}
|f(\gamma_1)-f(\gamma_0)|\leq\int_\gamma G
\end{equation}
holds out of a ${\rm Mod}_2$-null set of curves, for some $G\in
L^2(X,\mm)$. Then, she defined $|\rmD  f|_S$ as the function $G$ in
\eqref{eq:june6} with smallest $L^2$ norm and proved
\cite[Proposition~3.1]{Shanmugalingam00} that functions in
$N^{1,2}(X,\sfd,\mm)$ are absolutely continuous along ${\rm
Mod}_2$-almost every curve.

Remarkably, Shanmugalingam proved (the proofs in
\cite{Shanmugalingam00} work, with no change, even in the case of
extended metric measure spaces) this connection between Newtonian
spaces and Cheeger's functional $\underline{\sf Ch}_*$ described in
Remark~\ref{rem:compachsh}: $f\in D(\underline{\sf Ch}_*)$ if and
only if there is $\tilde{f}\in N^{1,2}(X,\sfd,\mm)$ in the Lebesgue
equivalence class of $f$, and the two notions of gradient $|\rmD 
f|_S$ and $|\rmD  f|_C$ coincide $\mm$-a.e. in $X$.

If $\calT$ denotes the class of plans with bounded compression
defined by \eqref{eq:incompre2}, the inclusion between null sets provided by
Remark~\ref{rem:comparenullsets} shows that the situation described
in Remark~\ref{rem:compachsh} is reversed. Indeed, while $|\rmD 
f|_C\leq\relgrad{f}$, the gradient $|\rmD  f|_S$ is larger
$\mm$-a.e. than $|\rmD  f|_{w,\calT}$, so that
\begin{equation}\label{eq:quantigradienti}
|\rmD  f|_{w,\calT}\leq |\rmD  f|_S=|\rmD  f|_C\leq\relgrad{f}
\qquad\text{$\mm$-a.e. in $X$.}
\end{equation}
{\nc With this choice of $\calT$, we can define the Sobolev space
$W^{1,2}_w(X,\sfd,\mm)$ with the same idea presented in
Remarks~\ref{re:sobolev}, \ref{rem:compachsh}: $f\in
W^{1,2}_w(X,\sfd,\mm)$ provided $f$ is in  $L^2(X,\mm)$, Sobolev along
$\calT$-a.e. curve and  $|\rmD  f|_{w,\calT}\in L^2(X,\mm)$. The
inequalities in \eqref{eq:quantigradienti} then yield
\begin{equation}
\label{eq:quantisobolev}
W_*^{1,2}(X,\sfd,\mm)\subset W^{1,2}(X,\sfd,\mm)=N^{1,2}(X,\sfd,\mm)\subset W^{1,2}_w(X,\sfd,\mm),
\end{equation}
(where some care should be used relating the Newtonian space with the Sobolev spaces 
because in the former the choice of $\mm$-a.e. representative of a function matters).}

Although we are not presently able to reverse the inclusion between
null sets, a nontrivial consequence of our identification of
$|\rmD  f|_{w,\cal T}$ and $\relgrad{f}$, proved in the next
section, is that all these gradients coincide $\mm$-a.e. in $X$, and hence all Sobolev/Newtonian spaces coincide as well.

Since $D(\C)\subset D(\underline{\sf Ch}_*)$, a byproduct of the
absolute continuity of functions in Newtonian spaces, that however
will not play a role in our paper, is that functions in $D(\C)$ have
a version which is absolutely continuous along ${\rm Mod}_2$-a.e.
curve.\fr
\end{remark}

\begin{remark}
  \label{rem:monotone-weakT}
  \upshape
  Notice that the notion of weak gradient do depend on the class
  $\calT$ of test plans
  (which, in turn, might depend on $\Wgh$).

  If $\calT_1\subset \calT_2$ are stretchable collections
  of test plans and a function $f:X\to\R$ is Sobolev along
  $\calT_2$-almost every absolutely continuous curve, then
  $f$ is Sobolev along $\calT_1$-almost every absolutely continuous
  curve and
  \begin{equation}
    \label{eq:98}
    |\rmD  f|_{w,\calT_1}\le |\rmD  f|_{w,\calT_2}.
  \end{equation}
  Thus, larger classes
  of test plans induce smaller classes of weak upper
  gradients, hence larger minimal weak upper gradients.\fr
\end{remark}

Another important property of weak upper gradients is their
stability w.r.t. $L^p$ convergence: we state it for all the
stretchable classes of test plans satisfying a condition weaker than
bounded compression, inspired to the ``democratic'' condition
introduced by \cite{Lott-Villani-Poincare}.
\begin{theorem}[Stability w.r.t. $\mm$-a.e. convergence]\label{thm:stabweak}
Let us suppose that $\calT$ is a stretchable collection of test
plans concentrated on $\AC p{(0,1)}X\sfd$ for some $p\in (1,\infty]$
such that for all $\ppi\in \calT$ and all $M\geq 0$ there exists
$C=C(\ppi,M)\in [0,\infty)$ satisfying
\begin{equation}\label{eq:incompre}
   \int_0^1(\e_t)_\sharp\ppi(B\cap\{\Wgh\leq M\})\,\dt\leq
   C(\ppi,M)\,\mm(B)\qquad\forall B\in\BorelSets{X}.
\end{equation}
Assume that $f_n$ are $\mm$-measurable and that $G_n$ are $\calT$-weak upper gradients of $f_n$.
Assume furthermore that $f_n(x)\to f(x)\in\R$ for $\mm$-a.e. $x\in
X$ and that $(G_n)$ weakly converges to $G$ in $L^q(\{\Wgh\leq
M\},\mm)$ for all $M\geq 0$, where $q\in [1,\infty)$ is the
conjugate exponent of $p$. Then $G$ is a $\calT$-weak upper gradient
of $f$.
\end{theorem}
\begin{proof}
{\nc Fix a test plan $\ppi$ and assume first that, for some constants $L$ and $M$,
$\Energy p\gamma\leq L<\infty$ $\ppi$-a.e. 
and that $\ppi$-a.e. $\gamma$ is contained in $\{\Wgh\leq M\}$.}
By Mazur's theorem we can find convex combinations
$$
H_n:=\sum_{i=N_h+1}^{N_{h+1}}\alpha_iG_i\qquad\text{with
$\alpha_i\geq 0$, $\sum_{i=N_h+1}^{N_{h+1}}\alpha_i=1$,
$N_h\to\infty$}
$$
converging strongly to $G$ in $L^q(\{\Wgh\leq M\},\mm)$. Denoting by
$\tilde f_n$ the corresponding convex combinations of $f_n$, $H_n$
are weak upper gradients of $\tilde f_n$ and still $\tilde f_n\to f$
$\mm$-a.e. in $\{\Wgh\leq M\}$.

Since for every nonnegative Borel function $\varphi:X\to [0,\infty]$
and any $M$ it holds (with $C=C(\ppi,M)$)
\begin{align}
  \notag\int\Big(\int_{\gamma\cap \{V\le M\}}\varphi\Big)\,\d\ppi&=
  \int\Big(\int_0^1 \nchi_{\{V\le M\}}(\gamma_t)\varphi(\gamma_t)|\dot
  \gamma_t|\,\d t\Big)\,\d\ppi
  \\&\notag
  \le
  \int\Big(\int_0^1 \nchi_{\{V\le M\}}(\gamma_t)\varphi^q(\gamma_t)\,\d
  t\Big)^{1/q}
  \Big(\int_0^1 |\dot
  \gamma_t|^p\,\d t\Big)^{1/p}\,\d\ppi
  \\&\notag
  \le
  \Big(\int_0^1 \int_{\{V\le M\}} \varphi^q\,\d(\rme_t)_\sharp\ppi\,\d
  t\Big)^{1/q}
  \Big(\int \Energy p\gamma\,\d\ppi\Big)^{1/p}
  \\&\le
  \Big(C\int_{\{V\le M\}} \varphi^q\,\d\mm\Big)^{1/q}\Big(\int \Energy p\gamma\,\d\ppi\Big)^{1/p},
\label{eq:21}
 \end{align}
we obtain, for $\bar C:=C^{1/q}L^{1/p}$,
$$
\int\biggl(\int_{\gamma\cap\{\Wgh\leq
  M\}}|H_n-G|\biggr)\,\d\ppi\leq
\bar C\|H_n-G\|_{L^q(\{V\le M\},\mm)} \to 0.
$$

By a diagonal argument we can find a subsequence $n(k)$ independent
of $M\in\N$ such that
$\int_\gamma|H_{n(k)}-G|\to 0$ as
$k\to\infty$ for $\ppi$-a.e. $\gamma$ contained in
$\{\Wgh\leq M\}$, and thus for $\ppi$-a.e. $\gamma$. Since
$\tilde{f}_n$ converge $\mm$-a.e. to $f$ and the marginals of $\ppi$
are absolutely continuous w.r.t. $\mm$ we have also that for
$\ppi$-a.e. $\gamma$ it holds $\tilde{f}_n(\gamma_0)\to f(\gamma_0)$
and $\tilde{f}_n(\gamma_1)\to f(\gamma_1)$. {\nc Still using \eqref{eq:21}, we have
$\int_\gamma G<\infty$ for $\ppi$-a.e. $\gamma$.}

{\nc
If we fix a curve $\gamma$ satisfying the above properties, we can pass to the limit in the weak
upper gradient property written for $\tilde f_{n(k)}$ to obtain
that $G$ is a $\calT$-weak upper gradient of $f$. Finally we remove the assumptions initially
made on $\ppi$ using \eqref{eq:13} and the fact that any curve
$\gamma$ is contained in $\{\Wgh\leq M\}$ for sufficiently large $M$.}
\end{proof}

\begin{corollary}\label{cor:weakrel}
Let $\mathcal T$ be a stretchable collection of test plans
satisfying \eqref{eq:incompre} and concentrated on $\AC
2{(0,1)}X\sfd$, and let $\tilde{\sf Ch}$ be defined as in
\eqref{eq:extchee}. If $f\in D(\tilde{\sf Ch})$ then $f$ is Sobolev
along $\calT$-almost every curve and $\weakgradA f\leq\relgrad f$
$\mm$-a.e. in $X$.
\end{corollary}
\begin{proof}
{\nc We consider first the case when $f$ is bounded.}
We already observed in Remark~\ref{rem:lipweak} that, for a
Borel $\sfd$-Lipschitz function $f$, the local Lipschitz constant is
a $\calT$-weak upper gradient. Now, pick a sequence $(f_n)$ of Borel
$\sfd$-Lipschitz functions converging to $f$ in $L^2(X,\mm)$ such
that $|\rmD  f_n|$ converge weakly in $L^2(X,\mm)$ to $\relgrad f$,
thus in particular weakly in $L^2(\{\Wgh\leq M\},\mm)$ to $\relgrad
f$ for all $M\geq 0$. Then, Theorem~\ref{thm:stabweak} ensures that
$\relgrad f$ is a $\calT$-weak upper gradient for $f$. 

{\nc In the general case, if $f_N$ are the standard truncations of $f$, we can pass
to the limit in the inequality $\weakgradA {f_N}\leq\relgrad {f_N}$ using the chain rule for relaxed
gradients and the stability of weak gradients.}
\end{proof}

We shall also need chain rules for minimal weak upper gradients.

\begin{proposition}[Chain rule for minimal weak upper gradients]\label{prop:chainweak}
Let $\calT$ be as in Theorem~\ref{thm:stabweak}. If a $\mm$-measurable function $f:X\to\R$ 
{\nc has a $\calT$-weak upper gradient}, the following properties hold:
\begin{itemize}
\item[(a)] for any $\Leb{1}$-negligible Borel set $N\subset\R$ it holds
$\weakgradA f=0$ $\mm$-a.e. on $f^{-1}(N)$;
\item[(b)] $\weakgradA{\phi(f)}=\phi'(f)\weakgradA f$ $\mm$-a.e. in $X$, with the convention $0\cdot\infty=0$,
for any nondecreasing function $\phi$, locally Lipschitz on an
interval containing the image of $f$.
\end{itemize}
\end{proposition}
\begin{proof} {\nc We use the equivalent formulation of Remark~\ref{re:restr}
and the well-known fact that both (a) and (b) are true when $X=\R$
endowed with Euclidean distance and Lebesgue measure
and $f$ is absolutely continuous.
We can prove (a) setting 
$$
G(x):=
\begin{cases}
\weakgradA f(x) &\text{if $f(x)\in\R\setminus N$};\\
0 &\text{if $f(x)\in N$}
\end{cases}
$$
and noticing the validity of (a) for real-valued absolutely continuous maps gives that $G$ is $\calT$-weak upper gradient of $f$.
Then, the minimality of $\weakgradA f$ gives $\weakgradA f\leq G$ $\mm$-a.e. in $X$.

By a similar argument based on \eqref{eq:pointwisewug} we can prove that $\weakgradA{\phi(f)}\leq\phi'(f)\weakgradA f$
$\mm$-a.e. in $X$. Then, the same subadditivity argument of Proposition~\ref{prop:chainrule}(d)
provides the equality $\mm$-a.e. in $X$.
}
%
\end{proof}

\begin{lemma}
  \label{le:conditioned_UG}
  Let $\calT$ be the collection of all the test plans
  concentrated on $\AC2{(0,1)}X\sfd$
  with bounded compression on the sublevels of $\Wgh$
  (i.e.\ satisfying \eqref{eq:incompre2}).\\
  Let $\mu\in\AC2{(0,T)}{\prob X}{W_2}$ be an absolutely continuous
  curve with uniformly bounded densities $f_t=\d\mu_t/\d\mm.$
  Let $\phi:[0,\infty)\to\R$ be a convex
  function with $\phi(0)=0$ and $\phi'$ locally Lipschitz in $(0,\infty)$.
  We suppose that for a.e.\ $t\in (0,T)$ $f_t$ is Sobolev along $\calT$-almost all
  curves and that
  \begin{gather}
    \label{eq:40}
    H^2_t:=\int_{X}\weakgradA {f_t}^2\,\d\mm<\infty,\quad
    G^2_t:=\int_{\{f_t>0\}} \Big(\phi''(f_t)\weakgradA
    {f_t}\Big)^2f_t\,\d\mm<\infty,
 \end{gather}
 for a.e. $t\in (0,T)$. Assume in addition that
 $G,\,H\in L^2(0,T)$ and that $\int_X |\phi(f_0)|\,\d\mm<\infty$. Then
 $t\mapsto\int_X|\phi(f_t)|\,\d\mm$ is bounded in $[0,T]$,
  \begin{equation}
    \label{eq:81}
    \Phi_t:=\int_X \phi(f_t)\,\d\mm
    \quad\text{is absolutely continuous in $[0,T]$ and}\quad
    \Big|\frac \d{\d t}\Phi_t\Big|\le  G_t\,|\dot\mu_t|
    \quad\text{a.e. in $(0,T)$.}
  \end{equation}
  If moreover $\phi'$ is Lipschitz on  an interval containing the image of $f_t$, $t\in
  [0,T]$, then the pointwise
  estimates hold
  \begin{equation}
    \label{eq:92}
    \limsup_{s\downarrow t}\frac{
      \Phi_t-\Phi_s}{s-t}\le  G_t \limsup_{s\downarrow t}\media_t^s
    |\dot\mu_r|\,\d r,\quad
    \liminf_{s\downarrow t}\frac{
      \Phi_t-\Phi_s}{s-t}\le  G_t \liminf_{s\downarrow t}\media_t^s
    |\dot\mu_r|\,\d r.
  \end{equation}
\end{lemma}
\begin{proof}
  It is not restrictive to assume $T=1$.
  Let $C$ be a constant satisfying $\mu_t\leq C\mm$ for
  all $t\in [0,1]$ and notice that, by interpolation, $f_t$ are
  uniformly bounded in all spaces $L^p(X,\mm)$. {\nc In addition, 
  by Kantorovich duality with $c(x,y)=\sfd(x,y)$,
  $f_s$ weakly converge to $f_t$ as $s\to t$ in duality with the class ${\cal Y}$ of Borel, bounded and
  $\sfd$-Lipschitz functions. Since ${\cal Y}\cap L^p(X,\mm)$} 
  is dense in $L^p(X,\mm)$ for {\nc $1\leq p<\infty$} (thanks to the
  existence of the $\sfd$-Lipschitz weight function $\Wgh$ whose
  sublevels have finite $\mm$-measure), we obtain that $t\mapsto f_t$
  is continuous in the weak topology of $L^q(X,\mm)$ (weak$^*$ if $q=\infty$),
  with $q$ dual exponent of $p$.
  Arguing as in \cite{Lisini07,Ambrosio-Gigli11}, see also the work in progress
  \cite{Lisini11} for the case of extended metric spaces, we can find
  $\ppi\in\prob{\CC{[0,1]}{X}{\tau}}$ concentrated in $\AC2{(0,1)}X\sfd$ and satisfying
  \begin{equation}
   \label{eq:14}
     \mu_t=(\rme_t)_\sharp\ppi\quad\text{for every $t\in
     [0,1]$},\qquad
     |\dot \mu_t|^2 =\int |\dot
     \gamma_t|^2\,\d\ppi(\gamma)\quad \text{for a.e.\ }t\in (0,1),
 \end{equation}
  so that $\ppi\in \calT$.
  Let us first suppose that $\phi'$ is locally Lipschitz continuous in
  $[0,\infty)$, so that $\Phi_t$ is everywhere finite. Possibly replacing $\phi(z)$ by
  $\phi(z)-\phi'(0)z$ we can assume that $\phi$
  is nonnegative and nondecreasing.
  {\GGG It follows
  that $\Phi_t$ is lower semicontinuous.}

  We pick a point $t$ such that $f_t$ is Sobolev along $\ppi$-almost all
  curves and $H_t<\infty$ and we set
  $h_t:=\phi'(f_t)$,
  $g_t:=\weakgradA{h_t}=\phi''(f_t)\weakgradA{f_t}$.
  Then for every $s\in (0,t)$ we have
  \begin{align}
    \notag
    \Phi_t&-\Phi_s \le \int_X
    \phi'(f_t)(f_t-f_s)\,\d\mm=
    \int \big(h_t(\gamma_t))-h_t(\gamma_s))\,\d\ppi(\gamma)
    \le
    \int \int_s^t g_t(\gamma_r)|\dot\gamma_r|\,\d r\d\ppi(\gamma)
    \\&
    \label{eq:76}
    \le
    \int_s^t \Big(\int_{X} |g_t|^2\,f_r\,\d\mm\Big)^{1/2} \Big(\int
    |\dot\gamma_r|^2\,\d\ppi(\gamma)\Big)^{1/2} \,\d r
    =\int_s^t \Big(\int_{X} |g_t|^2\,f_r\,\d\mm\Big)^{1/2} |\dot\mu_r| \,\d
    r.
  \end{align}
  Since $H\in L^2(0,1)$ we deduce from Lemma~\ref{lem:realanalysis} (with $w=-\Phi$,
  $L=|\dot\mu|$, $g=C^{1/2} H$ at all points $t$ such that $f_t$ is Sobolev along $\ppi$-almost all
  curves, $g=+\infty$ elsewhere) that $\Phi$ is absolutely continuous.

  Writing the inequalities analogous to \eqref{eq:76} for $s>t$,
  dividing by $s-t$, and passing to the limit as $s\downarrow t$,
  thanks to the $w^*$ continuity of $r\mapsto f_r$
  in $L^\infty(X,\mm)$ we get the bound \eqref{eq:92} (and thus
  \eqref{eq:81} when $t$ is also a differentiability point for $\Phi$
  and a Lebesgue point for $|\dot\mu|$).

  When $\phi$ is an arbitrary convex function, for $\eps\in (0,1]$ we set
  \begin{displaymath}
    \phi_\eps(r):=
    \begin{cases}
      r\phi'(\eps)&\text{if }0\le r\le \eps,\\
      \phi(z)-\phi(\eps)+\eps\phi'(\eps)&\text{if }r\ge \eps;
    \end{cases}
  \end{displaymath}
  it is easy to check that $\phi_\eps$ is convex, with locally
  Lipschitz derivative in
  $[0,\infty)$ and that $\phi_\eps\downarrow \phi$ as $\eps\downarrow 0$,
  since $\eps\mapsto\eps\phi'(\eps)-\phi(\eps)$ is increasing and
  converges to $0$ as $\eps\downarrow 0$.
  Notice moreover that $(\phi_\eps)''\le \phi''$.
  Applying the integral form of \eqref{eq:81} to
  $\Phi^\eps_t:=\int_X \phi_\eps(f_t)\,\d\mm$ we get
  \begin{equation}
    \label{eq:83}
   \big|\Phi^\eps_t-\Phi^\eps_s\big|\le \int_s^t  G_r\,|\dot\mu_r|\,\d
   r\quad\text{for every }0\le s<t\le 1.
  \end{equation}
  Since $\Phi^\eps_0\to\Phi_0$, it follows that all the functions $\Phi^\eps_t$ are
  uniformly bounded. In addition, \eqref{eq:76} with $\phi=\phi_\eps$ gives that
  $$\int_X(\phi_\eps)^-(f_s)\,\d\mm\leq\int_X(\phi_\eps)^+(f_s)\,\d\mm+R\leq
  \int_X(\phi_1)^+(f_s)\,\d\mm+R
  $$
  with $R$ uniformly bounded in $s$ and
  $\eps$ (notice that $t$ can be chosen independently of $s$ and $\eps$).
  Hence, applying the monotone convergence theorem we obtain the
  uniform bound on $\|\phi(f_t)\|_{L^1(X,\mm)}$ and
  pass to the limit in \eqref{eq:83} as $\eps\down0$, obtaining \eqref{eq:81}.
\end{proof}

\begin{remark}  [Invariance properties] \label{rem:conformalwug} {\rm
If $\calT$ is the collection of all test plans concentrated on
$\AC2{(0,1)}X\sfd$ with bounded compression on the sublevels of
$\Wgh$ (according to \eqref{eq:incompre2}) all the concepts
introduced so far (test plans, negligible sets of curves, weak upper
gradient and minimal weak upper gradient) are immediately seen to be
invariant if one replaces $\mm$ with the finite measure
$\tmm:=\rme^{-\Wgh^2}\mm$ (recall \eqref{eq:75}): indeed, any test
plan with bounded compression relative to $\tmm$ is a test plan with
bounded compression relative to $\mm$ and any test plan bounded
compression relative to $\mm$ can be monotonically approximated by
analogous test plans relative to $\tmm$. A similar argument holds
for plans satisfying \eqref{eq:incompre}. \fr}
\end{remark}

\begin{remark}{\rm As for Cheeger's energy and the relaxed gradient, if no additional
assumption on $(X,\tau,\sfd,\mm)$ is made, it is well possible that
the weak upper gradient is trivial.

This is the case of the second example considered in
Remark~\ref{re:puoesserebanale}, where it is easy to check that the
class of absolutely continuous curves contains just the constants,
so that $\weakgradA f\equiv0$ for every $f\in L^2([0,1];\mm)$
independently from the choice of $\calT$. In order to exclude such
situations, we are going to make additional assumptions on
$(X,\tau,\sfd,\mm)$ in the next sections, as the lower
semicontinuity of $|\rmD^-\entv|^2(f\mm)$: this ensures, as we
will see in Theorem~\ref{thm:slopefisherconv}, its agreement with
$8\C(\sqrt{f})$. Since $\entv$ is not trivial, the same is true for
$|\rmD^-\entv|$ and for $\C$. In turn, we will see that lower
semicontinuity of $|\rmD^-\entv|^2$ is implied by
$CD(K,\infty)$.}\fr
\end{remark}

\section{Identification between relaxed gradient and weak upper gradient}
\label{sec:identification1}

The key statement that will enable us to prove
 the main identification result of this section is provided in the following lemma. It
corresponds precisely to \cite[Proposition~3.7]{GigliKuwadaOhta10}:
the main improvement here is the use of the refined analysis of the
Hamilton-Jacobi equation semigroup  we did in
Section~\ref{sec:hopflax}, together with the use of relaxed
gradients, in place of the standard Sobolev spaces in Alexandrov
spaces. In this way we can also avoid any lower curvature bound on
$(X,\sfd)$ and we do not even require that $(X,\sfd)$ is a length
space.

\begin{lemma}[A key estimate for the Wasserstein velocity]
\label{le:key} Let $(X,\tau,\sfd,\mm)$ be a Polish extended measure
space satisfying
\begin{equation}
  \label{eq:61}
  \mm\big(\bigl\{x\in X:\sfd(x,K)\le r\bigr\}\big)<\infty\quad
  \text{for every compact $K\subset X$ and $r>0$.}
\end{equation}
Let $(f_t)$ be the gradient flow of $\C$ in $L^2(X,\mm)$ starting
from a nonnegative $f_0\in L^2(X,\mm)$ and let us assume that
\begin{equation}
  \label{eq:62}
  \int_X f_t\,\d\mm=1,\quad
  \int_0^t\int_{\{f_s>0\}} \frac{\relgrad {f_s}^2}{f_s}\,\d\mm\,\d
  s<\infty\quad\text{for every }t\ge0.
\end{equation}
Then, setting
$\mu_t:=f_t\mm\in\Probabilities{X}$,
the curve $t\mapsto \mu_t:=f_t\mm$ is locally
 absolutely continuous from $(0,\infty)$ to $(\pro{\mu_0}(X),W_2)$ and its metric speed
$|\dot\mu_t|$ satisfies
\begin{equation}\label{eq:lekey}
|\dot\mu_t|^2\leq\int_{\{f_t>0\}}\frac{\relgrad{f_t}^2}{f_t}\,
\d\mm\qquad\text{for a.e. $t\in (0,\infty)$}.
\end{equation}
\end{lemma}
\begin{proof}
We start from the duality formula \eqref{eq:dualitabasebis}: it is
easy to check that it can be written as
\begin{equation}\label{eq:dualityQ}
\frac{W_2^2(\mu,\nu)}2=\sup_{\nc\phi}\int_X Q_1\phi\,\d\nu-\int_X\phi\,\d\mu\\
\end{equation}
where the supremum runs in $C_b(X)$. Now, if $\mu\ll\mm$, we may
equivalently consider $\tau$-lower semicontinuous functions
{\nc $\phi$} of the form \eqref{eq:compactreduction}; indeed, given
{\nc $\phi\in C_b(X)$}, considering a sequence of compact sets
$K_n\subset X$ whose union is of full $\mm$-measure and setting
$$
{\nc \phi_n(x)}:=
\begin{cases}
\phi(x) &\text{if $x\in K_n$;}\\ \sup\phi &\text{if $x\in
X\setminus K_n$}
\end{cases}
$$
we obtain $\phi_n\downarrow\phi$ $\mm$-a.e. and
$Q_1\phi_n\geq Q_1\phi$.

Moreover, since \eqref{eq:dualityQ} is invariant by adding constants
to $\phi$, we can always assume that $M=0$ in
\eqref{eq:compactreduction}, so that $\phi$ vanishes outside a
compact set $K$.

Now, if $\phi$ is of the form \eqref{eq:compactreduction} with
$M=0$, we notice that for all $t>0$ the map $Q_t\phi$ is
$\sfd$-Lipschitz,
bounded and lower semicontinuous (the latter
property follows by Proposition~\ref{prop:goodBorel}), and
\begin{equation}
  \label{eq:63}
  Q_t\phi(x)=0\quad\text{if}\quad\sfd(x,K)\ge 2\sqrt {-\min_K \phi}
  \text{ and }t\le 2,
\end{equation}
so that, by \eqref{eq:61}, $(Q_t(\phi))_{t\in [0,2]}$ is
uniformly bounded in each $L^p(X,\mm)$.

{\nc Fix now a function $\phi$ of the form 
\eqref{eq:dualityQ} and set $\varphi:=Q_\eps\phi$, for some $\eps\in (0,1)$.}
Observe that, thanks to the pointwise estimates \eqref{eq:partialtquantitative} and
\eqref{eq:lipquantitative1}, the map $t\mapsto Q_t\varphi$ is
Lipschitz from $[0,1]$ with values in $L^\infty(X,\mm)$, and a fortiori in
$L^2(X,\mm)$ by \eqref{eq:63}. In addition, the ``functional''
derivative (i.e. the strong limit in $L^2$ of the difference
quotients) $\partial_t Q_t\varphi$ of this $L^2(X,\mm)$-valued map
is easily seen to coincide, for a.e. $t$, with the map
$\frac{\d^+}{\d t }Q_t\varphi(x)$. Recall also that, still thanks to
Proposition~\ref{prop:goodBorel}, the latter map is Borel and
$|\rmD  Q_t\varphi|$ is $\BorelSetsStar{X\times
(0,\infty)}$-measurable.

Fix also $0\leq t<s\leq 1$, set $\ell=(s-t)$ and recall that since $(f_t)$ is
the gradient flow of $\C$ in $L^2(X,\mm)$, the map $[0,\ell]\ni
r\mapsto f_{t+r}$ is Lipschitz with values in $L^2(X,\mm)$.

Now, for $a,\,b:[0,\ell]\to L^2(X,\mm)$ Lipschitz, it is well known
that $t\mapsto \int_X a_tb_t\,\d\mm$ is Lipschitz in $[0,\ell]$ and
that $\bigl(\int_X a_tb_t\,\d\mm)'=\int_X
b_t\partial_ta_t\,\d\mm+\int_Xa_t\partial_t b_t\,\d\mm$ for a.e.
$t\in [0,\ell]$. Therefore we get
$$
\frac{\d}{\d r}\int_X Q_{r/\ell}\varphi \, f_{t+r}\,\d\mm= \int_X
\frac{1}{\ell}
\xi_{t/\ell}
 \,f_{t+r}+ Q_{r/\ell}\varphi \,\Deltam
f_{t+r}\,\d\mm\qquad\text{for a.e. $r>0$,}
$$
 where $\xi_s(x):=\frac{\d^+}{\d t}Q_t\varphi\bigr\vert_{t=s}(x)$; we
have then:
\begin{equation}
\label{eq:step1}
\begin{split}
\int_X Q_1\varphi\,\d\mu_s-\int_X\varphi\,\d\mu_t&=\int_X Q_1\varphi
f_{t+\ell}\,\d\mm-\int_X\varphi f_t\,\d\mm\\
&=\int_0^\ell\int_X \frac{1}{\ell}\xi_{r/\ell}\,f_{t+r}+
Q_{r/\ell}\varphi
\,\Deltam f_{t+r}\,\d\mm \d r\\
&=\int_X\int_0^\ell\ \frac{1}{\ell}\xi_{r/\ell}\,f_{t+r}+
Q_{r/\ell}\varphi
\,\Deltam f_{t+r}\,\d r \d\mm\\
&\le
\int_X\int_0^\ell\ -\frac{|\rmD  Q_{r/\ell}\varphi|^2}{2\ell}+
Q_{r/\ell}\varphi\, \Deltam f_{t+r}\,\d r\, \d\mm.
\end{split}
\end{equation}
In the last two steps we used first Fubini's theorem and then
Theorem~\ref{thm:subsol}. Observe that by inequalities
\eqref{eq:delta1} and \eqref{eq:facile} we have (using also that $
\relgrad{f_s}=0$ $\mm$-a.e. on $\{f_s=0\}$)
\[
\begin{split}
\int_X Q_{r/\ell}\varphi \,\Deltam f_{t+r}\,\d\mm&\leq
\int_X\relgrad{Q_{r/\ell}\varphi}\,\relgrad{f_{t+r}}\,\d\mm\\
&\leq\frac1{2\ell}\int_X|\rmD  Q_{r/\ell}\varphi
|^2f_{t+r}\,\d\mm+\frac\ell2\int_{\{f_{t+r}>0\}}\frac{\relgrad{f_{t+r}}^2}{f_{t+r}}\,\d\mm.
\end{split}
\]
{\nc Plugging this inequality in \eqref{eq:step1}, using once more
Fubini's theorem and recalling that} 
{\GGG by the definition of $\varphi,$ $Q_1\varphi=Q_1(Q_\eps \phi)\ge
  Q_{1+\eps}\phi,$} we obtain
\[
\int_X Q_{1+\eps}\phi\,\d\mu_s-\int_XQ_\eps\phi\,\d\mu_t\leq
\frac\ell2\int_0^\ell\int_{\{f_{t+r}>0\}}\frac{\relgrad{f_{t+r}}^2}{f_{t+r}}\,\d
r\d\mm.
\]
Since $\phi$ is $\tau$-lower semicontinuous, by
Remark~\ref{rem:goodBorel} we have $Q_\eps\phi\uparrow\phi$ as
$\eps\downarrow 0$. Hence, taking also into account that
$Q_{1+\eps}\phi\to Q_1\phi$ as
$\eps\downarrow 0$ (recall the continuity property
\eqref{eq:continuityQtf} and that $t_*\equiv\infty$ in this case),
we obtain
\[
\int_X Q_1\phi\,\d\mu_s-\int_X\phi\,\d\mu_t\leq
\frac\ell2\int_0^\ell\int_{\{f_{t+r}>0\}}\frac{\relgrad{f_{t+r}}^2}{f_{t+r}}\,\d
r\d\mm.
\]
{\nc This latter bound holds for all functions $\phi$ of the form \eqref{eq:dualityQ}, so that
the remarks made at the beginning of the proof yield}
\[
W_2^2(\mu_t,\mu_s)\leq\ell\int_0^\ell\int_{\{f_{t+r}>0\}}\frac{\relgrad{f_{t+r}}^2}{f_{t+r}}\d
r\d\mm,\qquad
\ell=s-t.
\]
By \eqref{eq:62} we immediately get that $\mu_t\in\pro{\mu_0}(X)$
and \eqref{eq:lekey} holds.
\end{proof}

In the next two results we will consider the class of (measurable) functions
\begin{equation}
  \label{eq:101}
  f:X\to \R\quad\text{such that }
  f_N:=\min\{N,\max\{f,-N\}\}\in L^2(X,\mm)\quad\text{for every $N>0$.}
\end{equation}

\begin{theorem}[Relaxed and weak upper gradients coincide]\label{thm:rel=weak}
  Let $(X,\tau,\sfd,\mm)$ be a Polish extended measure space with
$\mm$ satisfying \eqref{eq:75}, let $\calT$ be the collection of
all test plans concentrated on $\AC 2{(0,1)}X\sfd$ with bounded compression on the 
sublevels of $\Wgh$ according to \eqref{eq:incompre2}, and let $f:X\to\R$
be a $\mm$-measurable function satisfying \eqref{eq:101}.
\\
Then $f$ has relaxed gradient
$\relgrad f$ (according to \eqref{eq:extrelgrad}) in $L^2(X,\mm)$
iff $f$ is Sobolev on $\calT$-almost all curves and
$\weakgradA f\in L^2(X,\mm)$. In this case
\begin{equation}
  \label{eq:22}
  \relgrad f=\weakgradA f\quad \mm\text{-a.e.\ in }X.
\end{equation}
{\nc Finally, it holds
\[
W_*^{1,2}(X,\sfd,\mm)=W^{1,2}(X,\sfd,\mm)=W_w^{1,2}(X,\sfd,\mm),
\]
and these spaces also coincide with $N^{1,2}(X,\sfd,\mm)$ provided we think this latter space as a space of $\mm$-a.e. 
equivalence classes of functions. For a function $f$ belonging to these spaces, 
the relaxed gradient $\relgrad f$ coincides $\mm$-a.e. in $X$ with
the Newtonian and Cheeger gradients $|\rmD f|_S$ and $|\rmD f|_C$ of Remark~\ref{rem:compachsh2} and with
 $\weakgradA f$.} 
\end{theorem}
\begin{proof}
Taking into account Remark~\ref{rem:conformalwug} and
Lemma~\ref{le:invariance} it is not restrictive to assume that
$\mm\in\prob X$, so that we can choose $V\equiv1$. Moreover, we can
assume that
$0<M^{-1}\le f\le M<\infty$ $\mm$-almost everywhere in
$X$ with $\int_X f^2\,\d\mm=1$. By Corollary~\ref{cor:weakrel} we
have to prove that if $f$ is Sobolev on $\calT$-almost every curve
with $\weakgradA f\in L^2(X,\mm)$ then
\begin{equation}
  \label{eq:32}
  \C( f)\le \frac12\int_X \weakgradA f^2\,\d\mm.
\end{equation}
We consider the gradient flow $(h_t)$ of the Cheeger's energy with
initial datum $h:=f^2$,
setting $\mu_t=h_t\mm$, and we apply
Lemma~\ref{le:key}. If $g=h^{-1}\weakgradA h$, we easily get arguing
as in \eqref{eq:76} and using inequality \eqref{eq:lekey}
\begin{align*}
  \int_X &\big(h\log h-h_t \log h_t\big)\,\d\mm\le
  \Big(\int_0^t\int_X  g^2 h_s\,\d\mm\,\d s \Big)^{1/2}\Big(
  \int_0^t|\dot \mu_s|^2\,\d s\Big)^{1/2}
  \\&\le
  \frac 12 \int_0^t\int_X  g^2 h_s\,\d\mm\,\d s+
  \frac 12 \int_0^t |\dot \mu_s|^2\,\d s
  \le
  \frac 12 \int_0^t\int_X  g^2 h_s\,\d\mm\,\d s+
  \frac 12 \int_0^t \int_{\{h_s>0\}}\frac{\relgrad{h_s}^2}{h_s}\,\d\mm\,\d s.
\end{align*}
Recalling the entropy dissipation formula \eqref{eq:eqdissrate} we
obtain
$$
  \int_0^t \int_{\{h_s>0\}}\frac {\relgrad{h_s}^2}{h_s}\,\d\mm\,\d s
  \le \int_0^t\int_X g^2h_s\,\d\mm\,\d s.
$$
Now, \eqref{eq:91} and the identity $g=2f^{-1}\weakgradA f$ give
$\int_0^t\C(\sqrt{h_s})\,\d s\leq\int_0^t\int_X\weakgradA f^2f^{-2}
h_s\,\d\mm\,\d s$, so that dividing by $t$ and passing to the limit
as $t\downarrow0$ we get \eqref{eq:32}, since $\sqrt{h_s}$ are
equibounded and converge strongly to $f$ in $L^2(X,\mm)$ as
$s\downarrow0$.

{\nc Finally we prove the last statement. Thanks to \eqref{eq:quantigradienti} and \eqref{eq:quantisobolev},  
it is sufficient to prove that any $f\in W_w^{1,2}(X,\sfd,\mm)$ belongs to $W_*^{1,2}(X,\sfd,\mm)$ with $\relgrad {f}\leq\weakgradA {f}$ 
$\mm$-a.e. in $X$. Fix such $f$: it follows by \eqref{eq:22} that 
$\relgrad {f_N}\leq\weakgradA {f_N}$ $\mm$-a.e. in $X$ for all truncates of 
$f_N$ of $f$, so that we can pass to the limit as $N\to\infty$ and use the chain rule to
obtain that $f\in D(\C)$ and $\relgrad {f}\leq\weakgradA {f}$ $\mm$-a.e. in $X$.
The statement now follows immediately from \eqref{eq:quantigradienti} of Remark~\ref{rem:compachsh} and
\eqref{eq:22}.}
\end{proof}

{\nc An immediate byproduct of the identification of the different notions of gradients previous result is the density in energy of Lipschitz
functions in $W_w^{1,2}(X,\sfd,\mm)$.}

\begin{theorem}\label{thm:density_energy}
{\nc Let $(X,\tau,\sfd,\mm)$ be a Polish extended measure space with
$\mm$ satisfying \eqref{eq:75}
and let $\calT$ as in Theorem~\ref{thm:rel=weak}.
Let $f\in W_w^{1,2}(X,\sfd,\mm)$. Then there exist Lipschitz functions
$f_n$ such that $f_n\to f$ in $L^2(X,\mm)$ and $|\rmD f_n|\to \weakgradA f$ in $L^2(X,\mm)$.}
\end{theorem}
\begin{proof}
{\nc It is a direct consequence of the previous identification theorem combined with Lemma~\ref{lem:strongchee}(c).}
\end{proof}

\begin{corollary}
  \label{cor:spacca_il_capello_in_quattro}
  Let $\calT_1$ be the collection of all test plans concentrated on
  $\AC2{(0,1)}X\sfd$ with bounded
  compression on the sublevels of $\Wgh$ as in \eqref{eq:incompre2},
  and let $\calT_2$ be the collection of all test plans concentrated
  on $\AC2{(0,1)}X\sfd$ satisfying
  \eqref{eq:incompre}.
  Let us suppose that a measurable function $f:X\to\R$ satisfying \eqref{eq:101}
  is Sobolev on $\calT_1$-almost all
  curves with $|\rmD  f|_{w,\calT_1}\in L^2(X,\mm)$.
  Then $f$ is Sobolev on $\calT_2$-almost all curves and
  \begin{equation}
    \label{eq:99}
    |\rmD  f|_{w,\calT_1}=|\rmD  f|_{w,\calT_2}=\relgrad f\quad
    \mm\text{-a.e.\ in }X.
  \end{equation}
\end{corollary}
\begin{proof}
  Applying Theorem \ref{thm:rel=weak} and Corollary \ref{cor:weakrel}
  we prove that $f$ is Sobolev on $\calT_2$-almost all curves with
  $|\rmD  f|_{w,\calT_1}\ge|\rmD  f|_{w,\calT_2}$.
  The converse inequality follows by Remark \ref{rem:monotone-weakT}.
\end{proof}

\begin{remark} [One-sided relaxed gradients] \label{rem:morecheeger} {\rm
Theorem~\ref{thm:rel=weak} shows that the one-sided Cheeger's
functionals $\C^\pm(f)$ (and the corresponding relaxed gradients
$|\rmD^\pm f|_*$) introduced in Remark~\ref{rem:onesidedcheeger}
coincide with $\C(f)$ (resp.\ with $\relgrad f$). In fact, we
already observed that $\C(f)\ge\C^\pm(f)$. On the other hand, since
$|\rmD^\pm f|$ are weak upper gradients for Borel $\sfd$-Lipschitz
functions, arguing as in Corollary~\ref{cor:weakrel} we get
$\weakgradA f\le |\rmD^\pm f|_*$, if $f\in D(\C^\pm)$ and $\calT$
is the class of all test plans concentrated on $\AC2{(0,1)}X\sfd$
with bounded compression on the sublevels of $\Wgh$.
 Corollary~\ref{cor:spacca_il_capello_in_quattro}
yields
\begin{displaymath}
  \C(f)=\C^\pm(f),\qquad\relgrad f=|\rmD^\pm f|_*\quad\text{$\mm$-a.e. in
  $X$.}
\end{displaymath}
}\end{remark}

\section{Relative entropy, Wasserstein slope, and Fisher information}
\label{sec:gflowrele}

In this section we assume that $(X,\tau,\sfd)$ is a Polish extended
space equipped with a $\sigma$-finite Borel reference measure $\mm$
such that $\tmm:=e^{-\Wgh^2}\mm$ has total mass less than 1 for
some Borel and $\sfd$-Lipschitz $\Wgh:X\to [0,\infty)$. We shall
work in the subspace
\begin{equation}\label{eq:probabilitiesWgh}
{\mathscr P}_\Wgh(X):=\left\{\mu\in\Probabilities{X}:\ \int_X
\Wgh^2\,\d\mu<\infty\right\}.
\end{equation}
We say that $(\mu_n)\subset{\mathscr P}_\Wgh(X)$ \emph{weakly
converges with moments} to $\mu\in{\mathscr P}_\Wgh(X)$ if
$\mu_n\to\mu$ weakly in $\Probabilities{X}$ and
$\int_X\Wgh^2\,\d\mu_n\to\int_X\Wgh^2\,\d\mu$. Analogously we define
strong convergence with moments in $\Probabilities{X}$, by requiring
that $|\mu_n-\mu|(X)\to 0$, instead of the weak convergence.

Since for every $\mu\in {\mathscr P}_\Wgh(X)$, $\nu\in\prob X$ with
$W_2(\mu,\nu)<\infty$ we have $\nu\in {\mathscr P}_\Wgh(X)$ and
\begin{equation}
  \label{eq:86}
  \Big(\int_X\Wgh^2\,\d\nu\Big)^{1/2}\leq
  {\rm
  Lip}(\Wgh)W_2(\nu,\mu)+\Big(\int_X\Wgh^2\,\d\mu\Big)^{1/2},
\end{equation}
we obtain that weak convergence with moments is implied by $W_2$
convergence. When the topology $\tau$ is induced by the distance
$\sfd$ and $V(x):= A\,\sfd(x,x_0)$ for some $A>0$ and $x_0\in X$,
then weak convergence with moments is in fact equivalent to $W_2$
convergence. When $\mm(X)<\infty$ we may take $\Wgh$ equal to a
constant, so that $\Probabilities{X}={\mathscr P}_\Wgh(X)$ and weak
convergence with moments reduces to weak convergence.

\subsection{Relative entropy}

\begin{definition}[Relative entropy]\label{def:relentr}
The relative entropy functional $\entv:\prob X\to (-\infty,+\infty]$
is defined as
\[
\ent\mu:=\left\{
\begin{array}{ll}
\displaystyle{\int_X\rho\log\rho\,\d\mm}&\textrm{ if }\mu=\rho\mm\in
{\mathscr P}_\Wgh(X),\\
+\infty&\textrm{ otherwise.}
\end{array}
\right.
\]
\end{definition}

Notice that, according to our definition, $\mu\in D(\entv)$ implies
$\int_X\Wgh^2\,\d\mu<\infty$, and that $D(\entv)$ is convex.
Strictly speaking the notation $\entv$ is a slight abuse, since the
functional depends also on the choice of $\Wgh$, which is not
canonically induced by $\mm$ (not even in Euclidean spaces endowed
with the Lebesgue measure). It is tacitly understood that we take
$\Wgh$ equal to a constant whenever $\mm(X)<\infty$ and in this case
$\entv$ is independent of the chosen constant.

When $\mm\in\prob X$ the functional $\entv$ is sequentially lower
semicontinuous w.r.t. weak convergence in $\Probabilities{X}$. In
addition, it is nonnegative, thanks to Jensen's inequality. More
generally, if $\mm$ is a finite measure and
$\bar\mm:=\mm(X)^{-1}\mm$,
\begin{equation}
  \label{eq:50}
  \ent\mu=\RE\mu{\bar \mm}-\log(\mm(X))\ge -\log(\mm(X))\quad\text{for every }\mu\in \prob X,
\end{equation}
and we have the general inequality (see for instance
\cite[Lemma~9.4.5]{Ambrosio-Gigli-Savare08})
\begin{equation}
  \label{eq:73}
  \RE{\pi_\sharp\mu}{\pi_\sharp\mm}\le \ent\mu\quad
  \text{for every }\mu\in \prob X\quad
  \text{and }\pi:X\to Y \text{ Borel map},
\end{equation}
which turns out to be an equality if $\pi$ is injective.
When $\mm(X)=\infty$, since the density $\tilde\rho$ of $\mu$
w.r.t.~$\tmm$ equals $\rho\,\rme^{\Wgh^2}$ we obtain that the negative part
of $\rho\log\rho$ is $L^1(\mm)$-integrable for $\mu\in {\mathscr
P}_V(X)$, so that Definition~\ref{def:relentr} is well posed and
$\entv$ does not attain the value $-\infty$. We also obtain the
useful formula
\begin{equation}
  \label{eq:60}
  \ent\mu=\RE{\mu}{\tmm}-\int_X \Wgh^2\,\d\mu
  \qquad\forall\mu\in {\mathscr P}_\Wgh(X).
\end{equation}
The same formula shows $\entv$ is sequentially lower semicontinuous
in ${\mathscr P}_\Wgh(X)$ w.r.t. convergence with moments, i.e.
\begin{equation}
  \label{eq:3}
  \mu_n\weakto\mu\text{ in }\prob X,\,\,
  \int_X\Wgh^2\,\d\mu_n\to\int_X\Wgh^2\,\d\mu<\infty\quad
  \Longrightarrow\quad
  \liminf_{n\to\infty}\ent{\mu_n}\ge \ent\mu.
\end{equation}
From \eqref{eq:86} we also get
\begin{equation}
  \label{eq:79}
  \mu\in{\mathscr P}_\Wgh(X),\quad
  W_2(\mu_n,\mu)\to0\quad\Longrightarrow\quad
  \liminf_{n\to\infty}\ent{\mu_n}\ge \ent\mu.
\end{equation}
The following lemma for the change of reference measure in the
entropy, related to \eqref{eq:60}, will be useful.

\begin{lemma}[Change of reference measure in the
entropy]\label{lem:changeref} Let $\nu\in \prob X$ and the positive
finite measure $\nn$ be satisfying $\RE\nu\nn<\infty$. If $\nu=g\mm$
for some $\sigma$-finite Borel measure $\mm$, then $g\log g\in
L^1(X,\mm)$ if and only if $\log(\d\nn/\d\mm)\in L^1(X,\nu)$ and
\begin{equation}
  \label{eq:69}
  \RE\nu\mm=\RE{\nu}\nn+\int_X
  \log\Big(\frac{\d\nn}{\d\mm}\Big)\,\d\nu.
\end{equation}
\end{lemma}
\begin{proof} Write $\nu=f\nn$ and let
$\nn=h\mm+\nn^s$ be the Radon-Nikod\'ym decomposition of
$\nn$ w.r.t. $\mm$. Since $g\mm=\nu=fh\mm+f\nn^s$ we obtain
that $g=fh$ $\mm$-a.e. in $X$ and $f=0$ $\nn^s$-a.e. in $X$.
Since $f\log f\in L^1(X,\nn)$ we obtain that
$\nchi_{\{h>0\}}g\log(g/h)$ belongs to $L^1(X,\mm)$, so that (taking
into account that $\{g>0\}\subset \{h>0\}$ up to $\mm$-negligible
sets) $g\log g\in L^1(X,\mm)$ if and only if $g\log h\in
L^1(X,\mm)$. The latter property is equivalent to $\log h\in
L^1(X,\nu)$.
\end{proof}

\begin{remark}[Tightness of sublevels of $\ent\mu$ and setwise
convergence]\label{re:tightsub}
  {\upshape
 We remark that the sublevels of the relative entropy functional are
tight if $\mm(X)<\infty$. Indeed, by Ulam's theorem $\mm$ is tight.
Then, using first the inequality $z\log(z)\geq -e^{-1}$ and then
Jensen's inequality, for $\mu=\rho\mm$ we get
\begin{equation}\label{eq:Jensenloca}
  \frac{\mm(X)}e+C\geq \frac{\mm(X\setminus E)}e+\ent{\mu}\geq
\int_E\rho\log\rho\,\d\mm\geq\mu(E)
\log\left(\frac{\mu(E)}{\mm(E)}\right)
\end{equation}
whenever $E\in\BorelSets{X}$ and $\ent{\mu}\leq C$. This shows that
$\mu(E)\to 0$ as $\mm(E)\to 0$ uniformly in the set $\{\entv\leq
C\}$.

In general, when $\int_X \rme^{-\Wgh^2}\,\d\mm\le 1$, we see that
\eqref{eq:60} yields
\begin{equation}
  \label{eq:59}
  \Big\{\mu\in \prob X:\int_X\Wgh^2\,\d\mu+\ent\mu\le
  C\Big\}\quad\text{is tight in $\prob X$ for every $C\in\R$.}
\end{equation}
Moreover, if a sequence $(\mu_n)$ belongs to a sublevel
\eqref{eq:59} and weakly converges to $\mu$, then the sequence of
the corresponding densities $\rho_n=\frac{\d\mu_n}{\d\mm}$ converges
to $\rho=\frac{\d\mu}{\d\mm}$ weakly in $L^1(X,\mm)$: 
it is sufficient to recall
\eqref{eq:50} and to apply de la Vall\'ee Puissen's criterion for
uniform integrability \cite[\S 4.5.10]{Bogachev07} to the densities
{\GGG $\tilde\rho_n$ (resp.~$\tilde\rho$)}
of $\mu_n$ (resp.~$\mu$) w.r.t.\ the finite measure $\tmm=\rme^{-\Wgh^2}\mm$,
{\GGG since for every $\varphi\in 
L^\infty(X,\tmm)=L^\infty(X,\mm)$
  \begin{displaymath}
    \int_X \rho_n\varphi\,\d\mm=
    \int_X \tilde\rho_n\varphi\,\d\tmm\to
    \int_X \tilde\rho\varphi\,\d\tmm=
    \int_X \rho\varphi\,\d\mm\quad\text{as }n\to\infty.
  \end{displaymath}
}
In particular the sequence $(\mu_n)$ setwise converges to $\mu$, i.e.
$\mu_n(B)\to\mu(B)$ for every $B\in\BorelSets{X}$.}\fr
\end{remark}

\subsection{Entropy dissipation, slope and Fisher information}

In this subsection we collect some general properties of the
relative entropy, its Wasserstein slope and the Fisher information
functional defined via the relaxed gradient that we introduced in the
previous section.

We will always assume that $\mm$ satisfies condition \eqref{eq:75},
so that Theorem~\ref{thm:infinitemass} will be applicable.

\begin{theorem}\label{thm:lowerboudslope} Let
$\mu=\rho\mm\in D(\entv)$ with $|\rmD^-\entv|(\mu)<\infty$. Then
$\sqrt\rho\in D(\C)$ and
\begin{equation}
  \label{eq:44pre}
  4\int_X \relgrad {\sqrt \rho}^2\,\d\mm\le |\rmD^-\entv|^2(\mu).
\end{equation}
\end{theorem}
\begin{proof}
  Let us first assume that $\rho\in L^2(X,\mm)$ and let $(\rho_t)$ be the
  gradient flow of the Cheeger's functional starting from $\rho$; we set
  $\mu_t:=\rho_t\mm$ and recall the definition \ref{def:Fisher}  of
  Fisher information functional $\mathsf F$.
  Applying Proposition~\ref{le:diss} and Lemma~\ref{le:key} we get
  \begin{align}
    \label{eq:38}
    \REG\mu&-\REG{\mu_t}\ge \frac 12 \int_0^t \mathsf F(\rho_s)\,\d s+\frac 12
    \int_0^t |\dot\mu_s|^2\,\d s
    \\
    \notag&\ge \frac 12 \Big(\frac 1{\sqrt t}\int_0^t
    \sqrt{\mathsf F(\rho_s)}\,\d s\Big)^{2}+\frac 1{2}\Big(\frac 1{\sqrt t}\int_0^t |\dot
    \mu_s|\,\d s\Big)^2
    \ge \frac 1t\Big(\int_0^t
    \sqrt{\mathsf F(\rho_s)}\,\d s\Big)W_2(\mu,\mu_t).
  \end{align}
  Dividing by $W_2(\mu,\mu_t)$ and passing to the limit as
  $t\downarrow0$ we get
  \eqref{eq:44pre}, since the lower semicontinuity of Cheeger's
  functional yields
  \begin{displaymath}
    \sqrt{\mathsf F(\rho)}\le\liminf_{t\downarrow0}\frac 1t\int_0^t
    \sqrt{\mathsf F(\rho_s)}\,\d s.
  \end{displaymath}
In the general case when only the integrability conditions
$\int_X\rho\log\rho\,\d\mm<\infty$ and
$\int_X\Wgh^2\rho\,\d\mm<\infty$ are available, we can still prove
\eqref{eq:38} by approximation. We set {\nc $\rho^n:=\min\{\rho,n\}$,
$z_n:=\int_X\rho^n\,\d\mm\uparrow 1$}, and denote by $\rho^n_t$
the gradient flows of Cheeger's energy starting from $\rho^n$. Since
Theorem~\ref{prop:maxprin}(\ref{maxprin:c}) provides the monotonicity property
$\rho^n_t\leq \rho^m_t$ $\mm$-a.e. in $X$ for $n\leq m$, we can
define $\rho_t:=\sup_n \rho^n_t$. Since $\int_X \rho^n_t\,\d\mm=z_n$
it is immediate to check that $\rho_t\mm\in\Probabilities{X}$ and a
simple monotonicity argument based on the apriori estimate
\eqref{eq:34} guaranteed by Theorem~\ref{thm:infinitemass} also
gives that $\mu_t:=\rho_t\mm\in {\mathscr P}_\Wgh(X)$ and that
$z_n^{-1}\rho^n_t\mm$ converge with moments to $\mu_t$. It is then
easy to pass to the limit in \eqref{eq:38}, using the sequential
lower semicontinuity of entropy with respect to convergence with
moments, to get
$$
\REG\mu-\REG{\mu_t}\geq\frac 1t\Big(\int_0^t
    \sqrt{\mathsf F(\rho_s)}\,\d s\Big)W_2(\mu,\mu_t)
$$
and then conclude as before.
\end{proof}
\
\begin{theorem}\label{thm:slopefishersempre} Let
$\mu=\rho\mm\in D(\entv)$. Assume that
$\rho=\max\{\rho_0,ce^{-2\Wgh^2}\}$, where $c>0$ and $\rho_0$ is a
$\sfd$-Lipschitz and bounded map identically 0 for $\Wgh$
sufficiently large.\\ Then
\begin{equation}
\label{eq:44} |\rmD^-\entv|^2(\mu)\le \int_X \frac{|\rmD^-
\rho|^2}{\rho}\,\d\mm=4\int_X |\rmD^-\sqrt \rho|^2\,\d\mm.
\end{equation}
\end{theorem}
\begin{proof} We set $L={\rm Lip}(V)$, $M:=\sup\rho_0$ and choose $C\geq 0$ in such a way that $\rho_0=0$ on
$\{2\Wgh^2>C\}$. Possibly multiplying $\rho$ and $\mm$ by constants we assume ${\rm Lip}(\rho_0)=1$.
Let us introduce the nonnegative $\BorelSetsStar{X\times
X}$-measurable function
\begin{equation}
\label{eq:43}
L(x,y):=
\begin{cases}
\dfrac{\big(\log\rho(x)-\log\rho(y)\big)^+}{\sfd(x,y)}&\text{if }x\neq y,\\
&\\
\dfrac{|\rmD^- \rho(x)|}{\rho(x)}&\text{if }x=y,
\end{cases}
\end{equation}
and notice that for every $x\in X$, the map $y\mapsto L(x,y)$ is $\sfd$-upper semicontinuous.

We claim that for some constants $C',\,C''$ depending only on $M$,
$c$ and $C$ it holds
\begin{equation}
\label{eq:Llineare} L(x,y)\leq C'+
C''\big(\sfd(x,y)+\Wgh(x)\big),\qquad\forall x,\,y\in X.
\end{equation}
To prove this, let $A:=\{\rho_0>ce^{-2\Wgh^2}\}$ and notice that
$\log\rho\geq\log c-2\Wgh^2$ gives
\begin{equation}
\label{eq:perllin1} \big(\log\rho(x)-\log\rho(y)\big)^+\leq \left\{
\begin{array}{ll}
|\log(\rho_0(x))-\log(\rho_0(y))|,&\qquad\textrm{ if }x,\,y\in A,\\
2|\Wgh^2(x)-\Wgh^2(y)|,&\qquad\textrm{ if }x\notin A,\\
\big(\log(\rho_0(x))+2\Wgh^2(y)-\log c\big)^+,&\qquad\textrm{ if
}x\in A,\ y\notin A.
\end{array}
\right.
\end{equation}
Since $2\Wgh^2\le C$ on $A$, the function $\rho_0$ is
$\sfd$-Lipschitz and bounded from below by $c\,\rme^{-C}$ on $A$, so
that
\begin{equation}
\label{eq:perllin2} |\log(\rho_0(x))-\log(\rho_0(y))|\leq
\frac{\rme^{C}}{c}|\rho_0(x)-\rho_0(y)|\leq\frac{\rme^{C}}c\sfd(x,y)\qquad \forall x,\,y\in A.
\end{equation}
Also, for all $x,\,y\in X$ it holds
\begin{equation}
\label{eq:perllin3}
|\Wgh^2(x)-\Wgh^2(y)|=|\Wgh(x)-\Wgh(y)||\Wgh(x)+\Wgh(y)|\leq
L\sfd(x,y)\big(L\sfd(x,y)+2\Wgh(x)\big).
\end{equation}
Finally, let us consider the case $x\in A$, $y\notin A$; since
$2\Wgh^2(y)-\log c\leq -\log\rho_0(y)$ and $\rho_0(y)\geq
c\,\rme^{-C}/2$ if $\sfd(x,y)\leq\bar a:= c\,\rme^{-C}/2$, we get
\begin{equation}
\label{eq:perllin4} \log(\rho_0(x))+2\Wgh^2(y)-\log c\leq\log(\rho_0(x))-\log(\rho_0(y))\leq
\frac{2\rme^C}{c}\sfd(x,y)
\end{equation}
for $\sfd(x,y)\leq\bar a$. If, instead, $\sfd(x,y)>\bar a$ we use
the fact that $\rho$ is bounded from above, the bound $2\Wgh^2(x)\le
C$ for $x\in A$, and \eqref{eq:perllin3} to get
\begin{eqnarray}
\label{eq:perllin5}
\log(\rho_0(x))+2\Wgh^2(y)-\log c&=&\log(\rho_0(x))+2\Wgh^2(x)-\log c+2\big(\Wgh^2(y)-\Wgh^2(x)\big)\\
&\leq&\frac{\sfd(x,y)}{\bar a}\bigl(\log
(M/c)+C\bigr)+2L\sfd(x,y)\big(L\sfd(x,y)+2\Wgh(x)\big).\nonumber
\end{eqnarray}
Inequalities \eqref{eq:perllin1}, \eqref{eq:perllin2}, \eqref{eq:perllin3},
\eqref{eq:perllin4}, \eqref{eq:perllin5} give the claim \eqref{eq:Llineare}.

Let us now consider a sequence $(\rho_n\mm)$ such that $W_2(\rho_n\mm,\mu)\to 0$ and
\begin{displaymath}
|\rmD^-\entv|(\mu)=\lim_{n\to\infty}\frac{\ent{\mu}-\ent{\mu_n}}{W_2(\mu,\mu_n)}.
\end{displaymath}
From the convexity of the function $r\mapsto r\log r$ we have
\begin{align*}
\ent{\mu}-\ent{\mu_n}&=
\int_X\Big(\rho\log\rho-\rho_n\log\rho_n\Big)\,\d\mm\le
\int_X\log\rho\Big(\rho-\rho_n\Big)\,\d\mm
\\&=\int_X\log\rho\,\d\mu-\int_X\log\rho\,\d\mu_n=
\int_{X\times X}\Big(\log\rho(x)-\log\rho(y)\Big)\,\d\ggamma_n
\\&\leq\int_{X\times X}
L(x,y)\sfd(x,y)\,\d\ggamma_n\le W_2(\mu,\mu_n) \Big(\int_{X\times X}
L^2(x,y)\,\d\ggamma_n\Big)^{1/2}
\\&= W_2(\mu,\mu_n)
\Big(\int_{X}\Big(\int_{X}
L^2(x,y)\,\d\ggamma_{n,x}\Big)\,\d\mu(x)\Big)^{1/2},
\end{align*}
where $\ggamma_n$ is any optimal plan between $\mu$ and $\mu_n$ and
$\ggamma_{n,x}$ is its disintegration w.r.t.\ its first marginal
$\mu$. Since $\int_X\bigl(\int_X
\sfd^2(x,y)\,\d\ggamma_{n,x}(y)\bigr)\,\d\mu(x)\to 0$ as
$n\to\infty$ we can assume with no loss of generality that
\begin{displaymath}
\lim_{n\to\infty}\int_X \sfd^2(x,y)\,\d\ggamma_{n,x}(y)=0\quad\text{for
$\mu$-a.e.\ $x\in X$},
\end{displaymath}
thus taking into account \eqref{eq:Llineare} we get
\[
\int_{X\setminus B_r(x)}L^2(x,y)\,\d\ggamma_{n,x}(y)\to 0\quad\text{for
$\mu$-a.e.\ $x\in X$},
\]
for all $r>0$. Taking
an arbitrary radius $r>0$ we get
\begin{align*}
\limsup_{n\to\infty}\int_{X} L^2(x,y)\,\d\ggamma_{n,x} &\le
\limsup_{n\to\infty}\int_{B_r(x)}L^2(x,y)\,\d\ggamma_{n,x}+
\limsup_{n\to\infty}\int_{X\setminus B_r(x)} L^2(x,y)\,\d\ggamma_{n,x}
\\&\le
\limsup_{n\to\infty}\int_{B_r(x)}L^2(x,y)\,\d\ggamma_{n,x}\le \sup_{y\in
B_r(x)}L^2(x,y).
\end{align*}
Since $L(x,\cdot)$ is $\sfd$-upper semicontinuous, taking the limit
as $r\downarrow0$ in the previous estimate we get $\limsup_n\int_X
L^2(x,y)\,\d\ggamma_{n,x}\le L^2(x,x)$ for $\mu$-a.e. $x\in X$.
Using again \eqref{eq:Llineare}, which provides a domination from
above with a strongly convergent sequence, we are entitled to use
Fatou's lemma to obtain
\begin{align*}
|\rmD^-\entv|(\mu)&=\lim_{n\to\infty}\frac{\ent{\mu}-\ent{\mu_n}}{W_2(\mu,\mu_n)}
\le\int_X \limsup_{n\to\infty}\Big(\int_X
L^2(x,y)\,\d\ggamma_{n,x}(y) \Big)^{1/2}\,\d\mu(x)
\\&\le \Big(\int_X L^2(x,x)\,\d\mu(x)\Big)^{1/2}.
\end{align*}
\end{proof}

\begin{theorem}\label{thm:slopefisherconv}
Let $(X,\tau,d,\mm)$ be a Polish extended space with $\mm$
satisfying \eqref{eq:75}. Then $|\rmD^-\entv|$ is sequentially
lower semicontinuous w.r.t.~strong convergence with moments in
$\Probabilities{X}$ on sublevels of $\entv$ if and only if
\begin{equation}\label{eq:48}
|\rmD^-\entv|^2(\mu)= 4\int_X\relgrad{\sqrt\rho}^2\,\d\mm\qquad
\forall\mu=\rho\mm\in D(\entv).
\end{equation}
In this case $|\rmD^-\entv|$ satisfies the following stronger
lower semicontinuity property:
\begin{equation}
  \label{eq:94}
  \mu_n(B)\to \mu(B)\ \text{for every $B\in \BorelSets X$}\quad
  \Longrightarrow\quad
  \liminf_{n\to\infty}|\rmD^-\entv|(\mu_n)\ge |\rmD^-\entv|(\mu).
\end{equation}
\end{theorem}
\begin{proof}
If \eqref{eq:48} holds then $|\rmD^-\entv|$ coincides on its
domain with a convex functional (by Lemma \ref{le:Fisher}) which is
lower semicontinuous with respect to strong convergence in
$L^1(X,\mm)$: therefore it is also $L^1$-weakly lower semicontinuous and
\eqref{eq:94} holds \cite[\S 4.7(v)]{Bogachev07}. In particular
$|\rmD^-\entv|$ is sequentially lower semicontinuous with respect
to strong convergence with moments.

To prove the converse implication, by Theorem~\ref{thm:lowerboudslope} it is sufficient to prove the inequality
$|\rmD^-\entv|^2(\mu)\le 4\int_X\relgrad{\sqrt\rho}^2\,\d\mm $.
Assume first that $\rho\leq M^2$ $\mm$-a.e. in $X$ for some $M\in
[0,\infty)$. Taking Theorem~\ref{thm:slopefishersempre} into
account, it suffices to find a sequence of functions
$\rho_m=\max\{f_m^2,c_m\rme^{-2\Wgh^2}\}$ convergent to $\rho$ in
$L^1(X,\mm)$ and satisfying:
\begin{itemize}
\item[(a)] $f_m$ is $\sfd$-Lipschitz, nonnegative, bounded from above
by $M$ and null for $\Wgh$ sufficiently large;
\item[(b)]
$\limsup_{n\to\infty}\tfrac12\int_X|\rmD \sqrt{\rho_n}|^2\,\d\mm\to\C(\sqrt{\rho})$;
\item[(c)] $\int_X\Wgh^2\rho_n\,\d\mm\to\int_X\Wgh^2\rho\,\d\mm$.
\end{itemize}
Since the $\sfd$-Lipschitz property of the weight implies
$\rme^{-\Wgh^2}\in W^{1,2}(X,\sfd,\mm)$ and
$\int\Wgh^2\rme^{-2\Wgh^2}\,\d\mm<\infty$, if we choose $c_m>0$
infinitesimal it suffices to find $f_m$ satisfying (a),
$\tfrac12\int_X|\rmD  f_m|^2\,\d\mm\to\C(\sqrt{\rho})$ and
$\int_X\Wgh^2f_m^2\,\d\mm\to\int_X\Wgh^2\rho\,\d\mm$.

To this aim, given $m>0$, we fix a compact set $K\subset X$ such
that $\int_{X\setminus K}\rho\,\d\mm<(1+m)^{-2}$ and a $1$-Lipschitz
function $\phi:X\to [0,1]$ equal to 1 on $K$ and equal to $0$ out of
the $1$-neighbourhood of $K$, denoted by $\tilde{K}$. Notice that
$\mm(\tilde{K})<\infty$, since $\Wgh$ is bounded from above in
$\tilde{K}$.

Let now $(g_n)$ be a sequence of $\sfd$-Lipschitz functions
convergent to $\sqrt{\rho}$ in $L^2(X,\mm)$ and satisfying
$\tfrac12\int_X|\rmD  g_n|^2\,\d\mm\to\C(\sqrt{\rho})$; by a simple
truncation argument we can assume that all $g_n$ satisfy $0\leq
g_n\leq M$. The bounded, nonnegative, $\sfd$-Lipschitz functions
$g_n\phi$ converge in $L^2(X,\mm)$ to $\sqrt{\rho}\phi$ and, thanks
to the inequality $\relgrad\phi\leq\nchi_{X\setminus K}$ $\mm$-a.e.
in $X$ {\GGG and to \eqref{eq:weak-leibn}, they} satisfy
$$
\limsup_{n\to\infty}\C(g_n\phi)\leq
(1+\frac{1}{m})\C(\sqrt{\rho})+(1+m)\int_{X\setminus K}\rho\,\d\mm
\leq (1+\frac{1}{m})\C(\sqrt{\rho})+\frac{1}{1+m}.
$$
In addition, $g_n^2\phi^2\to\rho\phi^2$ in $L^1(X,\mm)$ because the
functions vanish out of $\tilde{K}$. We conclude that we have also
$$
\lim_{n\to\infty}\int_X
\Wgh^2g_n^2\phi^2\,\d\mm=\int_X\Wgh^2\rho\phi^2\,\d\mm\leq
\int_X\Wgh^2\rho\,\d\mm.
$$
By a diagonal argument, choosing $f_m=g_n$ with $n=n(m)$
sufficiently large, the existence of a sequence $f_m$ with the
stated properties is proved.

In the case when $\rho$ is not bounded we truncate $\rho$, without
increasing its Cheeger's energy, and use once more the lower
semicontinuity of $|\rmD^-\entv|$.
\end{proof}

\subsection{Convexity of the squared slope}\label{sec:steepest}

This subsection adapts and extends some ideas extracted from
\cite{Gigli10} to the more general framework considered in this
paper. The main result of the section shows that the squared
Wasserstein slope of the entropy is always convex (with respect to
the linear structure in the space of measures), independently from
the identification with the Fisher information considered in
Theorem~\ref{thm:slopefisherconv} (the identification therein relies
on the assumption that $|\rmD^-\entv|$ is sequentially lower
semicontinuous with respect to strong convergence with moments).

Let us first introduce the notion of \emph{push forward of a measure
through a transport plan}: given $\ggamma\in \prob{X\times X}$ with
marginals $\gamma^i=\pi^i_\sharp \ggamma$ and given
$\mu\in\Probabilities{X}$ we set
\begin{equation}
  \label{eq:64}
  \ggamma_\mu:=(\rho\circ\pi^1)\ggamma\quad\text{with}
  \quad\mu=\rho{\GGG\gamma^1}+\mu^s,\,\,\mu^s\perp{\GGG\gamma^1},\qquad
  \PushPlan\ggamma\mu:=\pi^2_\sharp\ggamma_\mu.
\end{equation}
We recall that this construction first appeared, with a different
notation, in Sturm's paper \cite{Sturm06I}. Notice that
$\ggamma_\mu$ is a probability measure and
$\pi^1_\sharp\ggamma_\mu=\mu$ if $\mu\ll\gamma^1$; in this case, if
$(\ggamma_x)_{x\in X}$ is the disintegration of $\ggamma$ with
respect to its first marginal $\gamma^1$, we have
\begin{equation}
  \label{eq:65}
  \PushPlan\ggamma\mu(B)=\int_X \ggamma_x(B)\,\d\mu(x)\quad
  \text{for every }B\in\BorelSets X.
\end{equation}
Since moreover $\ggamma_\mu\ll\ggamma$ we also have that
$\PushPlan\ggamma\mu\ll\gamma^2$.

Notice that
\begin{equation}
  \label{eq:66}
  \ggamma=\int_X
  \delta_{\srr(x)}\,\d\nu(x),\quad\mu\ll\nu\qquad\Longrightarrow\qquad
  \ggamma_\sharp \mu=\rr_\sharp\mu.
\end{equation}
In the next lemma we consider the real-valued map
\begin{equation}
 \label{eq:67}
  \mu\mapsto G_\sggamma(\mu):=\ent\mu-\ent{\PushPlan\ggamma\mu},
\end{equation}
  defined in the convex set
  \begin{equation}
    \label{eq:74}
    R_\sggamma:=\big\{\mu\in \prob X:\mu\ll\pi^1_\sharp\ggamma,\
    \mu,\PushPlan\ggamma\mu\in D(\entv)\big\}.
  \end{equation}
In the simple case when $\ggamma=\int_X
  \delta_{\srr(x)}\,\d\mm(x)$ with $\rr:X\to X$ Borel bijection
  we may use first the representation formula \eqref{eq:66} for
  $\PushPlan\ggamma\mu$ and then \eqref{eq:73} to obtain
  $$\RE{\PushPlan\ggamma\mu}{\mm}=\RE{\rr_\sharp\mu}{\mm}=\RE{\mu}{\mm'}$$
  with $\mm':=(\rr^{-1})_\sharp\mm$.
  Since $\mu\in R_\sggamma$ we have $\RE{\mu}{\mm'}<\infty$ and
  we can use \eqref{eq:69} for
  the change of reference measure in the relative entropy to get
  $$
  \ent\mu-\ent{\PushPlan\ggamma\mu}=\int_X\log\bigl(\frac{\d\mm'}{\d\mm}\bigr)\,\d\mu,
  $$
  so that $G_\sggamma$ is linear w.r.t. $\mu$. In general, when
  $\rr$ is not injective or $\rr$ is multivalued, convexity persists:
\begin{lemma}
  \label{le:pre} For every $\ggamma\in \prob{X\times X}$ the map
$G_\sggamma$ in \eqref{eq:67} is convex in $R_\sggamma$.
\end{lemma}
\begin{proof} Let $\mu_1=\rho_1\mm,\,\mu_2=\rho_2\mm\in R_\sggamma$, set $\mu=\alpha_1\mu_1+\alpha_2\mu_2$ with
$\alpha_1+\alpha_2=1$, $\alpha_1,\,\alpha_2\in (0,1)$ and denote by $\theta_i\leq 1/\alpha_i$
the densities of $\mu_i$ w.r.t. $\mu$.

We apply \eqref{eq:69} of Lemma~\ref{lem:changeref} with $\nu:=\mu_i$
and $\nn:=\mu$ to get
$$
  \ent{\mu_i}=\RE{\mu_i}\mu+\int_X \log \rho\,\d\mu_i,
$$
where $\rho=\alpha_1\rho_1+\alpha_2\rho_2$ is the density of $\mu$
w.r.t. $\mm$. Taking a convex combination of the previous equalities
for $i=1,\,2$, we obtain
\begin{equation}
  \label{eq:70}
  \alpha_1\ent{\mu_1}+\alpha_2\ent{\mu_2}=
  \alpha_1\RE{\mu_1}{\mu}+\alpha_2\RE{\mu_2}{\mu}+\ent{\mu}.
\end{equation}
Analogously, setting $\nu_i:=\PushPlan\ggamma{\mu_i}$ and
$\nu:=\PushPlan\ggamma\mu=\alpha_1\nu_1+\alpha_2\nu_2$, we have
\begin{equation}
\label{eq:71}
  \alpha_1\REG{\nu_1}+\alpha_2\REG{\nu_2}=
  \alpha_1\RE{\nu_1}\nu+\alpha_2\RE{\nu_2}\nu+\REG\nu.
\end{equation}
Combining \eqref{eq:70} and \eqref{eq:71} we obtain
\begin{equation}
\label{eq:72}
\alpha_1G_\sggamma(\mu_1)+\alpha_2 G_\sggamma(\mu_2)=
G_\sggamma(\mu)+\sum_{i=1,2} \alpha_i\Big(\RE{\mu_i}\mu-\RE{\nu_i}{\nu}\Big).
\end{equation}
Since $\nu_i=\pi^2_\sharp (\ggamma_{\mu_i})$ and
$\nu=\pi^2_\sharp(\ggamma_\mu)$, \eqref{eq:73} yields
$$\RE{\nu_i}\nu\le \RE{\ggamma_{\mu_i}}{\sggamma_\mu}=\RE{\theta_i\ggamma_\mu}{\sggamma_\mu}=
\RE{\mu_i}{\mu},$$ where in the last equality we used that the first
marginal of $\ggamma_\mu$ is $\mu$. Therefore \eqref{eq:72}  yields
$\alpha_1G_\sggamma(\mu_1)+\alpha_2 G_\sggamma(\mu_2)\ge
G_\sggamma(\mu)$.
\end{proof}

\begin{theorem}
  \label{thm:slope_convex}
  The squared descending slope $|\rmD^- \entv|^2$ of the
  relative entropy is convex.
\end{theorem}
\begin{proof}
  Let $\mu_1,\,\mu_2\in D(\entv)$ be measures with finite descending slope
  and let $\mu=\alpha_1\mu_1+\alpha_2\mu_2$ with $\alpha_1,\alpha_2\in
  (0,1),\ \alpha_1+\alpha_2=1$. Obviously $\mu\in D(\entv)$ and
  since it is not restrictive to assume $|\rmD^-\entv|(\mu)>0$,
  by definition of descending slope we can find a sequence
  $(\nu^n)\subset D(\entv)$ with $\ent{\nu^n}\le \ent\mu$ such that
  $$
    W_2(\nu^n,\mu)\to 0,\qquad
    \frac{\ent\mu-\ent{\nu^n}}{W_2(\mu,\nu^n)}\to |\rmD^-\entv|(\mu).
  $$
  Let $\ggamma^n$ be optimal plans with marginals $\mu$ and $\nu^n$
  respectively and let $\nu_i^n:=\PushPlan{\ggamma^n}{\mu_i}$.
  Since $\ggamma^n_{\mu_i}$ are {\nc admissible plans from
  $\mu_i$ to $\nu^n_i$}, we have
  \begin{displaymath}
    W^2_2(\mu_i,\nu_i^n)\leq\int_{X\times
      X}\sfd^2(x,y)\theta_i(x)\,\d\ggamma^n(x,y)\to0\quad\text{as }n\to\infty,
  \end{displaymath}
  where $\theta_i\leq\alpha^{-1}_i$ are the densities of $\mu_i$ w.r.t. $\mu$. {\nc Since
  $\alpha_1\theta_1+\alpha_2\theta_2=1$, multiplying by
  $\alpha_i$ and adding the two inequalities, the joint convexity of $W_2^2$ yields
  \begin{equation}
    W^2_2(\mu,\nu^n)\leq \alpha_1W_2^2(\mu_1,\nu_1^n)+\alpha_2 W_2^2(\mu_2,\nu_2^n)\leq
    \int_{X\times X}\sfd^2\,\d\ggamma^n=W_2^2(\mu,\nu^n),
  \end{equation} 
  so that}
  \begin{equation}
    W^2_2(\mu,\nu^n)=\alpha_1W_2^2(\mu_1,\nu_1^n)+\alpha_2 W_2^2(\mu_2,\nu_2^n).\label{eq:80}
  \end{equation}
  Since $|\rmD^-\entv|(\mu_i)<\infty$,
  by the very definition of descending slope for every
$S_i>|\rmD^-\entv|(\mu_i)$
  there exists $\bar n\in\N$ satisfying
  $$
    \ent{\mu_i}-\ent{\nu_i^n}\le
    S_iW_2(\mu_i,\nu_i^n)\quad\text{for every }n\ge\bar n.
  $$
  By Lemma~\ref{le:pre} and \eqref{eq:80} we get, for $n\ge \bar n$
  $$
    \ent{\mu}-\ent{\nu^n}\le
    \alpha_1 S_1W_2(\mu_1,\nu_1^n)+\alpha_2 S_2 W_2(\mu_2,\nu_2^n)
    \le
    \Big(\alpha_1 S_1^2+\alpha_2 S_2^2\Big)^{1/2}W_2(\mu,\nu^n),
  $$
  so that, passing to the limit as $n\to\infty$, our choice of
  $(\nu^n)$ yields that
  $|\rmD^-\entv|(\mu)$ does not exceed $\bigl(\alpha_1 S_1^2+\alpha_2
  S_2^2\bigr)^{1/2}$. Taking the infimum with respect to $S_i$ we conclude.
\end{proof}

\section{The Wasserstein gradient flow of the entropy and
  its identification with the $L^2$ gradient flow of Cheeger's energy}
\label{sec:identification_flows}

\subsection{Gradient flow of $\entv$: the case of bounded densities.}
\label{subs:bounded} 

In the next result we show that any Wasserstein
gradient flow (recall Definition~\ref{def:dissKconv}) of the entropy
functional with uniformly bounded densities coincides with the
$L^2$-gradient flow of the Cheeger's functional. We prove in fact a
slightly stronger result, starting from the energy dissipation
inequality \eqref{eq:edi} instead of the identity \eqref{eq:ede},
where we use the Fisher information functional $\mathsf F$ defined
by Definition~\ref{def:Fisher} instead of the squared slope of
$\entv$. Recall that $\mathsf F(f)\le |\rmD^-\entv|^2(f\mm)$  by
Theorem~\ref{thm:lowerboudslope}.

\begin{theorem}
  \label{thm:coincidence_infty}
  Let $(X,\tau,\sfd,\mm)$ be an Polish extended space satisfying
  \eqref{eq:75} and let $\mu_t=f_t\mm\in D(\entv)$, $t\in [0,T]$, be a curve in
  $\AC2{(0,T)}{\prob X}{W_2}$ satisfying the
   Entropy-Fisher dissipation inequality
  \begin{equation}
    \label{eq:95}
    \ent{\mu_0}\ge \ent{\mu_T}+\frac12\int_0^T|\dot{\mu}_t|^2\,\d t
    +\frac12\int_0^T \mathsf F(f_t)\,\d t.
  \end{equation}
  If $\sup_{t\in [0,T]}\|f_t\|_{L^\infty(X,\mm)}<\infty$
  then $f_t$ coincides in $[0,T]$ with the
  gradient flow $\sfH_t(f_0)$ of Cheeger's energy starting from
  $f_0$.\\
  In particular, for all $f_0\in L^\infty(X,\mm)$ there exists at
  most one Wasserstein gradient flow $\mu_t=f_t\mm$ of
  $\entv$ in $(\pro\mu (X),W_2)$  starting from $\mu_0=f_0\mm$
  with uniformly bounded densities $f_t$.
\end{theorem}
\begin{proof} Let us set $\mu^1_t=\mu_t$, $f^1_t=f_t$ and
  let us first observe that by
Lemma~\ref{le:conditioned_UG} and \eqref{eq:22} the curve $\mu^1_t$
satisfies
\begin{equation}
    \label{eq:95bis}
    \ent{\mu^1_0}= \ent{\mu^1_t}+\frac12\int_0^t|\dot{\mu}^1_s|^2\,\d s
    +\frac 12\int_0^t\mathsf F(f^1_s)\,\d s
    \quad
    \text{for every }t\in [0,T].
  \end{equation}
  Indeed, \eqref{eq:81} and \eqref{eq:22} show that the function defined by the
  right-hand side of \eqref{eq:95bis} is nondecreasing with respect to
  $t$ and coincides with $\ent{\mu^1_0}$ at $t=0$ and at $t=T$ by
  \eqref{eq:95}.

 Let $\mu^2_t=f^2_t\mm$, with $f^2_t:=\sfH_t(f_0)$, be the solution of
  the $L^2$-gradient flow of the Cheeger's energy. Theorem~\ref{prop:maxprin}
  shows that $\|f^2_t\|_{L^\infty(X,\mm)}\le
  \|f_0\|_{L^\infty(X,\mm)}$; by Lemma \ref{le:key} and
  Proposition~\ref{le:diss} we get
  \begin{equation}
    \label{eq:95tris}
    \ent{\mu^2_0}\ge \ent{\mu^2_t}+\frac12\int_0^t|\dot{\mu}^2_s|^2\,\d s
    +\frac 12\int_0^t\mathsf F(f^2_s)\,\d s\quad
    \text{for every }t\in [0,T].
  \end{equation}
  We recall that the squared Wasserstein distance is convex w.r.t.
linear interpolation of measures. Therefore, given two absolutely
continuous curves $(\mu^1_t)$ and $(\mu^2_t)$, the curve $t\mapsto
\mu_t:=(\mu^1_t+\mu^2_t)/2$ is absolutely continuous as well and its
metric speed can be bounded by
\begin{equation}
\label{eq:w2conv}
|\dot{\mu_t}|^2\leq\frac{|\dot{\mu}^1_t|^2+|\dot{\mu}^2_t|^2}{2}\qquad\text{for
a.e. $t\in (0,T)$.}
\end{equation}
Adding up \eqref{eq:95bis} and \eqref{eq:95tris} and using the
convexity of the Fisher information functional (see
Lemma~\ref{le:Fisher}), the convexity of the squared metric speed
guaranteed by \eqref{eq:w2conv} and taking into account the
\emph{strict} convexity of $\entv$ we deduce that for the curve
$\mu_t$ it holds
\[
\ent{\mu_0}>\ent{\mu_t}+\frac12\int_0^t|\dot{\mu_s}|^2\,\d s
+\frac12\int_0^t\mathsf F(f_s)\,\d s
\]
for every $t$ such that $\mu^1_t\neq\mu^2_t$, where $f_t:=\tfrac
12(f^1_t+f^2_t)$ is the density of $\mu_t$. This contradicts
Lemma~\ref{le:conditioned_UG}, which yields the opposite inequality.
\end{proof}
Although the result will not play a role in the paper, let's see
that we can apply the previous theorem to characterize all limits of
the JKO \cite{JordanKinderlehrerOtto98} -- Minimizing Movement
Scheme (see \cite[Definition 2.0.6]{Ambrosio-Gigli-Savare08})
generated by the entropy functional in $\proV (X)$. The result shows
that starting from an initial datum with bounded density, the JKO
scheme \emph{always converges} to the $L^2$-gradient flow of
Cheeger's energy, without any extra assumption on the space, except
for the integrability condition \eqref{eq:75}.

For a given initial datum $\mu_0=f_0\mm\in D(\entv)$ and a time step
$h>0$ we consider the
sequence $\mu^h_n=f^h_n\mm$ defined by the recursive variational problem
\begin{displaymath}
  \mu^h_n\in \argmin_{\mu\in \proV (X)}\Big\{\frac
  1{2h}W^2_2(\mu,\mu^h_{n-1})+\ent{\mu}\Big\},
\end{displaymath}
and we set $\mu^h(t)=f^h(t)\mm:=\mu^h_n$ if $t\in ((n-1)h,nh]$.

\begin{corollary}[Convergence of the minimizing movement scheme]
  \label{cor:conv}
  Let $(X,\tau,\sfd,\mm)$ be an Polish extended space satisfying
  \eqref{eq:75} and let $\mu_0=f_0\mm\in D(\entv)$ with $f_0\in L^\infty(X,\mm)$.
  Then for every $t\ge 0$ the family $\mu^h(t)$ weakly converges to $\mu_t=f_t\mm$
  as $h\downarrow0$, where $f_t=\sfH_t(f_0)$ is the $L^2$-gradient flow of Cheeger's energy.
\end{corollary}
\begin{proof}
  Arguing exactly as in \cite[\S2.1]{Agueh02},
  \cite[Proposition~2]{Otto96} it is not hard to show that
  $\|f^h_n\|_\infty\leq\|f_0\|_\infty$.

  We want to apply the theory developed in
  \cite[Chap.\ 2-3]{Ambrosio-Gigli-Savare08}: according to the
  notation therein $\mathscr S$ is the metric space $\pro {\mu_0}(X)$
  endowed with the Wasserstein distance $W_2$, $\sigma$ is the weak
  topology in $\prob X$,
  and $\phi$ is the Entropy functional $\entv$.
  Since by \eqref{eq:86} and \eqref{eq:60} the negative
  part of $\entv$ has at most quadratic growth in ${\mathscr
    P}_{\mu_0}(X)$, the basic assumptions
  \cite[2.1(a,b,c)]{Ambrosio-Gigli-Savare08} are satisfied and
  we can apply the compactness result \cite[Corollary
  3.3.4]{Ambrosio-Gigli-Savare08}: from any vanishing sequence
  of time steps $h_m\downarrow0$ we can extract a subsequence
  (still denoted by $h_m$) such that $\mu^{h_m}(t)\to \mu_t=f_t\mm$ weakly in $\prob X$, with
  $f^{h_m}(t)\weakto f_t$ weakly in any $L^p(X,\mm)$, $p\in
  [1,\infty)$,
  and $\|f_t\|_\infty\le \|f_0\|_\infty$.
 Since the relaxed slope of the entropy functional, defined as
  \begin{displaymath}
    |\partial^-\entv|(\mu):=\inf\Big\{\liminf_{n\to\infty}|\rmD^-\entv|(\mu_n):\mu_n\weakto
    \mu,\quad
    \sup_n W_2(\mu_n,\mu),\ent{\mu_n}<\infty\Big\}
  \end{displaymath}
  still satisfies
  the lower bound \eqref{eq:44pre} $|\partial^- \entv|(\rho\mm)\ge \mathsf F(\rho)$
  thanks to the lower semicontinuity of the Fisher information with
  respect to the weak $L^1(X,\mm)$-topology,
  the energy inequality \cite[(3.4.1)]{Ambrosio-Gigli-Savare08}
  based on De Giorgi's variational interpolation
  yields
  \begin{displaymath}
    \ent{\mu_0}\ge \ent{\mu_T}+\frac12\int_0^T|\dot{\mu}_t|^2\,\d t
    +\frac12\int_0^T \mathsf F(f_t) \,\d t.
  \end{displaymath}
  Applying the previous theorem we conclude that $f_t=\sfH_t(f_0)$.
  Since the limit is uniquely characterized, all the family $\mu^h(t)$
  converges to $\mu_t$ as $h\downarrow0$.
\end{proof}

\subsection{Uniqueness of the Wasserstein gradient flow if $|\rmD^-
  \entv|$ is an upper gradient}
In the next theorem we prove uniqueness of the gradient flow of
$\entv$, a result that will play a key role in the equivalence
results of the next section. Here we can avoid the uniform
$L^\infty$ bound assumed in Theorem~\ref{thm:coincidence_infty}, but
we need to suppose that $|\rmD^-\entv|$ is an upper gradient for
the entropy functional
 (a condition which is ensured by its geodesically $K$-convexity,
see the next section).

\begin{theorem}[Uniqueness of the gradient flow of $\entv$]\label{thm:gfent}
Let $(X,\tau,\sfd,\mm)$ be a Polish extended space be such that
$|\rmD^-\entv|$ is an upper gradient of $\entv$ and let $\mu\in
D(\entv)$. Then there exists at most one gradient flow of $\entv$
starting from $\mu$ in $(\pro\mu (X),W_2)$.
\end{theorem}
\begin{proof} As in \cite{Gigli10} and in the proof of
Theorem~\ref{thm:coincidence_infty}, assume that starting from some
$\mu\in D(\entv)$ we can find two different gradient flows
$(\mu^1_t)$ and $(\mu^2_t)$. Then we have
\[
\begin{split}
\ent{\mu}&=\ent{\mu^1_T}+\frac12\int_0^T|\dot{\mu}^1_t|^2\,\d t
+\frac12\int_0^T|\rmD^-\entv|^2(\mu^1_t)\,\d t \qquad\forall T\geq 0,\\
\ent{\mu}&=\ent{\mu^2_T}+\frac12\int_0^T|\dot{\mu}^2_t|^2\,\d t
+\frac12\int_0^T|\rmD^-\entv|^2(\mu^2_t)\,\d t \qquad\forall T\geq
0.
\end{split}
\]
Adding up these two equalities and using the convexity of the
squared slope guaranteed by Theorem~\ref{thm:slope_convex}, the
convexity of the squared metric speed guaranteed by
\eqref{eq:w2conv} and taking into account the \emph{strict}
convexity of $\entv$ we deduce that for the curve
$t\mapsto\mu_t:=(\mu^1_t+\mu^2_t)/2$ it holds
\[
\ent{\mu}>\ent{\mu_T}+\frac12\int_0^T|\dot{\mu_t}|^2\,\d t
+\frac12\int_0^T|\rmD^-\entv|^2(\mu_t)\,\d t ,
\]
for every $T$ such that $\mu^1_T\neq\mu^2_T$. Taking the upper
gradient property into account, this contradicts
\eqref{eq:boundtuttecurve}.
\end{proof}

\begin{remark}
  \label{rem:altra_stima}
  \upshape
The proofs of Theorem~\ref{thm:gfent} and
  Theorem~\ref{thm:coincidence_infty} do not rely on contractivity of
  the Wasserstein distance. Actually, as proved by Ohta and Sturm in
  \cite{Ohta-Sturm10}, the property
$$
W_2(\mu_t,\nu_t)\leq e^{Kt}W_2(\mu_0,\nu_0)
$$
for gradient flows of $\entv$ in Minkowski spaces
$(\R^n,\|\cdot\|,\Leb{n})$ whose norm is not induced by an inner
product fails for any $K\in\R$.
\fr
\end{remark}

\subsection{Identification of the two gradient flows}\label{sec:duality}

Here we prove one of the main results of this paper, namely the
identification of the gradient flow of $\C$ in $L^2(X,\mm)$ and the
gradient flow of $\entv$ in $(\prob X,W_2)$. The strategy consists
in considering a gradient flow $(f_t)$ of $\C$ with nonnegative
initial data and in proving that the curve $t\mapsto \mu_t:=f_t\mm$
is a gradient flow of $\entv$ in $(\prob X,W_2)$. All these results
will be applied to the case of metric spaces satisfying a
$CD(K,\infty)$ condition in the next section.

\begin{theorem}[Identification of the two gradient flows]\label{thm:main}
Let $(X,\tau,\sfd,\mm)$ be a Polish extended space such that
\eqref{eq:75} holds and let us assume that $|\rmD^-\entv|$ is
lower semicontinuous with respect to strong convergence with moments
in $\Probabilities{X}$ on sublevels of $\entv$. For all $f_0\in
L^2(X,\mm)$ such that $\mu_0=f_0\mm\in {\mathscr P}_\Wgh(X)$ the
following equivalence holds:
\begin{itemize}
\item[(i)] If $f_t$ is the gradient flow of $\C$ in
$L^2(X,\mm)$ starting from $f_0$, then $\mu_t:=f_t\mm$ is the
gradient flow of $\entv$ in $(\pro{\mu_0}(X),W_2)$ starting from
$\mu_0$, $t\mapsto\entv(\mu_t)$ is locally absolutely continuous in
$(0,\infty)$ and
\begin{equation}\label{eq:allequalities}
-\frac{\d}{\d t}\entv(\mu_t)= |\dot\mu_t|^2=|\rmD^-\entv(\mu_t)|^2
\qquad\text{for a.e. $t\in (0,\infty)$.}
\end{equation}
\item[(ii)] Conversely, if $|\rmD^-\entv|$ is an upper gradient of
$\entv$, and $\mu_t$ is the gradient flow of $\entv$ in
$(\pro{\mu_0}(X),W_2)$ starting from $f_0\mm$, then $\mu_t=f_t\mm$
and $f_t$ is the gradient flow of $\C$ in $L^2(X,\mm)$ starting from
$f_0$.
\end{itemize}
\end{theorem}
\begin{proof} (i) First of all, we remark that assumption
\eqref{eq:62} of Lemma~\ref{le:key} is satisfied, thanks to
Theorem~\ref{thm:infinitemass}; in addition, the same theorem
ensures that $\int_X\Wgh^2f_t^2\,\d\mm<\infty$ for all $t\geq 0$.
Defining $\mu_t:=f_t\mm$, we know by Proposition~\ref{le:diss} that
the map $t\mapsto\entv(\mu_t)$ is locally absolutely continuous in
$(0,\infty)$ and that \eqref{eq:eqdissrate} holds.

On the other hand, since we assumed the lower semicontinuity of
$|\rmD^-\entv|$, we can prove that $\entv(\mu_t)$ satisfies the
energy dissipation inequality \eqref{eq:edi}. Indeed, by
Lemma~\ref{le:key} and Theorem~\ref{thm:slopefisherconv} it holds:
\[
\int_{\{f_t>0\}}\frac{\relgrad{f_t}^2}{f_t}\,\d\mm\geq
\frac12|\dot\mu_t|^2+\frac12|\rmD^-\entv|^2(\mu_t) \qquad\text{for
a.e. $t\in (0,\infty)$.}
\]
This proves that $\entv(\mu_t)$ satisfies the energy dissipation
inequality. But, since we know that $t\mapsto \entv(\mu_t)$ is
locally absolutely continuous we can apply
Remark~\ref{rem:whenslopesare} to obtain that $|\tfrac{\d}{\d
t}\entv(\mu_t)|\leq|\rmD^-\entv|(\mu_t)|\dot\mu_t|$ for a.e. $t\in
(0,\infty)$. Hence, as explained in Section~\ref{se:prelGF},
\eqref{eq:edi} in combination with Young inequality and the previous
inequality yield that all the inequalities turn a.e. into
equalities, so that \eqref{eq:allequalities} holds.

\noindent (ii) We know that a gradient flow $\tilde f_t$ of $\C$
starting from $f_0$ exists, and part (i) gives that
$\tilde\mu_t:=\tilde f_t\mm$ is a gradient flow of $\entv$. By
Theorem~\ref{thm:gfent}, there is at most one gradient flow starting
from $\mu_0$, hence $\mu_t=\tilde\mu_t$ for all $t\geq 0$.
\end{proof}

As a consequence of the identification result, we present a general
existence result of the Wasserstein gradient flow of $\entv$ which
includes also the case of $\sigma$-finite measures and requires no curvature assumption. 
When the initial probability density is not in $L^2(X,\mm)$ we are only able to obtain a
gradient flow in the weaker sense of the maximal energy dissipation inequality \eqref{eq:edi}.
In the case when $|\rmD^- \entv|$ is also an upper gradient for the entropy functional
we can of course recover a gradient flow,
namely the energy dissipation identity \eqref{eq:ede}.

\begin{theorem}[Existence of the gradient flow of $\entv$]\label{thm:gfent2}
Let $(X,\tau,\sfd,\mm)$ be a Polish extended space satisfying
assumption \eqref{eq:75} and such that
$|\rmD^-\entv|$ is lower semicontinuous with respect to strong
convergence with moments in $\Probabilities{X}$ on sublevels of
$\entv$.\\
Then for all $\mu=\rho\mm\in D(\entv)$ there exists a locally absolutely continuous
curve $\mu_t:[0,\infty)\to\pro\mu (X)$ starting from $\mu$ and satisfying \eqref{eq:edi}.
\end{theorem}
\begin{proof} We can take advantage of the identification of
gradient flows and immediately obtain existence, even in the stronger sense
\eqref{eq:ede}, when $\rho\in L^2(X,\mm)$. If only the integrability conditions
$\int_X\rho\log\rho\,\d\mm<\infty$ and
$\int_X\Wgh^2\rho\,\d\mm<\infty$ are available, we can 
set $\rho^n:=\min\{\rho,n\}$ and use the same
monotone approximation argument as in the proof of
Theorem~\ref{thm:lowerboudslope}: keeping that notation, 
we set $\rho_t:=\sup_n\rho^n_t$, $\mu_t=\rho_t\mm$, $\mu^n_t:=z_n^{-1}\rho^n_t\mm$
and we decompose the function $s \log s$ into the sum $h_-(s)+h_+(s)$ where
$h_-(s)=\min(s,\rme^{-1})\log(\min(s,\rme^{-1}))$ and
$h_+(s)=\max(s,\rme^{-1})\log(\max(s,\rme^{-1}))+\rme^{-1}$
are decreasing and increasing functions respectively. Applying the
monotone convergence theorem to $h_{\pm}(\rho^n_t)$ we easily get
$\int_X \rho^n_t\log\rho^n_t\,\d\mm\to\int_X\rho_t\log\rho_t\,\d\mm$, so that
$\ent{\mu^n_t}\to \ent{\mu_t}$ as $n\to\infty$ because $z_n\uparrow 1$. 
%
 %
 We can now pass to the limit in \eqref{eq:edi} written for $\mu^n_t$
 by using the lower semicontinuity of $|\rmD^-\entv|$ and of the
 $2$-energy to obtain that $\mu_t$ still satisfies \eqref{eq:edi}. 
 \end{proof}

 \begin{remark}  
 {\rm  For completeness, we can
provide a proof that does not use the
identification of gradient flows: indeed, we can
 apply the existence result \cite[Prop. 2.2.3, Thm.
2.3.3]{Ambrosio-Gigli-Savare08}, achieved via the so-called
minimizing movements technique, with the topology of weak
convergence in duality with $C_b(X)$. Remark~\ref{re:tightsub},
\eqref{eq:86},
 and the lower semicontinuity part of
Theorem~\ref{thm:slopefisherconv} give that the assumptions are
satisfied, and we get measures $\mu_t$ satisfying
\begin{equation}\label{eq:nougg}
\entv(\rho\mm) \geq \entv(\mu_t)+\int_0^t\frac{1}{2}|\dot\mu_s|^2+
\frac{1}{2}|\rmD^-\entv|^2(\mu_s)\,\d s\qquad\forall t\geq 0.
\end{equation}} \fr
\end{remark}

%
%

\section{Metric measure spaces satisfying
$CD(K,\infty)$}\label{sec:LSV}

In this section we present the applications of the previous theory
in the case when the Polish extended space $(X,\tau,d,\mm)$ has
Ricci curvature bounded from below, according to
\cite{Lott-Villani07} and \cite{Sturm06I}. Under this condition the
Wasserstein slope $|\rmD^-\entv|$ turns out to be a lower
semicontinuous upper gradient of the entropy, so that all the
assumptions of Theorems \ref{thm:gfent}, \ref{thm:main}, and
\ref{thm:gfent2} are satisfied.

\begin{definition}[$CD(K,\infty)$] \label{def:cdkinfty}
We say that $(X,\tau,\sfd,\mm)$ has Ricci curvature bounded from below
by $K\in\R$ if $\entv$ is $K$-convex along geodesics in $(\prob
X,W_2)$. More precisely, this means that for any $\mu_0,\,\mu_1\in
D(\entv)\subset\Probabilities{X}$ with $W_2(\mu_0,\mu_1)<\infty$
there exists a constant speed geodesic
$\mu_t:[0,1]\to\Probabilities{X}$ between $\mu_0$ and $\mu_1$
satisfying
\begin{equation}\label{eq:Kdisplconv}
\entv(\mu_t)\leq (1-t)\entv(\mu_0)+ t\entv(\mu_1)-\frac{K}{2}t(1-t)W_2^2(\mu_0,\mu_1)\qquad \forall
t\in [0,1].
\end{equation}
\end{definition}
Notice that unlike the definitions given in \cite{Lott-Villani07}
and \cite{Sturm06I}, here we are allowing the distance $\sfd$ to
attain the value $+\infty$. Also, even if $\sfd$ were finite, this
definition slightly differs from the standard one, as typically
geodesic convexity is required only in the space $(\probt X,W_2)$,
while here we are assuming it to hold for any couple of probability
measures with finite entropy and distance. Actually, the two are
equivalent, as a simple approximation argument based on the
tightness given by Remark~\ref{re:tightsub} shows.

\begin{remark}[The integrability condition \eqref{eq:75}]
  \label{rem:integrability}
  \upshape {\rm 
  If $(X,\tau,\sfd,\mm)$ satisfies a $CD(K,\infty)$ condition and
  $\tau$ is the topology induced by the finite distance $\sfd$,
  then \eqref{eq:75} is equivalent {\nc (see \cite[Theorem 4.24]{Sturm06I}) to assume the existence of $x\in X$  and $r>0$ such that
  $\mm(B_r(x))<\infty$, and also equivalent to the fact that all bounded sets have finite measure.} In this case one can choose
  $V(x):=A\sfd(x,x_0)$ for a suitable constant $A\ge 0$ and $x_0\in
  X$.} \fr
\end{remark}

\begin{theorem}[Slope, Fisher, and gradient flows]\label{thm:mainCDK}
Let $(X,\tau,\sfd,\mm)$ be a Polish extended space satisfying
$CD(K,\infty)$ and \eqref{eq:75}.
\begin{itemize}
\item[(i)]
  For every $\mu=f\mm\in \proV (X)$ the Wasserstein slope
  $|\rmD^-\entv|^2(\mu)$ coincides with the Fisher information of
  $f$, it is lower semicontinuous under setwise convergence, according to
  \eqref{eq:94}, and it is an upper gradient for $\entv$.
\item[(ii)]
  For every $\mu_0=f_0\mm\in D(\entv)$ there exists a unique gradient flow
  $\mu_t=f_t\mm$ of $\entv$ starting from $\mu_0$ in
  $(\pro \mu(X),W_2)$.
\item[(iii)]
  If moreover $f_0\in L^2(X,\mm)$, the
  gradient flow $f_t=\mathsf H(f_0)$ of $\C$ in $L^2(X,\mm)$ starting from $f_0$
  and the gradient flow $\mu_t$ of $\entv$ in
  $(\pro{\mu_0}(X),W_2)$ starting from $\mu_0$ coincide, i.e.\
  $\mu_t=f_t\mm$ for every $t>0$.
\end{itemize}
\end{theorem}
Thanks to this theorem, under the $CD(K,\infty)$ assumption we can
unambiguously say that a \emph{Heat Flow} on $(X,\tau,\sfd,\mm)$ is
either a gradient flow of Cheeger's energy in $L^2(X,\mm)$ or a
gradient flow of the relative entropy in $(\prob X,W_2)$, at least
for square integrable initial conditions with finite moment.

Concerning the \emph{proof} of Theorem~\ref{thm:mainCDK}, we observe
that applying the results of the previous section it is sufficient
to show that the Wasserstein slope $|\rmD^- \entv|$ is lower
semicontinuous w.r.t.\ strong convergence with moments in $\prob X$
on the sublevel of $\entv$. In fact, if this property holds,
\eqref{eq:48} of Theorem \ref{thm:slopefisherconv} shows that
$|\rmD^-\entv|$ coincides with the Fisher functional  and thus
satisfies the lower semicontinuity property \eqref{eq:94}: in
particular it is lower semicontinuous w.r.t.\ weak convergence with
moments in $\prob X$. Applying Theorem \ref{thm:gfent2} we prove the
existence of the Wasserstein gradient flow starting from $\mu_0$;
its uniqueness follows from Theorem \ref{thm:gfent}, since the slope
is always an upper gradient of $\entv$ under $CD(K,\infty)$.
Applying Theorem \ref{thm:main} we can thus obtain the
identification of the two gradient flows.

In order to prove the lower semicontinuity of the slope
$|\rmD^-\entv|$ w.r.t.\ strong convergence with moments in $\prob
X$ (Proposition~\ref{prop:convlscslope}) we proceed in various
steps, adapting the arguments of \cite{Gigli10}.
\begin{definition}[Plans with bounded deformation]
  Let us set $\tmm:={\GGG 
    \rme^{-a^2\Wgh^2}\mm}$, 
  where
  $V$ satisfies
  \eqref{eq:75} {\GGG and $a>1.$}
  We say that $\ggamma\in\prob{X^2}$ has bounded deformation
  if
\begin{equation}
  \sfd\in L^\infty(X\times X,\ggamma)\quad
  \text{and}\quad
  \text{$c\tmm\leq\pi^i_\sharp\ggamma\leq \frac{1}{c}\tmm$,
  $i=1,\,2$, for some $c>0$.}
  \label{eq:87}
\end{equation}
\end{definition}
{\GGG Notice that if $V$ satisfies \eqref{eq:75} then 
  $aV$, $a>1$, 
  satisfies \eqref{eq:75} as well and $\proV(X)=\mathscr P_{aV}(X)$.

Let $\ggamma$ be a plan with bounded deformation; 
in the next proofs we will use the following simple properties:
\begin{enumerate}[(i)]
  \item If a sequence of probability densities $(\eta_n)$ converges
    weakly to $\eta$ in $L^1(X,\tmm)$ as $n\to\infty$ then
    $\eta_n\circ\pi^1$ converges weakly to $\eta\circ\pi^1$ in
    $L^1(X\times X,\ggamma)$.
  \item
    If a sequence of probability
    densities $(g_n)$  weakly converges to $g$ in $L^1(X\times
    X;\ggamma)$ 
    then 
    $\tilde g_n:=\d(\pi^2_\sharp (g_n\ggamma))/\d\tmm$ weakly converges
    to the corresponding $\tilde g:=\d(\pi^2_\sharp
    (g\ggamma))/\d\tmm$ in $L^1(X,\tmm)$.
  \item
    If 
    $W:X\to[0,\infty]$ is a Borel function, and 
    $(h_n)$ is a sequence of probability densities
    weakly converging to $h$ in $L^1(X,\tmm)$, then
    \begin{equation}
      \label{eq:20}
      \liminf_{n\to\infty}\int_Y Wh_n\,\d\tmm\ge 
      \int_Y Wh\,\d\tmm.
    \end{equation}
  \end{enumerate}
  (i) follows easily by the
  disintegration theorem, since
  denoting by $\rho\in L^\infty(X,\tmm)$ 
  the density of $\gamma^1=\pi^1_\sharp\ggamma$ w.r.t.~$\tmm$ and
  by $(\ggamma_x)_{x\in X}$ the disintegration of $\ggamma$
  w.r.t.~$\gamma^1$, for every $\varphi\in L^\infty(X\times X;\ggamma)$
  we have
  \begin{displaymath}
    \int_X \varphi(x,y)\eta_n(x)\,\d\ggamma(x,y)=
    \int_X \eta_n(x)\Big(\int_X \varphi(x,y)\,\d\ggamma_x(y)\Big)
    \rho(x)\,\d\tmm(x),
  \end{displaymath}
  and $\rho\int_X \varphi(\cdot,y)\,\d\ggamma_\cdot(y)$ belongs to
  $L^\infty(X,\tmm)$. 
  An even easier argument yields (ii): with obvious notation, for any
  $\varphi\in L^\infty(X,\tmm)$ we have
  \begin{displaymath}
    \int_{X}\tilde g_n(y)\varphi(y)\,\d\tmm(y)=
    \int_{X\times X}g_n(x,y)\varphi(y)\,\d\ggamma(x,y).
  \end{displaymath}
  (iii) follows by a standard approximation by truncation,
  since for every $N>0$
  \begin{displaymath}
    \liminf_{n\to\infty}\int_{X}Wg_n\,d\sigma\ge
    \liminf_{n\to\infty}\int_{X}\max\{W,N\}g_n\,d\sigma=
    \int_{X}\max\{W,N\}g\,d\sigma.
  \end{displaymath}
}
\begin{proposition}[Sequential lower semicontinuity of $G_\sggamma$]\label{prop:baseconv}
  For any plan $\ggamma$ with bounded deformation the map
  $\mu\mapsto G_\sggamma(\mu)=\ent{\mu}-\ent{\PushPlan\ggamma\mu}$
  (recall Section~\ref{sec:steepest})
  is sequentially lower semicontinuous with respect to weak convergence with moments,
  on sequences with $\entv$ uniformly bounded from above.
\end{proposition}
\begin{proof} 
 Let $\mu_n=\eta_n\tmm\in {\mathscr P}_\Wgh(X)$ be weakly convergent with
moments to $\mu=\eta\tmm$, with $\ent{\mu_n}$ uniformly bounded. 
{\GGG The formula \eqref{eq:60} for the change of reference measure in the entropy
and Remark~\ref{re:tightsub} show that $\eta_n$ weakly
  converge to $\eta$ in $L^1(X,\tmm)$.} If
$\rho$ denotes the density of $\pi^1_\sharp\ggamma$ w.r.t. $\tmm$,
we have that $(\eta_n/\rho)\circ\pi^1\ggamma$ is an admissible plan
between $\mu_n$ and $\PushPlan{\ggamma}{\mu_n}$; hence $\rho^{-1}\in
L^\infty(X,\tmm)$ and $\sfd\in L^\infty(X\times X,\ggamma)$ ensure
that $\PushPlan{\ggamma}{\mu_n}$ belong to ${\mathscr P}_\Wgh(X)$ as
well. 
{\GGG Combining the previous properties (i) and (ii) 
  (with $g_n:=(\eta_n/\rho)\circ\pi^1$) we can then 
  show that the densities $h_n$ of 
  $\PushPlan\ggamma{\mu_n}$ w.r.t.\ $\tmm$ weakly converge to
  the corresponding density $h$ of $\PushPlan\ggamma\mu$ 
  in $L^1(X,\tmm)$, and \eqref{eq:20} yields 
  \begin{displaymath}
    \liminf_{n\up\infty}\int_X \Wgh^2\,\d(\PushPlan\ggamma{\mu_n})=
    \liminf_{n\up\infty}\int_X \Wgh^2\,h_n\,\d\tmm\ge
    \int_X \Wgh^2\,h\,\d\tmm
    =\int_X \Wgh^2\,\d(\PushPlan\ggamma{\mu}).
  \end{displaymath}
}
From
\eqref{eq:60} we obtain that
$$
\ent{\mu_n}-\ent{\PushPlan\ggamma\mu_n}=\RE{\mu_n}{\tmm}-\RE{\PushPlan\ggamma\mu_n}{\tmm}-
\int_X\Wgh^2\,\d\mu_n+\int_X\Wgh^2\,\d\PushPlan\ggamma\mu_n
$$
{\GGG with $\RE{\mu_n}{\tmm}$} uniformly bounded. So, we are
basically led, after a normalization, to the case of a probability
reference measure $\tmm$. In this case the proof uses the
equiintegrability in $L^1(\tmm)$ of $\eta_n$, ensured by the upper
bound on entropy, see \cite[Proposition~11]{Gigli10} for details.
\end{proof}

\begin{lemma}[Approximation]\label{le:approximation}
If $\mu,\,\nu\in D(\entv)$ satisfy $W_2(\mu,\nu)<\infty$ then there
exist plans $\ggamma_n$ with bounded deformation satisfying
\[
\begin{split}
  \int_{X\times X}\sfd^2\,\d(\ggamma_n)_\mu\to W_2^2(\mu,\nu)
  \quad\text{and}\quad
  \ent{(\ggamma_n)_\sharp\mu}\to\ent{\nu}.
\end{split}
\]
\end{lemma}
\begin{proof} 
{\GGG 
Recalling \eqref{eq:50}, the proof of \cite[Lemma 10]{Gigli10} 
provides a sequence $(\ggamma_n)$ of 
plans with bounded deformation satisfying 
\[
\begin{split}
  \int_{X\times X}\sfd^2\,\d(\ggamma_n)_\mu\to W_2^2(\mu,\nu),
  \quad
  \nu_n:=(\ggamma_n)_\sharp\mu\to \nu\quad\text{in }\prob X,
  \quad
  \RE{(\ggamma_n)_\sharp\mu}{\tmm}\to\RE{\nu}{\tmm}.
\end{split}
\]
We want to show that the convergence of the relative entropy w.r.t.\
$\tmm$ yields the same property for $\entv$. 
By \eqref{eq:60} (with $aV$ instead of $V$)
this is equivalent to prove the convergence of the moments.

Denoting by $h_n $ the density of 
$\nu_n$ w.r.t.\ $\tmm$, 
\eqref{eq:60} and Remark \ref{re:tightsub}
show that $h_n$ weakly converge to the corresponding density
$h$ of $\nu$ in $L^1(X,\tmm)$, so that \eqref{eq:20} yields
\begin{equation}
  \label{eq:84}
  \liminf_{n\to\infty}\int_X V^2\,\d\nu_n=
  \liminf_{n\to\infty}\int_X V^2h_n\,\d\tmm=
  \int_X V^2h\,\d\tmm=
  \int_X V^2\,\d\nu.
\end{equation}
On the other hand,
if $\bar\mm:=\rme^{-V^2}\,\d\mm$, since
$\tmm=\rme^{-(a^2-1)V^2}\,\d\bar\mm$, \eqref{eq:60} 
and the lower semicontinuity of $\RE\cdot{\bar\mm}$ in $\prob X$ yield
\begin{align*}
  \limsup_{n\to\infty}\,(a^2-1)\int_X
  V^2h_n\,\d\tmm&=
  \limsup_{n\to\infty}\Big(\RE{\nu_n}\tmm-\RE{\nu_n}{\bar\mm}\Big)
  =
  \RE{\nu}\tmm-
  \liminf_{n\to\infty}\RE{\nu_n}{\bar\mm}
  \\&\le
  \RE{\nu}{\bar\mm}-\RE{\nu}{\bar\mm}=
  (a^2-1)\int_X V^2 h\,\d\tmm.
\end{align*}
Since $a^2-1>0$, combining with \eqref{eq:84} we conclude.}
\end{proof}
%
%
%

%
%
\begin{proposition}[$|\rmD^-\entv|$ is a l.s.c. slope in $CD(K,\infty)$ spaces]
\label{prop:convlscslope} Assume that $(X,\tau,d,\mm)$ is a Polish
extended space satisfying $CD(K,\infty)$ and \eqref{eq:75} holds.
Then $D(\entv)\ni\mu\mapsto |\rmD^-\entv|^2(\mu)$ is sequentially
lower semicontinuous w.r.t. weak convergence with moments on the
sublevels of $\entv$. In particular \eqref{eq:48} holds.
\end{proposition}
\begin{proof}
In this proof we denote by $C(\ggamma)$ the cost of $\ggamma$, i.e.
$C(\ggamma):=\int \sfd^2\,\d\ggamma$. We closely follow \cite[Theorem 12
and Corollary 13]{Gigli10}.

Let $\mu=\rho\mm$ in the domain of the entropy. Taking
\eqref{eq:slopesup} and the $K$-geodesic convexity of $\entv$ into account, we first prove that it holds
\begin{equation}
\label{eq:representation} |\rmD^-\entv|(\mu) =\sup_{\sggamma}
\frac{
  \bigl(G_\sggamma(\mu)
  -\frac{K^-}2C(\ggamma_\mu)\bigr)^+}{\sqrt{C(\ggamma_\mu)}},
\end{equation}
where the supremum runs in the class of plans $\ggamma$ with bounded
deformation. Indeed, inequality $\geq$ follows
choosing $\nu=\PushPlan\ggamma\mu$ in \eqref{eq:slopesup} and using
the trivial inequality
\[
a\in\R,\,\,c\geq b>0\qquad\Longrightarrow\qquad
\frac{(a-b)^+}{\sqrt{b}}\geq\frac{(a-c)^+}{\sqrt{c}},
 \]
 with $a=G_\sggamma(\mu)$, $b=W_2^2(\mu,\nu)$ and
 $c=C(\ggamma_\mu)$,
together with the fact that $C(\ggamma_\mu)\geq W_2^2(\mu,\nu)$. The
other inequality is a consequence of the approximation
Lemma~\ref{le:approximation}.

To conclude, it is sufficient to prove that for all $\ggamma$ with
bounded deformation the map
$\mu\mapsto\bigl(G_\sggamma(\mu) -\tfrac{K^-}2C(\ggamma_\mu)\bigr)^+
/C(\ggamma_\mu)^{1/2}$ is sequentially lower semicontinuous with
respect to weak convergence with moments on the sublevels of the
entropy. This follow by Proposition~\ref{prop:baseconv} and the fact
that $\mu\mapsto C(\ggamma_\mu)$ is continuous along these
sequences. In turn, the continuity property along these sequences
follows by the representation
$$
C(\ggamma_\mu)=\int_{X\times X}\frac{\d\mu}{\d\tmm}(x)
\bigl(\frac{\d\pi^1_\sharp\ggamma}{\d\tmm}(x)\bigr)^{-1}\sfd^2(x,y)\,\d\ggamma(x,y).
$$
Indeed, both $\bigl(\d\pi^1_\sharp\ggamma/\d\tmm\bigr)^{-1}$ and
$\sfd$ are essentially bounded, while the densities $\d\mu/\d\tmm$
are equiintegrable, as we saw in the proof of
Proposition~\ref{prop:baseconv}.
\end{proof}

\section{A metric Brenier theorem and gradients of Kantorovich
potentials}\label{sec:weakbrenier}

In this section we provide a ``metric'' version of Brenier's theorem
and we identify ascending slope and minimal weak upper gradient of
Kantorovich potentials. These results depend on $L^\infty$ upper
bound on interpolations, a property that holds in spaces with
Riemannian lower bounds on Ricci curvature, see
\cite{Ambrosio-Gigli-Savare11bis}, or in non-branching
$CD(K,\infty)$ metric spaces (because the non-branching property is
inherited by $(\Probabilities{X}, W_2)$, see
\cite[Corollary~7.32]{Villani09},
\cite[Proposition~2.16]{Ambrosio-Gigli11}, and all $p$-entropies are
convex). {\nc See also \cite{R2011b} for more recent results in this direction,
independent of the non-branching assumption.}

 If $\sfd$ is bounded, the $L^\infty$ bound can be relaxed
to an easier bound on entropy, but modifying the class of test
plans, see Remark~\ref{rem:alternoption}. We assume throughout this
section that
\begin{displaymath}
  \text{$\sfd$ is a finite distance and $\mm$ satisfies \eqref{eq:75}}.
\end{displaymath}
However, we keep the
possibility of considering the case when $\tau$ is not induced by
$\sfd$.

In this section we denote by $\calT$ the class of test plans
concentrated on $\AC 2{[0,1]}X\sfd$ with bounded compression on the
sublevels of $\Wgh$ and by $\mathcal G\subset \calT$ the subclass of
test plans concentrated on $\geo(X)$. By
Remark~\ref{rem:monotone-weakT} we have the obvious relation
\begin{equation}\label{eq:nizza}
\weakgradG{f}\leq\weakgradA{f}.
\end{equation}

 In the next lemma we prove that, for Kantorovich potentials
$\varphi$, $t\mapsto\varphi(\gamma_t)$ is not only Sobolev but also
absolutely continuous along $\calT$-almost every curve in $\AC
{2}{[0,1]}X\sfd$. This holds even though in our general framework no
Lipschitz continuity property (not even a local one) of $\varphi$
can be hoped for; in particular, by Remark~\ref{rem:whenslopesare}
we obtain that $|\rmD^+\varphi|$ is a $\calT$-weak upper gradient
of $\varphi$. 
\begin{lemma} [Slope is a weak upper gradient for Kantorovich
potentials] \label{le:slopekanto} Let
$\mu=\rho\mm\in\Probabilities{X},\,\nu\in\Probabilities{X}$ with
$W_2(\mu,\nu)<\infty$ and let $\varphi:X\to\R\cup\{-\infty\}$ be a
Kantorovich potential relative to some optimal plan $\ggamma$
 between $\mu$ and $\nu$.
If
$\rho$ satisfies
\begin{equation}\label{eq:wbre22}
\rho\geq c_M>0\quad\text{$\mm$-a.e. in $\{\Wgh\leq M\}$, for all
$M\geq 0$}
\end{equation}
then $\varphi$ is absolutely continuous along $\calT$-almost every
curve of $\AC 2{[0,1]}X\sfd$ and the slope $|\rmD^+\varphi|$ is a
$\calT$-weak upper gradient of $\varphi$.
\end{lemma}
\begin{proof} Set $f=-\varphi^c$, so that $\varphi=Q_1f$
(here we adopt the notation of \S \ref{sec:hopflax})
and the set $\mathcal
D(f)$ in \eqref{eq:defdf} coincides with $X$. By
Proposition~\ref{prop:slopeKA} we know that the function
$$
D^*(x):=\int_X \sfd(x,y)\,\d\ggamma_x(y)
$$
(where $\{\ggamma_x\}_{x\in X}$ is the disintegration of $\ggamma$
w.r.t.~$\mu$)
belongs to $L^2(X,\mu)$ and bounds $\mu$-a.e.~from above $\rmD^-(x,1)$
by \eqref{eq:100},
and then $\mm$-a.e.; we know also from \eqref{eq:hjups} that
$|\rmD^+\varphi|\in L^2(X,\mu)$ and that
$\rmD^-(x,1)\geq|\rmD^+\varphi|(x)$ wherever $\varphi(x)>-\infty$. We
modify $D^*$ in a $\mm$-negligible set, getting a function
$\tilde{D}\in L^2(X,\mu)$ larger than $\rmD^-(x,1)$ everywhere and
equal to $+\infty$ on the $\mm$-negligible set
$\{\varphi=-\infty\}$.

We claim now that the condition $\int_\gamma\tilde{D}<\infty$ is
fulfilled for $\mathcal T$-almost every $\gamma$ in $\AC
2{[0,1]}X\sfd$. Indeed, arguing as in \eqref{eq:21}, for any test
plan $\ppi\in \calT$ with $\Energy 2\gamma\leq N^2<\infty$
$\ppi$-a.e. we have
$$
\int\int_{\gamma\cap\{\Wgh\leq M\}} \tilde{D}\,\d\ppi\leq N
\Big(C(\ppi,M) \int_{\{\Wgh\leq M\}}\tilde{D}^2\,\d\mm\Big)^{1/2}
\leq N \Big(c_M^{-1} C(\ppi,M)\int_X
\tilde{D}^2\,\d\mu\Big)^{1/2}<\infty,
$$
thanks to the fact that $\rho\geq c_M$ on $\{\Wgh\leq M\}$. Since
$M$ is arbitrary and since $\ppi$-a.e. curve $\gamma$ is contained
in $\{\Wgh\leq M\}$ for sufficiently large $M$, the claim follows.

Now, let $\gamma\in \AC2{[0,1]}X\sfd$ with
$\int_\gamma\tilde{D}<\infty$, and $\lambda=|\dot\gamma|\Leb
1\restr{[0,1]}$. $\varphi$ is $\sfd$-upper semicontinuous and
$\varphi\circ\gamma$ is finite $\lambda$-a.e. (since $\tilde D\circ
\gamma$ is finite $\lambda$-a.e.). By \eqref{eq:9} with
$x=\gamma_{s}$ and $y=\gamma_{t}$, taking also the inequality
$\rmD^-(x,1)\leq\tilde{D}(x)$ into account, we get
$$
\varphi(\gamma_s)-\varphi(\gamma_t)\leq
\sfd(\gamma_{s},\gamma_{t})\Big(\tilde D(\gamma_{t})+ \frac
{\sfd(\gamma_{s},\gamma_{t})}2\Big)\le \left|\int_s^t
|\dot\gamma_r|\,\d r\right|\Big(\tilde
D(\gamma_{t})+\diam(\gamma)\Big)
$$
for all $t$ such that $\varphi(\gamma_t)>-\infty$. Hence we can
apply Corollary~\ref{cor:real_curves} to conclude that
$\varphi\circ\gamma$ is absolutely continuous in $[0,1]$. Recalling
Remark~\ref{rem:whenslopesare} we get
$$
\biggl|\int_{\partial\gamma}\varphi\biggr|\leq
\int_\gamma|\rmD^+\varphi|.
$$
\end{proof}
Using \eqref{eq:nizza}, the previous lemma and
Proposition~\ref{prop:slopeKA}, we have the chain of inequalities
\begin{equation}\label{nizza1}
\weakgradG\varphi (x)\leq\weakgradA\varphi(x)\leq
|\rmD^+\varphi|(x)\leq \sfd(x,y)\qquad\text{$\ggamma$-a.e. in
$X\times X$}
\end{equation}
for any optimal plan $\ggamma$. In the next theorem we show that an
$L^\infty$ bound on geodesic interpolation ensures that the
inequalities are actually equalities.

{In order to present the notion of geodesic plan we assume for a moment that $(X,\sfd)$ is
a geodesic space.} In such spaces, the optimal transport problem can
be ``lifted'' to $\geo(X)$ considering all $\ppi\in\Probabilities{\geo(X)}$ (called geodesic transport plans) 
whose marginals at time $0$ and at time $1$ are respectively $\mu$ and $\nu$ and minimizing
$$
\int \int_0^1|\dot\gamma_s|^2\,\d s\,\d\ppi(\gamma)=\int
\sfd^2(\gamma_0,\gamma_1)\,\d\ppi(\gamma)
$$
in this class. Since
$(e_0,e_1)_\sharp\ppi$ is an admissible plan between $\mu$ and
$\nu$, it turns out that the infimum is larger than
$W_2^2(\mu,\nu)$. But
a simple measurable geodesic selection
argument provides equivalence of the problems and existence of
optimal $\ppi$.

This motivates the next definition.
\begin{definition}[Optimal geodesic plans]\label{def:optgeo}
Let $\mu,\,\nu\in\prob X$ be such that $W_2(\mu,\nu)<\infty$. A plan
$\ppi\in\Probabilities{\geo(X)}$ is an optimal geodesic plan between
$\mu$ and $\nu$ if
$$
(\rme_0)_\sharp\ppi=\mu,\quad (\rme_1)_\sharp\ppi=\nu,\quad \int
\sfd^2(\gamma_0,\gamma_1)\,\d\ppi(\gamma)=\int
\int_0^1|\dot\gamma_s|^2\,\d s\,\d\ppi(\gamma)=W_2^2(\mu,\nu).
$$
\end{definition}

It is easy to check that
\begin{equation}\label{eq:geoinduced}
t\quad\mapsto\quad(\e_t)_\sharp\ppi,
\end{equation}
is a constant speed geodesic in $\prob X$ from $\mu$ to $\nu$ for
all optimal geodesic plans between $\mu$ and $\nu$. In particular, $(\prob X,W_2)$ is
geodesic as well. Also, $(e_0,e_1)_\sharp\ppi$ is an optimal coupling
whenever $\ppi$ is an optimal geodesic plan.

Adapting the arguments in \cite[Theorem~7.21,
Corollary~7.22]{Villani09} for the locally compact case and
\cite{Lisini07,Ambrosio-Gigli11} for the complete case, it can be
shown that in any geodesic Polish extended space $(X,\tau,\sfd)$
\eqref{eq:geoinduced} provides a description of \emph{all} constant
speed geodesics, see \cite{Lisini11}. {\nc In the next theorem
we don't assume really that $(X,\sfd)$ is geodesic, but rather the existence
of an optimal geodesic plan according to Definition~\ref{def:optgeo}.}

\begin{theorem}[A metric Brenier's theorem]\label{thm:brweak}
Let $\mu=\rho\mm\in\Probabilities{X}$ be satisfying
\eqref{eq:wbre22}, let $\nu\in\Probabilities{X}$ with
$W_2(\mu,\nu)<\infty$, let $\ppi$ be an optimal geodesic plan
between $\mu$ and $\nu$ and let $\varphi:X\to\R\cup\{-\infty\}$ be a
Kantorovich potential relative to $(\e_0,\e_1)_\sharp\ppi$. Assume
that $(\e_s)_\sharp\ppi=\mu_s=\rho_s\mm$ for all $s>0$ sufficiently
small and that
\begin{equation}\label{eq:veryweakco}
  \limsup_{s\downarrow 0}\Vert\rho_s\Vert_{L^\infty(\{V\le M\},\mm)}<\infty
  \qquad\forall M>0.
\end{equation}
Then
\begin{equation}\label{eq:wbre2}
\sfd(\gamma_1,\gamma_0)=|\rmD^+\varphi|(\gamma_0)=\weakgradA\varphi(\gamma_0)=
\weakgradG\varphi(\gamma_0)\qquad\text{for $\ppi$-a.e.
$\gamma\in\geo(X)$}.
\end{equation}
As a consequence, $W_2^2(\mu,\nu)=\int_X|\rmD^+\varphi|^2\,\d\mu$
and $|\rmD^+\varphi|=\weakgradA\varphi=\weakgradG\varphi$
 $\mm$-a.e. in $X$.
\end{theorem}
\begin{proof} Set $g:=\weakgradG\varphi$, which belongs to $L^2(X,\mu)$ by \eqref{nizza1}, and $L={\rm Lip}(\Wgh)$.
Still taking \eqref{nizza1}
into account, \eqref{eq:wbre2} can be achieved if we show that $\int
\sfd^2(\gamma_1,\gamma_0)\,\d\ppi\leq\int g^2(\gamma_0)\,\d\ppi$.
Setting $f=-\varphi^c$ so that $\varphi=Q_1 f$, for $\ppi$-a.e.
$\gamma\in\geo(X)$ we have
\begin{align}
 \varphi(\gamma_0)-\varphi(\gamma_t)&\geq
 \Big(f(\gamma_1)+\frac{\sfd^2(\gamma_0,\gamma_1)}2\Big)-\Big(f(\gamma_1)+\frac{\sfd^2(\gamma_t,\gamma_1)}2\Big)
 \nonumber \\&=
 \frac {1-(1-t)^2}2 \sfd^2(\gamma_0,\gamma_1)=\left(\frac{2t-t^2}2\right) \sfd^2(\gamma_0,\gamma_1).
 \label{eq:simple_ineq}
\end{align}
Since the speed of $\gamma$ is $\sfd(\gamma_0,\gamma_1)$ we have
\begin{displaymath}
 \Big(\varphi(\gamma_0)-\varphi(\gamma_t)\Big)^2\le
 \Big(\int_0^t \weakgradG\varphi(\gamma_s) \sfd(\gamma_1,\gamma_0)\,\d s\Big)^2\le
 t \sfd^2(\gamma_1,\gamma_0)\int_0^t g^2(\gamma_s)\,\d s.
\end{displaymath}
Set now $Z_M:=\left\{\gamma\in\geo(X):\ \Wgh(\gamma_0)\leq M,\,\,
\sfd(\gamma_0,\gamma_1)\leq M\right\}$ and notice that
the curves in $Z_M$ are contained in $\{\Wgh\leq M+\delta\}$
for all {\nc $\delta>LMt$}.
Dividing by $t^2
\sfd^2(\gamma_1,\gamma_0)= \sfd^2(\gamma_t,\gamma_0)$ and
integrating on $Z_M$ with respect to $\ppi$, {\nc we can 
use the fact that $\chi_{\{\Wgh\leq M+\delta\}}\ppi$ when rescaled
on a sufficiently small interval $[0,t]$ is a $\mathcal G$-test plan to
obtain}
$$
 \frac1t\int_0^t\int_{Z_M}g^2(\gamma_s)\,\d\ppi(\gamma) \d s\ge
 \int_{Z_M}
 \Big(\frac{\varphi(\gamma_0)-\varphi(\gamma_t)}{\sfd(\gamma_0,\gamma_t)}\Big)^2
 \,\d\ppi\ge\frac{(2-t)^2}{4}\int_{Z_M}\sfd^2\,\d\ppi
$$
for $t$ sufficiently small. Setting $\mu_s=(\e_s)_\sharp\ppi$ we get
\begin{equation}
 \label{eq:uppeboucong}
 \frac1t\int_0^t\int_{\{\Wgh\leq M+\delta\}}g^2\,\d\mu_s
 \d s\ge
 \int_{Z_M}
 \Big(\frac{\varphi(\gamma_0)-\varphi(\gamma_t)}{\sfd(\gamma_0,\gamma_t)}\Big)^2
 \,\d\ppi\ge\frac{(2-t)^2}{4}\int_{Z_M}\sfd^2\,\d\ppi.
\end{equation}
In order to pass to the limit as $t\downarrow 0$, we observe that
\eqref{eq:veryweakco} gives
\begin{equation}\label{eq:gtest}
  \forall\, N>0:\quad
  \int_{\{V\le N\}} f\,\d \mu_s
  \to\int_{\{V\le N\}}  f\,\d\mu\quad\text{as $s\downarrow 0$ for all
    $f\nchi_{\{V\le N\}}\in
   L^1(X,\mm)$.}
\end{equation}
Indeed for every bounded, Borel, and $\sfd$-Lipschitz function
$h:X\to\R$ we have
\begin{equation}
  \label{eq:106}
  \left|\int_X h\,\d \mu_s -\int_X h\,\d\mu\right|\le
  \int |h(\gamma_s)-h(\gamma_0)|\,\d\ppi(\gamma)\le s\,{\rm Lip}(h)\,W_2(\mu,\nu).
\end{equation}
On the other hand, arguing exactly as in the proof of
Proposition~\ref{prop:densitydlip}, if $f\nchi_{\{V\le N\}}\in
L^1(X,\mm)$ we can find a sequence $(h_n)\subset L^1(X,\mm)$ of
bounded, Borel, $\sfd$-Lipschitz functions strongly converging to
$f\nchi_{\{V\le N\}}$ in $L^1(X,\mm)$. Upon multiplying $h_n$ by the
$\sfd$-Lipschitz function $k_N(x):=\min\{1,(N+1-V(x))^+\}$, it is
not restrictive to assume that $h_n$ identically vanishes on $\{V>
N+1\}$. If $\|\rho\|_{L^\infty(\{V\le N+1\},\mm)}\le C$ and
$\|\rho_s\|_{L^\infty(\{V\le N+1\},\mm)}\le C$ for
sufficiently small $s$ according to \eqref{eq:veryweakco}, we thus
have
\begin{displaymath}
  \bigg|\int_{\{V\le N\}} f\,\d\mu_s-\int_{\{V\le N\}} f\,\d\mu\bigg|
  \le 2C \|f\nchi_{\{V\le N\}} -h_n\|_{L^1(X,\mm)}+  \left|\int_X h_n\,\d \mu_s -\int_X h_n\,\d\mu\right|.
\end{displaymath}
Taking first the $\limsup$ as $s\down0$ thanks to \eqref{eq:106} and
then the limit as $n\to\infty$ we obtain \eqref{eq:gtest}.

By \eqref{eq:wbre22} the functions $g^2\nchi_{\{\Wgh\leq
M+\delta\}}$ belong to $L^1(X,\mm)$. Therefore, using
\eqref{eq:gtest} with $f:=g$ and $N:=M+\delta$, passing to the limit
in \eqref{eq:uppeboucong} first as $t\downarrow 0$ and then as
$\delta\downarrow 0$ gives
\begin{equation}\label{eq:11}
 \int_{\{\Wgh\leq M\}}g^2\,\d\mu\geq
 \limsup_{t\downarrow 0}
 \int_{Z_M}\Big(\frac{\varphi(\gamma_0)-\varphi(\gamma_t)}{\sfd(\gamma_0,\gamma_t)}\Big)^2
 \,\d\ppi(\gamma)\geq\int_{Z_M} \sfd^2(\gamma_1,\gamma_0)\,\d\ppi(\gamma).
\end{equation}
Letting $M\to\infty$ this completes the proof of \eqref{eq:wbre2}.
\end{proof}

The identification \eqref{eq:wbre2} could be compared to Theorem~6.1
of \cite{Cheeger00}, where Cheeger identified the relaxed gradient
of Lipschitz functions with the local Lipschitz constant, assuming
that the metric measure space $(X,\sfd,\mm)$ is doubling and
satisfies the Poincar\'e inequality. Without doubling conditions,
but assuming the validity of good interpolation properties, we are
able to obtain an analogous identification at least in a suitable
class of $c$-concave functions.

For finite reference measures $\mm$ and densities $\rho$ uniformly
bounded from below, we can also prove a more precise convergence
result for the difference quotients of $\varphi$.

\begin{theorem}\label{thm:brweak1}
Let $\mu=\rho\mm\in\Probabilities{X}$ be satisfying $\rho\geq c>0$
$\mm$-a.e. in $X$ and let $\varphi$, $\ppi$ as in
Theorem~\ref{thm:brweak}. Then
\begin{equation}\label{eq:goodslopepi}
\lim_{t\downarrow
0}\frac{\varphi(\gamma_0)-\varphi(\gamma_t)}{\sfd(\gamma_0,\gamma_t)}=|\rmD^+\varphi|(\gamma_0)
\qquad\text{in $L^2(\geo(X),\ppi)$.}
\end{equation}
\end{theorem}
\begin{proof} The lower bound on $\rho$ yields in this case
$|\rmD^+\varphi|\in L^2(X,\mm)$, hence one can argue as in the proof
of Theorem~\ref{thm:brweak}, this time integrating on the whole of
$\geo(X)$, to get
$$
 \int_X\weakgradG\varphi^2\,\d\mu\geq
 \limsup_{t\downarrow 0}
 \int_{\geo(X)}\Big(\frac{\varphi(\gamma_0)-\varphi(\gamma_t)}{\sfd(\gamma_0,\gamma_t)}\Big)^2
 \,\d\ppi(\gamma)\geq\int_{\geo(X)} \sfd^2(\gamma_1,\gamma_0)\,\d\ppi(\gamma)
$$
in place of \eqref{eq:11}. Since \eqref{eq:wbre2} yields that all
inequalities are equalities, and \eqref{eq:simple_ineq} yields
$$
\liminf_{t\downarrow
0}\frac{\varphi(\gamma_0)-\varphi(\gamma_t)}{\sfd(\gamma_0,\gamma_t)}\geq
|\rmD^+\varphi|(\gamma_0)\qquad\text{for $\ppi$-a.e.
$\gamma\in\geo(X)$}
$$
we can use Lemma~\ref{le:verysimple} below to obtain
\eqref{eq:goodslopepi}.
\end{proof}

\begin{lemma}\label{le:verysimple}
Let $\sigma$ be a positive, finite measure in a measurable space
$(Z,{\mathcal F})$ and let $f_n,\,f\in L^2(Z,{\mathcal F},\sigma)$
be satisfying
\begin{equation}\label{eq:nocancell}
\limsup_{n\to\infty}\int_Z f_n^2\,\d\sigma\leq\int_Z
f^2\,\d\sigma<\infty
\end{equation}
 and $\liminf_nf_n\geq f\geq 0$ $\sigma$-a.e. in
$Z$. Then $f_n\to f$ in $L^2(Z,{\mathcal F},\sigma)$.
\end{lemma}
\begin{proof} If $f_n\geq 0$, it suffices to expand the square $(f_n-f)^2$ and to
apply Fatou's lemma. In the general case we obtain first the
convergence of $f_n^+$ to $f$ in $L^2$, and then use
\eqref{eq:nocancell} once more to obtain that $f_n^-\to 0$ in $L^2$.
\end{proof}

 Example~\ref{ex:strict} shows that
 the localization technique provided by the potential $\Wgh$ and
 \eqref{eq:75} plays an important role:
indeed, in the same situation of that example, let
$\mm=\delta_0+x^{-1}\Leb{1}$ be a $\sigma$-finite measure in
$X=[0,1]$, so that ${\d\mu_t}/{\d\mm}(x)\leq 1$ for any $t,\,x$. In
this case the conclusions of the metric Brenier theorem are not
valid, since $\mu_0$ is concentrated at $0$ and $d(0,y)$ takes all
values in $[0,1]$. Notice that $\mm$ is not locally finite and the
class of continuous $\mm$-integrable functions is not dense in
$L^1([0,1];\mm)$ (any continuous and integrable function must vanish
at $x=0$).

\begin{remark}{\rm
We remark that in the generality we are working with, it is not
possible to prove uniqueness of the optimal plan, and the fact that
it is induced by a map, not even if we add a $CD(K,\infty)$
assumption. To see why, consider the following example. Let $X=\R^2$
with the $L^\infty$ distance and the Lebesgue measure. Let
$\mu_0:=\nchi_{[0,1]^2}\Leb{2}$ and
$\mu_1:=\nchi_{[3,4]\times[0,1]}\Leb{2}$. Then, using standard tools
of optimal transport theory, one can see that the only information
that one can get by analyzing the $c$-superdifferential of an
optimal Kantorovich potential is, shortly said, that any vertical
line $\{t\}\times[0,1]$ must be sent onto the vertical line
$\{t+3\}\times[0,1]$. The constraint on the marginals gives that
this transport of $\{t\}\times[0,1]$ on  $\{t+3\}\times[0,1]$ must
send the 1-dimensional Hausdorff measure on $\{t\}\times[0,1]$ in
the 1-dimensional Hausdorff measure on $\{t+3\}\times[0,1]$ for a.e.
$t$. Apart from this, there is no other constraint, so we see that
there are quite many optimal plans and that most of them are not
induced by a map. Yet, the metric Brenier theorem is true, as the
distance each point travels is independent of the optimal plan
chosen (and equal to 3 for $\mu_0$-a.e. $x$).}\fr
\end{remark}

\begin{remark}\label{rem:alternoption} {\rm
Theorem~\ref{thm:brweak} and Theorem~\ref{thm:brweak1}, with the
same proof, hold if we replace condition \eqref{eq:veryweakco} with
the weaker one (at least in finite measure spaces)
\[
\limsup_{s\downarrow 0}\int_X\rho_s\log\rho_s\,\d\mm<\infty,
\]
but adding the condition $|\rmD^+\varphi|\in L^\infty(\{\Wgh\leq
M\},\mm)$ for all $M\geq 0$. This, however, requires a slight
modification of the class of test plans, and consequently of the
concept of minimal weak upper gradient, requiring that the marginals
have only bounded entropy instead of bounded density. This approach,
that we do not pursue here, might be particularly appropriate when
$\sfd$ is a bounded distance (e.g. in compact metric spaces),
because in this situation Kantorovich potentials are Lipschitz.}\fr
\end{remark}

\def\cprime{$'$}

\end{document}